\documentclass[ejsv2,preprint,noshowframe]{imsart}

\RequirePackage[authoryear]{natbib}
\RequirePackage[colorlinks,citecolor=blue,urlcolor=blue,linkcolor=blue]{hyperref}
\RequirePackage{graphicx}
\RequirePackage{amsmath}
\makeatletter
\@ifpackageloaded{newtxmath}{}{\RequirePackage{amsfonts}\RequirePackage{amssymb}}
\makeatother
\RequirePackage{booktabs}
\RequirePackage{alltt}
\RequirePackage{xcolor}

\startlocaldefs
\setpkgattr{figurename}{size}{\upshape\scshape}
\setpkgattr{tablename}{size}{\upshape\scshape}

\def\stitle@fmt#1{\noindent\textbf{#1\ }}
\def\supplement@name{Supplementary material}
\theoremstyle{plain}
\newtheorem{Theorem}{Theorem}
\newtheorem{Lemma}{Lemma}
\newtheorem{Conjecture}{Conjecture}
\endlocaldefs

\begin{document}
\nonfrenchspacing  %
\begin{frontmatter}
\title{On the median of the beta distribution in relation to the median of the gamma distribution}
\runtitle{The beta median and the gamma median}

\begin{aug}
\author[A]{\fnms{Richard F.}~\snm{Lyon}\ead[label=e1]{dicklyon@acm.org}}
\address[A]{Google Research, Mountain View, California, USA\printead[presep={,\ }]{e1}}
\runauthor{R. F. Lyon}
\end{aug}

\begin{abstract}
The median of the beta distribution has no closed form, but it is closely tied to the median of the
gamma distribution.  Kerman's approximation has the ratio form $g(\alpha)/(g(\alpha)+g(\beta))$ with
$g(x) = x - \frac13$, a lower bound on the gamma median, and it is a lower bound on the beta median for
$\frac13 < \alpha < \beta$.  The ratio form built from the true gamma median is a tighter lower bound,
asymptotically exact as $\beta \to \infty$.  The ratio form built from the Wilson--Hilferty approximation to
the gamma median, implicit in a classical approximation to the $F$ distribution but little known as an
approximation to the beta median, is an upper bound on the line $\alpha = 1$ and, along each ray $\beta = r\alpha$ with $r > 1$, for large
shape parameters.  The ratio form that is exact on the lines $\alpha = 1$ and $\beta = 1$,
$(2^{1/\beta} - 1)/(2^{1/\alpha} + 2^{1/\beta} - 2)$, is numerically within $0.3\%$ of the median when both shape
parameters are at least 1, and provably less than $25\%$ above it for $\alpha \le \beta$; it yields a new
approximation to the gamma median, which is an upper bound on it for shape parameter below 1 and a lower
bound above 1.  Interpolating between it and a ratio form with a $1/x$ correction gives, with no fitted constant,
an approximation within $0.016\%$ of the median when both shape parameters are at least 1.  No ratio form gives the beta median exactly.  For large shape parameters, the logit of the median
has an explicit second-order expansion.  For small shape parameters, bounds in elementary functions are asymptotically tight.  An approximation in
elementary functions that joins the two ends, with four fitted constants, is exact on the lines $\alpha = 1$,
$\beta = 1$, and $\alpha = \beta$, and numerically within $0.065\%$ of the median over the whole domain.
Most of the proofs are machine-verified in Lean~4.
\end{abstract}



\end{frontmatter}

\section{Introduction}\label{sec-intro}

The beta distribution has a long history in probability and statistics,
originating in the inverse probability problem studied by Thomas Bayes in his
1764 essay and later generalized by Pierre-Simon Laplace \citep{dale1999history}.
Today, it is widely used as a flexible model for random variables limited to
finite intervals.  Importantly, it serves as the standard conjugate prior for the
Bernoulli, binomial, and geometric distributions in Bayesian updating,
just as the gamma distribution is used as
conjugate prior for Poisson, exponential, gamma, and
normal distributions \citep{kerman2011neutral,orloff2022conjugate}.
In each case, the median of an updated posterior distribution of the same family
is often taken as a point estimator of a parameter, as it minimizes the expected
absolute error, and is thus more robust than the mean, which minimizes
the expected squared error \citep{minsker2014scalable,hartwig2020median}.
So the medians of these distributions are useful, even if not simple to calculate.

The beta distribution's probability density function (PDF), with shape
parameters $\alpha > 0$ and $\beta > 0$, for $0 < x < 1$ is:
\begin{equation}
p_{\textrm{B}}(x) = \frac{1}{B(\alpha, \beta)} x^{\alpha-1}(1 - x)^{\beta-1} \, ,
\end{equation}
where $B(\alpha, \beta)$ is the beta function.

The beta distribution's mean is well known to be
$\mu_{\textrm{B}}(\alpha, \beta) = \alpha/(\alpha + \beta)$.  The median
$\nu_{\textrm{B}}(\alpha, \beta)$, on the other hand, is implicitly defined as the value
of $x$ where the cumulative distribution function (CDF) equals one-half:
\begin{equation}
\frac{1}{2} = \int_{0}^{\nu_{\textrm{B}}(\alpha,\beta)} p_{\textrm{B}}(x)\,dx.
\end{equation}
The median is never very far from the mean, but there is no explicit formula for it.

While we can find the median using standard numerical methods, effectively
evaluating the inverse incomplete beta function, these approaches are
computational black boxes, offering no analytical insight into how the median
behaves as a function of its shape parameters.  Closed-form approximations and
analytical bounds provide a useful alternative view:  they
facilitate theoretical analysis and in some cases explicitly preserve
the symmetry
$\nu_{\textrm{B}}(\alpha, \beta) = 1 - \nu_{\textrm{B}}(\beta, \alpha)$.

A very simple and well-known closed-form approximation to the median of the
beta distribution, good for high enough values of the shape parameters,
is due to \citet{kerman2011closed}.  An older approximation of the same ratio form, built on the
Wilson--Hilferty cube-root transformation, is implicit in \citeauthor{paulson1942approximate}'s
(\citeyear{paulson1942approximate}) normal approximation to the $F$ distribution, and approximations to
percentage points of Fisher's $z$ distribution, such as \citeauthor{carter1947approximation}'s
(\citeyear{carter1947approximation}), also give the beta median, though they are designed for the tails.
At the other end of the range, for small
values of the shape parameters, asymptotic analysis by \citet{askitis2017asymptotic}
provides good approximations, but not in closed form, as these involve gamma functions.
Other than these, and classical inequalities relating the mean, median, and mode
\citep{groeneveld1977mode,payton1989bounds}, the literature seems to be largely
silent on the beta median.
The literature on the median of the gamma distribution, on the other hand,
is relatively richer, suggesting an ongoing interest in problems of this sort.

The need for tight analytical bounds and approximations is especially clear in the
extreme parameter regime where $\alpha$ and $\beta$ are less than 1.  While much of
the statistical literature focuses on the unimodal case ($\alpha, \beta > 1$),
U-shaped and J-shaped beta distributions are standard in several applied fields.
In population genetics, for example, the Wright--Fisher model relies on beta
approximations to describe the frequencies of rare alleles undergoing genetic
drift \citep{tataru2015inference}.  Similarly, in objective Bayesian inference,
updating a sparse, noninformative prior, such as the Beta$(\tfrac{1}{2}, \tfrac{1}{2})$ Jeffreys
prior, with highly skewed early data yields asymmetric posterior distributions
where the shape parameters may remain below 1 \citep{gelman2013bayesian}.
Because the probability density
diverges to infinity at the boundaries in these sub-unity cases, numerical
extraction of the median can become unstable, making closed-form approximations
desirable.

In this paper, we study the behavior of the beta distribution median by
examining its asymptotes, some of its special cases, and its relationship to the
gamma distribution median.  We first generalize and improve upon Kerman's
approximation, showing ways to formulate approximations that perform
well at high $\beta$ and near $\beta = 1$.  One of them,
$\nu_{\textrm{B}4}(\alpha,\beta) = (2^{1/\beta} - 1)/(2^{1/\alpha} + 2^{1/\beta} - 2)$, exact on the lines $\alpha = 1$,
$\beta = 1$, and $\alpha = \beta$, is simple enough to use anywhere:  it is within $0.3\%$ of the median when both
shape parameters are at least 1, and it never exceeds the median by $25\%$.  Unlike Kerman's and the
Wilson--Hilferty forms, which go negative for $\alpha$ below $\tfrac{1}{3}$ and $\tfrac{1}{9}$, it stays between 0 and 1
for all positive shape parameters, and it falls far below the median only where both shape parameters are small
and the median is tiny.  We then look at the low end of the
shape spectrum, deriving tight bounding approximations for the U-shaped and
J-shaped cases where $\alpha \ll 1$ and $\alpha \ll \beta$.  Finally, we combine results from
both ends into a single elementary-function approximation that is within $0.065\%$ of the median over the
whole domain.

This paper is one of a series.  It grew out of earlier work on the median of the gamma distribution
\citep{lyon2021closed,lyon2023tight}, which supplies several of the approaches used here for the more
complicated beta distribution:  asymptotes at the low and high ends of the shape parameter, closed-form
bounds that are tight at one end or the other, and approximations that
interpolate between a lower and an upper bound.  A recent short note \citep{lyon2026kerman} develops one improvement on
Kerman's approximation in detail:  the family of ratio forms with $g(x) = x - \tfrac{1}{3} + c/x$, its
asymptotic bounds for large shape parameters, and the member $\nu_{\textrm{B}3}$ with $c = \tfrac{1}{36}$.
A companion paper \citep{lyon2026gamma} develops approximations to the
gamma median from the one found in Section~\ref{sec:gamma}, and compares them with earlier forms.

\subsection{Relation to gamma distribution}

A fundamental generative property of the beta distribution is its direct
construction from a pair of independent gamma-distributed random variables.
Specifically, if $X \sim \text{Gamma}(\alpha,1)$ and
$Y \sim \text{Gamma}(\beta,1)$ are independent standard gamma random
variables, their ratio $\frac{X}{X+Y}$ exactly follows a beta distribution
with shape parameters $\alpha$ and $\beta$.

The gamma distribution's PDF, with shape parameter $\alpha > 0$ and
rate parameter $\beta > 0$, for $x > 0$ is:
\[
p_{\Gamma}(x) = \frac{1}{\Gamma(\alpha)\beta^{-\alpha} } x^{\alpha - 1} e^{-\beta x} \,.
\]
We mostly default its rate parameter to $\beta = 1$, for the ``standard'' gamma distribution,
and drop it from the notation, since its effect is a simple scaling of the distribution and of the median:
\[
p_{\Gamma}(x) = \frac{1}{\Gamma(\alpha)} x^{\alpha - 1} e^{-x} \,.
\]

The normalizing denominators in the PDFs are the beta function and the gamma function,
defined as the values that make the integral of the PDF equal to 1:
\[
\Gamma(\alpha) = \int_0^\infty x^{\alpha - 1} e^{-x} dx \,,
\]
\[
B(\alpha,\beta) = \int_0^1 x^{\alpha - 1} (1-x)^{\beta - 1} dx
= \frac{\Gamma(\alpha)\Gamma(\beta)}{\Gamma(\alpha+\beta)} \,.
\]

The mean $\mu_{\Gamma}(\alpha)$ of the standard gamma distribution is well known
to be $\alpha$, which is easy to verify since the first moment of the PDF
evaluates to $\Gamma(\alpha+1) / \Gamma(\alpha) = \alpha$.
The median $\nu_{\Gamma}(\alpha)$, on the other hand, implicitly defined as the
value of $x$ where the integral (the CDF) equals one-half, has no explicit formula:
\[
\frac{1}{2} = \int_0^{\nu_{\Gamma}(\alpha)} p_{\Gamma}(x) dx =
\frac{1}{\Gamma(\alpha)} \int_0^{\nu_{\Gamma}(\alpha)} x^{\alpha - 1} e^{-x} dx \,.
\]

For the scaled (as opposed to standard) gamma distribution, $\beta \ne 1$, the factor
$e^{-\beta x}$ takes the place of the beta distribution's $(1 - x)^{\beta-1}$.
These are asymptotically equal for high $\beta$ and low $x$, which is why the
beta distribution approaches the shape of a gamma distribution at high $\beta$.

See Table~\ref{tab:medians} for a list of beta and gamma medians and approximations
discussed in this paper.

\begin{table}[!htp]
\centering
 \caption{{\bf List of medians and approximations considered.}
 Gamma and beta distribution medians and approximations discussed in the paper are tabulated here for reference; $\Gamma()$ is the gamma function.}
  \small
  \begin{tabular}{@{}lll@{}}
    \rule{0pt}{12pt}Name & {Formula} & Description  \\  [2 pt]
    \hline

    \rule{0pt}{12pt}$\nu_{\Gamma}(\alpha)$ & --- & True median of gamma distribution \\  [2 pt]
    $\nu_{\Gamma 1}(\alpha)$ & $ \alpha - \frac{1}{3}$ & \begin{tabular}[t]{@{}l@{}}\citet{chen1986bounds} and\\ \citet{berg2006chen} lower bound\end{tabular} \\ [2 pt]
    $\nu_{\Gamma 2}(\alpha)$ & $ \alpha \left( 1 - \frac{1}{9\alpha}\right)^3 $ & \citet{wilson1931distribution} approx. \\ [4 pt]
    $\nu_{\Gamma 3}(\alpha)$ & $ \frac{(\alpha - \frac{1}{6})^2}{\alpha} = \alpha - \frac{1}{3} + \frac{1}{36\alpha}$ & $\alpha - \frac{1}{3} + \frac{c}{\alpha}$ with $c = \frac{1}{36}$; \citet{lyon2026kerman} \\ [4 pt]
    $\nu_{\Gamma 4}(\alpha)$ & $ \frac{\log 2}{2^{1/\alpha} - 1} $ & Approx.~from beta median at $\beta = 1$ \\ [4 pt]
    $\nu_{\Gamma 5}(\alpha)$ & $ \lim_{\beta\to\infty} \beta\,\nu_{\textrm{B}10}(\alpha,\beta)$ & Two-piece approx., see Section~\ref{sec:gamma} \\ [4 pt]

    \hline

    \rule{0pt}{12pt}$\nu_{\textrm{B}}(\alpha,\beta)$ & --- & True median of beta distribution \\  [2 pt]

    $\nu_{\textrm{B}0}(\alpha,\beta)$ & $ \frac{\nu_{\Gamma}(\alpha)}{\nu_{\Gamma}(\alpha)+\nu_{\Gamma}(\beta)} $ & Based on true gamma median \\  [4 pt]

    $\nu_{\textrm{B}1}(\alpha,\beta)$ & $ \frac{\nu_{\Gamma 1}(\alpha)}{\nu_{\Gamma 1}(\alpha)+\nu_{\Gamma 1}(\beta)} = \frac{\alpha - \frac{1}{3}}{\alpha + \beta - \frac{2}{3}}$ & \citet{kerman2011closed} approximation \\  [4 pt]

    $\nu_{\textrm{B}2}(\alpha,\beta)$ & $ \frac{\nu_{\Gamma 2}(\alpha)}{\nu_{\Gamma 2}(\alpha)+\nu_{\Gamma 2}(\beta)} = \frac{\alpha \left( 1 - \frac{1}{9\alpha }\right)^3}{\alpha \left( 1 - \frac{1}{9\alpha }\right)^3 + \beta \left( 1 - \frac{1}{9\beta }\right)^3}$ & \citet{paulson1942approximate}, via Wilson--Hilferty \\  [4 pt]

    $\nu_{\textrm{B}3}(\alpha,\beta)$ & $ \frac{\nu_{\Gamma 3}(\alpha)}{\nu_{\Gamma 3}(\alpha)+\nu_{\Gamma 3}(\beta)} = \frac{(\alpha - \frac{1}{6})^2/\alpha}{(\alpha - \frac{1}{6})^2/\alpha + (\beta - \frac{1}{6})^2/\beta}$ & Based on $\nu_{\Gamma 3}$; \citet{lyon2026kerman} \\  [4 pt]

    $\nu_{\textrm{B}4}(\alpha,\beta)$ & $ \frac{\nu_{\Gamma 4}(\alpha)}{\nu_{\Gamma 4}(\alpha)+\nu_{\Gamma 4}(\beta)} = \frac{2^{1/\beta} - 1}{2^{1/\alpha} + 2^{1/\beta} - 2} $ & Based on $\nu_{\Gamma 4}$ or Lemma~\ref{lem:special} \\  [4 pt]

    $\nu_\textrm{E}(\alpha,\beta)$ &
    $\frac{g_c(\alpha)}{g_c(\alpha)+g_c(\beta)}$, $g_c(x) = x - \frac{1}{3} + \frac{c_\textrm{E}(\alpha,\beta)}{x}$ &
    Edge-exact; see Section~\ref{sec:asymexp} \\ [4 pt]
    $\nu_{1/30}(\alpha,\beta)$ &
    $\frac{\alpha - \frac{1}{3} + \frac{1}{30\alpha}}{\alpha + \beta - \frac{2}{3} + \frac{1}{30\alpha} + \frac{1}{30\beta}}$ &
    Upper bound, $1 \le \alpha \le \beta$ \\ [2 pt]
    $\nu_\textrm{I}(\alpha,\beta)$ &
    $\frac{\alpha - 1}{\alpha}\,\nu_{1/30}(\alpha,\beta) + \frac{1}{\alpha}\,\nu_{\textrm{B}4}(\alpha,\beta)$ &
    Interpolated, $1 \le \alpha \le \beta$; see Section~\ref{sec:interp} \\ [4 pt]

    $\nu_{\textrm{B}5}(\alpha,\beta)$  & $ 2^{-1/\alpha} \left(\frac{\Gamma(\alpha+1) \Gamma(\beta)}{\Gamma(\alpha+\beta)} \right)^{1/\alpha}$ & \citet{askitis2017asymptotic} low-$\alpha$ asymptote \\ [4 pt]
    & $\quad = \left(2 \frac{\beta}{\alpha+\beta}\right)^{-1/\alpha} \left(\frac{\Gamma(\alpha+1) \Gamma(\beta+1)}{\Gamma(\alpha+\beta+1)} \right)^{1/\alpha}$ & factored as $P(\alpha,\beta)\,F(\alpha,\beta)$ \\ [4 pt]

    $\nu_{\textrm{B}6}(\alpha,\beta)$ &
    $\left(2\frac{\beta}{\alpha+\beta}\right)^{-1/\alpha} \underline{c}_0(\beta)\,e^{\underline{k}(\beta)\alpha - \eta(\beta)\alpha^2}$ &
    Elementary lower bound for $\beta > 1$ \\ [4 pt]
    $\nu_{\textrm{B}7}(\alpha,\beta)$ &
    $\left(2\frac{\beta}{\alpha+\beta}\right)^{-1/\alpha} \overline{c}_0(\beta)\,e^{\overline{k}(\beta)\alpha}$ &
    Elementary upper bound for $\beta < 1$ \\ [2 pt]
    & \multicolumn{2}{l}{\quad with $\underline{c}_0$, $\overline{c}_0$, $\underline{k}$, $\overline{k}$, and $\eta$ as in Section~\ref{sec:lowrefactor}} \\ [4 pt]
        $\nu_{\textrm{B}8}(\alpha,\beta)$ &
    $v / (1 + v)^{e_\textrm{D}(\beta)\,(1 - 0.22\,(1-\beta)(1-\alpha/\beta))}$ &
    Upper bound, $\alpha \le \beta < 1$ \\ [2 pt]
    $\nu_{\textrm{B}9}(\alpha,\beta)$ &
    $v / (1 + v)^{e_1(\beta)\,(1 + 0.2\,(1-\alpha)\beta/(\beta+2))}$ &
    Lower bound, $\alpha \le 1 < \beta$ \\ [2 pt]
    & \multicolumn{2}{l}{\quad with $v = P\,G$, $G$, $e_\textrm{D}$, and $e_1$ as in Section~\ref{sec:odds}} \\ [4 pt]
        $\nu_{\textrm{B}10}(\alpha,\beta)$ &
    Steered odds-corrected form for $\alpha < 1$, &
    Combined approximation, $\alpha \le \beta$; \\
    & $\nu_\textrm{I}$ for $\alpha \ge 1$ & see Section~\ref{sec:whole} \\ [2 pt]

    \hline
  \end{tabular}
  \label{tab:medians}
\end{table}

\subsection{Kerman's approximation}

Figure~\ref{fig:linear} shows the true beta distribution median (computed numerically, as described
in Section~\ref{sec:methods}) along with Kerman's approximation \citep{kerman2011closed}
$\nu_{\textrm{B}1}(\alpha,\beta)$ and one simple better approximation
$\nu_{\textrm{B}2}(\alpha,\beta)$ discussed later, for seven $\beta$ values
(1, 2, 5, 20, and their reciprocals) that we use repeatedly in sections to follow.

Note the symmetries obeyed by the beta distribution, and therefore by its mean and
median:  $\mu_\textrm{B}(\alpha,\beta) = 1 - \mu_\textrm{B}(\beta,\alpha)$, and
$\nu_\textrm{B}(\alpha,\beta) = 1 - \nu_\textrm{B}(\beta,\alpha)$, implying mean
and median both equal to one-half when $\alpha = \beta$.  So without loss of
generality we focus on the cases with $\alpha \le \beta$, but it would also be
nice to find functional forms for bounds and approximations that embody the
symmetry and are exact at $\alpha = \beta$, as with Kerman's median approximation.
For $\tfrac{1}{3} < \alpha < \beta$, Kerman's approximation takes values in $(0, \tfrac{1}{2})$
and is a lower bound on the median (Theorem~\ref{thm:kerman}).

\begin{figure}[!htb]
\centerline
{\includegraphics[width=0.8\linewidth]{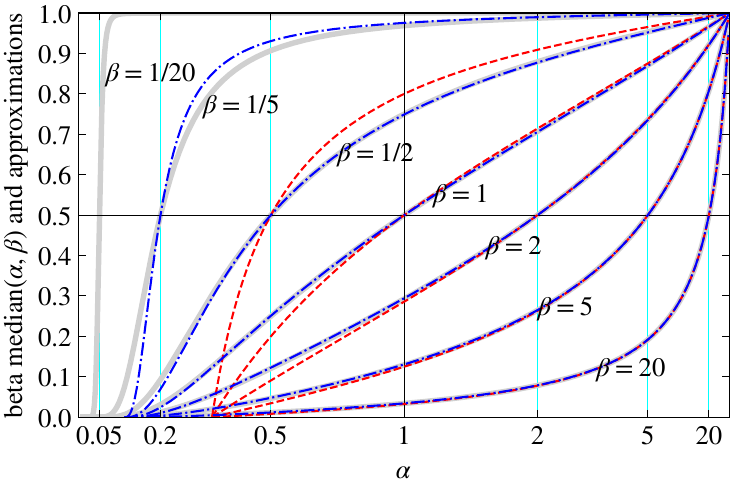}}
    \caption{{\bf Beta distribution medians and approximations.}
    The true median (wide light-gray curves) and approximations (dashed and dash-dot curves) are plotted
    on an arctan($\alpha$) axis for seven values of $\beta$ (1, 2, 5, 20, and their
    reciprocals).  Kerman's approximation $\nu_{\textrm{B}1}$ (red dashed), based on a well-known
    lower bound on the median of a gamma distribution, is defined only for
    $\alpha, \beta > \frac{1}{3}$.  The other approximation, $\nu_{\textrm{B}2}$ (blue dash-dot),
    valid for $\alpha, \beta > \frac{1}{9}$, is based on a formula similar to
    Kerman's, but using the Wilson--Hilferty approximation to the median of the
    gamma distribution instead of the Berg and Pedersen lower bound; it is implicit in an
    approximation to the $F$ distribution by \citet{paulson1942approximate}.
    The curves all cross 0.5
    where $\alpha = \beta$.  In most subsequent figures we only plot where
    $\alpha \le \beta$, corresponding to curves in the lower half of this figure.}
  \label{fig:linear}
\end{figure}

\subsection{Example}

Figure~\ref{fig:example} shows an example beta distribution in which the PDF shows two extreme modes due to $\alpha$ and $\beta$ both being less than 1.  Though Kerman's approximation is defined here, with $\alpha =  0.4$ and $\beta = 0.8$ both being above $\tfrac{1}{3}$, it is not usefully accurate in this region, as the figure shows.  The figure also shows two other approximations, $\nu_{\textrm{B}0}$ of Section~\ref{sec:ratioforms} and $\nu_{\textrm{B}5}$ of Section~\ref{sec:lowend}, which are better in this case.

\begin{figure}[!htb]
  \centerline
  {\includegraphics[width=0.8\linewidth]{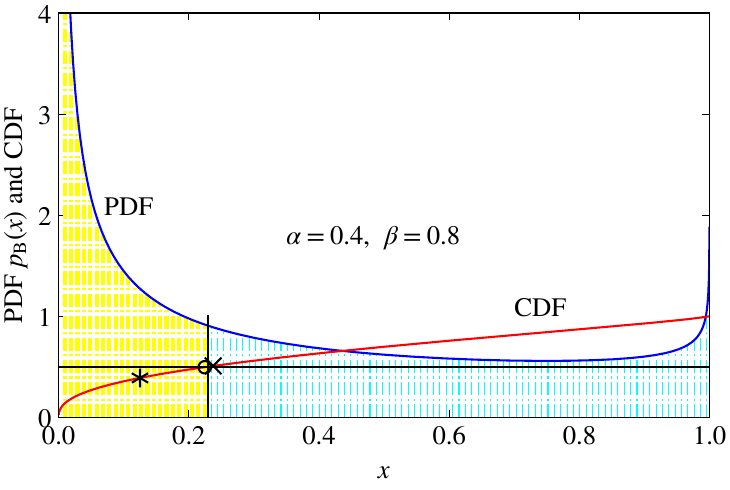}}
    \caption{{\bf Example bimodal beta distribution with \texorpdfstring{$\alpha = 0.4$}{alpha = 0.4} and \texorpdfstring{$\beta = 0.8$}{beta = 0.8}.}
    The beta distribution PDF (blue curve) has a mode at 0 since $\alpha < 1$, and a mode at 1 since $\beta < 1$.  The CDF (red curve) is annotated to show the median (crossed black lines); the vertical line divides the area under the PDF into two equal parts (stippled yellow and cyan).  Several median approximations are marked along the CDF\@.  Kerman's approximation $\nu_{\textrm{B}1}$ (asterisk) is not very good in this low-$\alpha$ region.  Using a ratio of actual medians of the gamma distribution yields an approximate beta median, $\nu_{\textrm{B}0}$ (circle), that is better at medium and high values of $\alpha$ and $\beta$---only about $2\%$ low here.  And the low-$\alpha$ asymptote $\nu_{\textrm{B}5}$ (cross) is about 4\% high here.}
  \label{fig:example}
\end{figure}

\section{Methods}
\label{sec:methods}

We study the behavior of the beta distribution median and of \citeauthor{kerman2011closed}'s (\citeyear{kerman2011closed}) popular approximation to it, which is based on \citeauthor{berg2006chen}'s (\citeyear{berg2006chen}) approximation to the gamma distribution median, and discuss ways of generalizing and improving to better approximations of a similar form.

We further investigate the median more analytically in several regions, in hopes of finding relations that will lead to good bounds and approximations.  We first examine the high-$\beta$ region, then the neighborhood of $\beta = 1$, then the low-$\alpha$ region, and finally combine the results into an approximation over the whole domain.  In each region, we examine and exploit asymptotes and limits, and relationships to the gamma distribution median.

The true medians were originally computed with the double-precision \texttt{betaincinv}
function and other standard functions of Matlab.  The figures and the numerical claims were then
recomputed in a Python port of those scripts, with the corresponding SciPy functions
\citep{virtanen2020scipy}.  Double-precision medians can lose accuracy for small shape parameters, and
underflow altogether when $\alpha$ is small, so the error ranges quoted for the low-$\alpha$ bounds and the
combined approximation (Sections~\ref{sec:odds} and~\ref{sec:whole}) were checked again on the same grids
with medians computed to 30 digits with the Python library mpmath \citep{mpmath}.  Where the
double-precision median is representable, it agrees with the 30-digit one to within $10^{-8}$ in relative
terms, and the 30-digit check gives the same error ranges, including at the grid points where the median
underflows in double precision.  A denser grid near the diagonal for $0.001 \le \beta \le 0.1$ finds errors of
$\nu_{\textrm{B}10}$ between $-0.064\%$ and $+0.065\%$, $3{,}000$ random off-grid points in each low-$\alpha$
region confirm the bounds and the error range, and a scan near the diagonal for $\beta$ from $10^{-20}$ to
$10^{-4}$ follows the error below the grids.  The numerical evidence for the conjectures in
the appendix was computed with 30- to 40-digit medians in the same way.  The Python notebook and the mpmath
script are included as supplementary material.

AI tools were used in this research to check, simplify, and find proofs, to write the Lean formalization,
and to port the code; their use is described in the Statement on AI use, just before the references.

\section{Approach from the gamma distribution median}\label{sec:ratioforms}

\begin{figure}[!htb]
  \centerline
  {\includegraphics[width=0.8\linewidth]{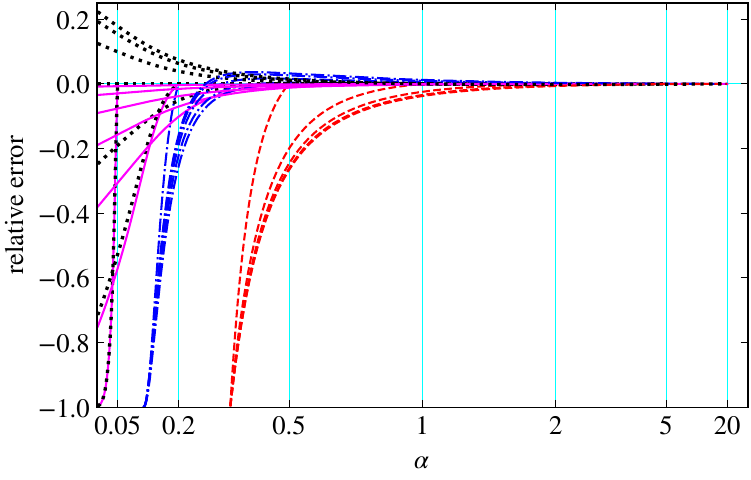}}
    \caption{{\bf Relative errors of median approximations for \texorpdfstring{$\alpha \le \beta$}{alpha <= beta}.}
    Relative errors of Kerman's approximation, $\nu_{\textrm{B}1}(\alpha,\beta)$ (red dashed); approximation using the actual gamma median, $\nu_{\textrm{B}0}(\alpha,\beta)$ (magenta solid); approximation based on the Wilson--Hilferty approximation, $\nu_{\textrm{B}2}(\alpha,\beta)$ (blue dash-dot); and one based on the exact beta median solution where $\alpha$ or $\beta$ is equal to $1$, $\nu_{\textrm{B}4}(\alpha,\beta)$ (black dotted).  Each curve is for a different one of the seven $\beta$ values used.  The $\beta$ for a particular curve is the $\alpha$ coordinate where the curve meets the zero-error line (the high-$\beta$ curves are better resolved in the zoomed-in view in the next figure).}
  \label{fig:rel3}
\end{figure}

\begin{figure}[!htb]
  \centerline
  {\includegraphics[width=0.8\linewidth]{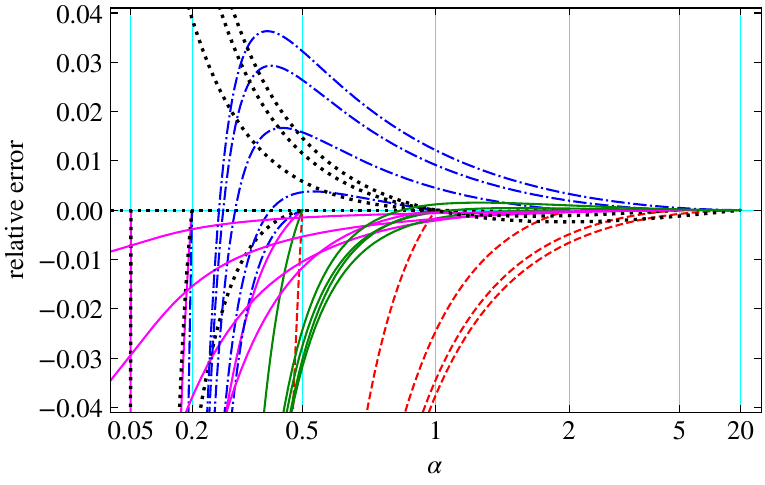}}
    \caption{{\bf Relative errors zoomed in.}
    Same as Figure~\ref{fig:rel3}, but zoomed to show where the approximations are within 4\% of the true median, and with a mostly tighter $\nu_{\textrm{B}3}(\alpha,\beta)$ added (green solid).  No version is uniformly best. }
  \label{fig:rel4}
\end{figure}

In terms of the lower bound and high-$\alpha$ asymptote to the median of the gamma distribution \citep{chen1986bounds,berg2006chen}, $\nu_\Gamma(\alpha) \gtrapprox \nu_{\Gamma 1}(\alpha) = \alpha - \frac{1}{3}$, Kerman's approximation can be written as
\[
\nu_{\textrm{B}1}(\alpha,\beta) = \frac{\nu_{\Gamma 1}(\alpha)} {\nu_{\Gamma 1}(\alpha) + \nu_{\Gamma 1}(\beta) }
\]
as \citet{kerman2011closed} points out.  So it makes sense to ask whether using the gamma median itself in such an expression would give an exact result (it does not) or a useful approximation, as:
\[
\nu_{\textrm{B}0}(\alpha,\beta) = \frac{\nu_{\Gamma}(\alpha)} {\nu_{\Gamma}(\alpha) + \nu_{\Gamma}(\beta)}\,.
\]

Since there is no actual formula for $\nu_{\Gamma}$, if it looks promising we'll want to try other bounds and approximations to the gamma median in such a formula.  We also consider whether there is any function $g$ of one variable such that
\[
\nu_\textrm{B}(\alpha,\beta) = \frac{g(\alpha)} {g(\alpha) + g(\beta)} \,.
\]

Functions $g$ that provide good bounds or approximations of that form would be interesting to find, whether they are from good gamma median approximations or not.

\subsection{High-\texorpdfstring{$\beta$}{beta} asymptote}

For large $\beta$, the shape of the beta distribution approaches the shape of a scaled gamma distribution of the same parameters, so their medians are both asymptotic to $\nu_{\Gamma}(\alpha)/\beta$, which matches the asymptote of our approximation $\nu_{\textrm{B}0}(\alpha,\beta)$.
This limit is proved as Theorem~\ref{thm:asymptote}.
That is, the approximation $\nu_{\textrm{B}}(\alpha,\beta) \approx \nu_{\textrm{B}0}(\alpha,\beta)$ is asymptotically tight for large $\beta$:
\[
\nu_{\textrm{B}0}(\alpha,\beta) = \frac{\nu_{\Gamma}(\alpha)} {\nu_{\Gamma}(\alpha) + \nu_{\Gamma}(\beta)} \quad \implies \quad
\lim_{\beta \to \infty} \frac{\nu_{\textrm{B}0}(\alpha,\beta)}{\nu_{\textrm{B}}(\alpha,\beta)} = 1 \quad \textrm{for all} \quad \alpha > 0 \, .
\]

The asymptotic analysis does not tell us whether the approximation approaches the true value from above or from below, but from the plots it appears that the approximation $\nu_{\textrm{B}0}(\alpha,\beta)$ so defined is a lower bound on the true median.  We formalize and prove this as Theorem~\ref{thm:lower_bound}.

\subsection{Wilson--Hilferty approximation}

Another simple approximation to the median of the gamma distribution comes from the Wilson--Hilferty approximation \citep{wilson1931distribution,abdel1954approximate,ouimet2022refined}, based on noticing that the cube root of a chi-square- or gamma-distributed random variable is very nearly normally distributed, so its median approximately equals its mean, which transforms back via cubing to:
\[
\nu_\Gamma(\alpha) \approx \nu_{\Gamma 2}(\alpha) = \alpha \left( 1 - \frac{1}{9\alpha} \right)^3 \quad \textrm{for} \quad \alpha > \frac{1}{9} \, .
\]
This gamma median approximation leads to a beta median approximation in a ratio form similar to Kerman's:
\[
\nu_{\textrm{B}2}(\alpha,\beta) = \frac{\nu_{\Gamma 2}(\alpha)}{\nu_{\Gamma 2}(\alpha) + \nu_{\Gamma 2}(\beta)}
= \frac{\alpha \left( 1 - \frac{1}{9\alpha } \right)^3}{\alpha \left( 1 - \frac{1}{9\alpha } \right)^3 + \beta \left( 1 - \frac{1}{9\beta } \right)^3}
\quad \textrm{for} \quad \alpha, \beta > \frac{1}{9} \,.
\]
This approximation is not new, though it seems to be little known as an approximation to the beta median.
It is implicit in the normal approximation to the $F$ distribution of \citet{paulson1942approximate},
which applies the cube-root transformation of \citet{wilson1931distribution} to the numerator and
denominator of $F$.  Setting Paulson's normal deviate to zero gives the median of $F$ with $n_1$ and $n_2$
degrees of freedom as $\big((1 - \frac{2}{9n_1})/(1 - \frac{2}{9n_2})\big)^3$; since
$\alpha F/(\alpha F + \beta)$ has the $\textrm{Beta}(\alpha,\beta)$ distribution when $n_1 = 2\alpha$ and
$n_2 = 2\beta$, this is $\nu_{\textrm{B}2}$.  Approximations to percentage points of Fisher's
$z = \frac{1}{2}\log F$ also give the beta median, but they are designed for the tails.  At the median,
\citeauthor{carter1947approximation}'s (\citeyear{carter1947approximation}) refinement of the
Fisher--Cochran formula is singular at $\alpha$ or $\beta = \frac{1}{2}$, and over $1 \le \alpha \le \beta$
its worst relative error is $1/\log 2 - 1 \approx +44\%$, at $\alpha = 1$ as $\beta \to \infty$.

$\nu_{\textrm{B}2}$ is the approximation shown as blue dash--dot in Figure~\ref{fig:linear}.  Figures~\ref{fig:rel3} and~\ref{fig:rel4} compare its relative error to that of other ratio-form approximations.  It is clear that the true gamma median is still mostly better in that form, though not perfect.

This approximation goes negative for $\alpha < \tfrac{1}{9}$, so it is less than the true median for small enough
values of $\alpha$, and it is greater than the true median at $\alpha = 1$ (for $\alpha < \beta$
as usual; Theorem~\ref{thm:wh_edge}).  Numerically it crosses from high to low only once, always
with $\alpha < 1$, and we conjecture that it is an upper bound for $1 \le \alpha < \beta$
(Conjecture~\ref{conj:wh}), but we have not been able to prove that in general.

For large shape parameters the picture is clearer.  Along any ray $\beta = r\alpha$ with $r > 1$, the logit of the median has an explicit second-order expansion (Theorem~\ref{thm:largeexp}), which shows that $\nu_{\textrm{B}2}$ is an upper bound there for large enough $\alpha$ (Theorem~\ref{thm:whray}).  A uniform bound on the remainder of that expansion (Conjecture~\ref{conj:largeR}) would imply Conjecture~\ref{conj:wh} over the whole region $1 \le \alpha < \beta$.  The $\nu_{\textrm{B}4}$ conjecture (Conjecture~\ref{thm:B4bound}) is proved outright once both shape parameters exceed about 6 (Theorem~\ref{thm:B4large}).

Being defined, or positive, only down to $\tfrac{1}{9}$, $\nu_{\textrm{B}2}$ is again not fully satisfying our quest for approximations that work over the full domain of shape parameters.

\subsection{Gamma median's asymptotic expansion}
\label{sec:asymexp}

Partial sums of the known asymptotic expansion for $\nu_{\Gamma}(\alpha)$ \citep{choi1994medians,berg2006chen} can be used for $g$:
\[
g(\alpha) = \alpha - \tfrac{1}{3} + \tfrac{8}{405 \alpha} +
  \tfrac{184}{25515 \alpha^2} + \tfrac{2248}{3444525 \alpha^3} -
  \tfrac{19006408}{15345358875 \alpha^4} + ... \,.
\]
These make approximations $g(\alpha) / (g(\alpha)+g(\beta))$ that can be quite accurate for
$\alpha$ and $\beta$ greater than 1, and more accurate than Kerman's even down to $\alpha$ below $\tfrac{1}{2}$.
However, because this is an asymptotic series whose coefficients grow factorially,
it diverges for any fixed $\alpha$ as the number of terms grows \citep{berg2006chen}.
At lower values of the shape parameters, the sum diverges very rapidly, with increasing-order
poles at 0.

For large shape parameters, Theorem~\ref{thm:cfamily}, also proved in \citet{lyon2026kerman}, settles which ratio forms with
$g(x) = x - \tfrac{1}{3} + c/x + O(x^{-2})$ are asymptotic bounds.  Along every ray $\beta = r\alpha$ with $r > 1$,
such a form is eventually a lower bound if $c \le \tfrac{8}{405}$ and an upper bound if
$c \ge \tfrac{1}{45}$; for $c$ in between, which side it falls on depends on the ray.  Kerman's
$\nu_{\textrm{B}1}$ ($c = 0$) and every partial sum of the series beyond the first
($c = \tfrac{8}{405}$) are therefore asymptotic lower bounds, and the Wilson--Hilferty $\nu_{\textrm{B}2}$
($c = \tfrac{1}{27}$, from $\nu_{\Gamma 2}(x) = x - \tfrac{1}{3} + \tfrac{1}{27x} - \tfrac{1}{729x^2}$)
is an asymptotic upper bound.  The true gamma median itself has $c = \tfrac{8}{405}$, so the dividing
line begins exactly at the coefficient of its asymptotic expansion.  In fact, along each ray
$\beta = r\alpha$, the median itself agrees to second order with the ratio form whose $c$ is the
threshold $c_r$ of Theorem~\ref{thm:cfamily}, whose formula is unchanged when $r$ is replaced by $1/r$.

Though no member of this family serves over the full domain, the family is far from a dead end.
Over most of the domain its members bracket the median from both sides, and the value of $c$
that reproduces the median exactly varies only slowly with $\alpha$ and $\beta$.  That makes $c$ a
natural quantity to tune as a function of the shape parameters.  The ratio form increases with $c$
for $\tfrac{1}{6} < \alpha < \beta$ (where $g > 0$):  its logit $\log g(\alpha) - \log g(\beta)$ has
$c$-derivative $1/(\alpha g(\alpha)) - 1/(\beta g(\beta))$, which is positive because
$x\,g(x) = x^2 - \tfrac{1}{3}x + c$ is increasing for $x > \tfrac{1}{6}$.  So the ratio form is an upper bound
wherever $c$ exceeds the value $c^*(\alpha,\beta)$ that makes it exact.  Numerically, $c^*$ has a single peak,
about $0.0344$ at $\alpha = \beta \approx 0.468$ (as a limit along the diagonal), so any constant above it,
such as $c = \tfrac{1}{29}$, gives an upper bound for $\alpha < \beta$ wherever both shape parameters exceed
$\tfrac{1}{6}$; \citet{lyon2026kerman} maps $c^*$ in detail.

Letting $c$ vary with the shape parameters does better still.  The form is then
\[
\frac{g_c(\alpha)}{g_c(\alpha) + g_c(\beta)} \, , \qquad g_c(x) = x - \tfrac{1}{3} + \frac{c}{x} \, , \quad
c = c(\alpha,\beta) \, ,
\]
which is not a ratio form in the usual sense:  since $c$ depends on both shape parameters, $g_c$ is not a
function of one variable, and Theorem~\ref{thm:non} does not rule out exactness; indeed $c = c^*(\alpha,\beta)$
makes it exact by definition.  The aim is an elementary $c$ close to $c^*$.  On the line $\beta = 1$, where the median
is $2^{-1/\alpha}$, solving $g_c(\alpha)/(g_c(\alpha) + g_c(1)) = 2^{-1/\alpha}$ for $c$ gives
\[
c^*(\alpha,1) = \frac{\alpha\left(\tfrac{2}{3}\,\nu_{\Gamma 4}(\alpha) - (\alpha - \tfrac{1}{3})\log 2\right)}{\log 2 - \alpha\,\nu_{\Gamma 4}(\alpha)} \, ,
\]
where $\nu_{\Gamma 4}$ is the gamma median approximation of Section~\ref{sec:gamma}; by symmetry,
$c^*(1,\beta) = c^*(\beta,1)$.  As a function of $\alpha$, $c^*(\alpha,1)$ rises to a maximum of about $0.03347$
near $\alpha = 0.404$, is exactly $\tfrac{1}{30} \approx 0.03333$ at $\alpha = \tfrac{1}{2}$, and falls through
$c^*(1,1) \approx 0.03177$ (a removable $0/0$) to about $0.02648$ as $\alpha \to \infty$.  For
large shape parameters, $c^*$ tends instead to the ray asymptote of Theorem~\ref{thm:cfamily}, which, written
symmetrically, is $c_\infty(\alpha,\beta) = \tfrac{8}{405} + 4\alpha\beta/\left(405(\alpha+\beta)^2\right)$.

To match both, let $D(x) = c^*(x,1) - c_\infty(x,1)$, the departure of $c^*(x,1)$ from the asymptote, and take
\[
c_\textrm{E}(\alpha,\beta) = c_\infty(\alpha,\beta) + \frac{D(\alpha)}{\beta} + \frac{D(\beta)}{\alpha}
- \frac{D(1)}{\alpha\beta} \, .
\]
On the line $\beta = 1$, the last two terms cancel, leaving $c_\textrm{E}(\alpha,1) = c_\infty(\alpha,1) + D(\alpha)
= c^*(\alpha,1)$; likewise $c_\textrm{E}(1,\beta) = c^*(\beta,1)$, since $c_\infty$ is symmetric.  As both shape parameters
grow, the correction terms vanish, leaving $c_\infty$.  The resulting approximation $\nu_\textrm{E}$, the form
above with $c = c_\textrm{E}$, is therefore exact on the lines $\alpha = 1$ and $\beta = 1$, as well as on
$\alpha = \beta$, where every form of this kind gives $\tfrac{1}{2}$.  We call it \emph{edge-exact}, since the lines
$\alpha = 1$ and $\beta = 1$ are edges of the regions that Sections~\ref{sec:odds} and~\ref{sec:whole} work in.  With no fitted constants, it is within
$0.016\%$ of the median for $1 \le \alpha \le \beta \le 10^6$, and adding the single term
$0.000674\,(1 - 1/\alpha)(1 - 1/\beta)$ to $c_\textrm{E}$ brings that to $0.005\%$.

Written out in full, with $c^*(x,1)$, $\nu_{\Gamma 4}$, and $c_\infty$ expanded in place (the factors $\log 2$
in $c^*(x,1)$ cancel), the edge-exact approximation is
\begin{equation}
\nu_\textrm{E}(\alpha,\beta) = \frac{\alpha - \frac{1}{3} + c_\textrm{E}/\alpha}
{\alpha + \beta - \frac{2}{3} + c_\textrm{E}/\alpha + c_\textrm{E}/\beta} \, ,
\label{eq:nuE}
\end{equation}
where
\begin{align*}
c_\textrm{E} &= \frac{8}{405} + \frac{4\alpha\beta}{405(\alpha+\beta)^2} + \frac{D(\alpha)}{\beta}
+ \frac{D(\beta)}{\alpha} - \frac{D(1)}{\alpha\beta} \, , \\[4pt]
D(x) &= \frac{x\left(\frac{2}{3} - \left(x - \frac{1}{3}\right)\left(2^{1/x} - 1\right)\right)}{2^{1/x} - 1 - x}
- \frac{8}{405} - \frac{4x}{405(1+x)^2} \, , \\[4pt]
D(1) &= \frac{3 - 4\log 2}{3(1 + 2\log 2)} - \frac{1}{45} \, ,
\end{align*}
and $D(x)$ has a removable $0/0$ at $x = 1$.

Back with a constant $c$, the form is again a ratio form with a function $g$ of one variable, and even a
single value does well, as developed in \citet{lyon2026kerman}, where the approximation below is called
$\nu_{1/36}$.
With $c = \tfrac{1}{36}$, $g(x)$ factors as $(x - \tfrac{1}{6})^2/x$,
giving
\begin{equation}
\nu_{\textrm{B}3}(\alpha, \beta) = \frac{(\alpha - \tfrac{1}{6})^2/\alpha}
  {(\alpha - \tfrac{1}{6})^2/\alpha + (\beta - \tfrac{1}{6})^2/\beta} \, .
\label{eq:nuB3}
\end{equation}
Numerically, its relative error lies between $-0.171\%$ and $+0.212\%$ for all
$1 \le \alpha < \beta$, an order of magnitude better than $\nu_{\textrm{B}1}$ and
six times better than $\nu_{\textrm{B}2}$ there, though it is not a bound.  Numerically it is more
accurate than Kerman's approximation wherever the latter is defined, and $\tfrac{1}{36} \approx 0.027778$
is within $1\%$ of the minimax constant over $1 \le \alpha \le \beta$, numerically $c \approx 0.027537$.
Its accuracy degrades quickly for $\alpha < \tfrac{1}{2}$, since $g$ vanishes at
$x = \tfrac{1}{6}$.  It is the ratio form built from the gamma median approximation
$\nu_{\Gamma 3}(x) = (x - \tfrac{1}{6})^2/x$, and $\beta\,\nu_{\textrm{B}3}(\alpha,\beta) \to \nu_{\Gamma 3}(\alpha)$
as $\beta \to \infty$; we return to it in Section~\ref{sec:gamma}.

\subsection{Middle special case}

Figures~\ref{fig:rel3} and~\ref{fig:rel4} also include one more approximation (black dotted), based on a similar ratio formula designed to be accurate when either $\alpha$ or $\beta$ is equal to $1$.  Note that a function $g$ that satisfies $\nu_\textrm{B}(\alpha,\beta) = \frac{g(\alpha)}{g(\alpha) + g(\beta)}$ for all $\alpha > 0$ and $\beta > 0$ must satisfy it at $\beta = 1$, where the median is easily found to be $2^{-1/\alpha}$, and therefore $g$ must satisfy:
\[
\frac{g(\alpha)}{g(1)} = \frac{1}{2^{1/\alpha} - 1} \quad \textrm{for all} \quad \alpha > 0 \, ,
\]
which we prove as Lemma~\ref{lem:special} in the appendix.

For any nonzero $g(1)$, a function $g$ that satisfies the constraint of Lemma~\ref{lem:special} leads to this solution in the given form, equal to the true median when $\beta = 1$ or $\alpha = 1$ or $\alpha = \beta$:
\[
\nu_{\textrm{B}4}(\alpha,\beta) = \frac{2^{1/\beta} - 1}{2^{1/\alpha} + 2^{1/\beta} - 2} \, .
\]

\begin{figure}[!htb]
  \centering
  \includegraphics[width=0.6\linewidth]{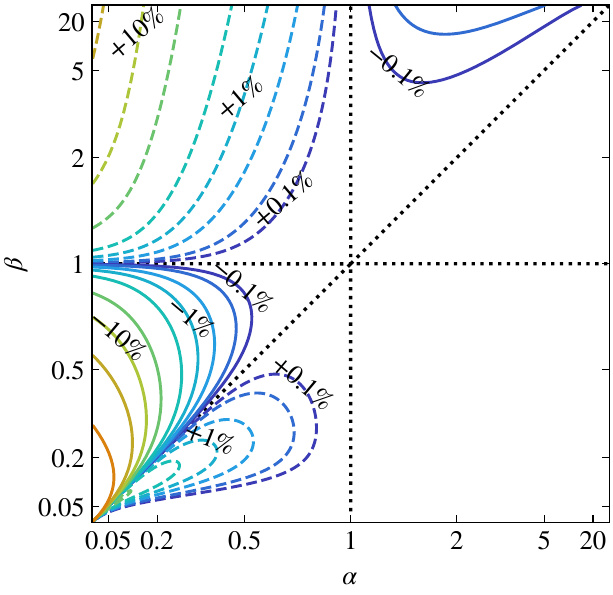}
    \caption{{\bf Relative error contours of \texorpdfstring{$\nu_{\textrm{B}4}$}{nu B4}.}
    Large regions of parameter space lead to relative errors better than 1\%, but there is also a large low-$\alpha$ region with larger errors.  In this figure, the errors in the lower-right triangle with $\alpha > \beta$ are formed by reflection; the magnitudes of the relative errors there are less, since the errors are relative to higher median values.  Contours are at $\pm 0.1\%$, $\pm 0.2\%$, $\pm 0.5\%$, $\pm 1\%$, $\pm 2\%$, $\pm 5\%$, $\pm 10\%$, $\pm 20\%$, and $\pm 50\%$, dashed for positive and solid for negative errors, with the zero contour dotted; those at powers of 10 are labeled.  Where $\beta < 1 < \alpha$, the error stays within $0.1\%$, so no contours appear there.  The negative errors
    approach $-100\%$ as both shape parameters approach 0; the positive errors are below $25\%$
    (Theorem~\ref{thm:B4over}).}
  \label{fig:contour1}
\end{figure}

The error sign of the approximation $\nu_{\textrm{B}4}$ divides the $\alpha-1,\beta-1$ plane by quadrants, and along the diagonal, as is clear in Figure~\ref{fig:contour1}.  That is, the sign of the error $\nu_{\textrm{B}4}(\alpha,\beta) - \nu_{\textrm{B}}(\alpha,\beta)$ is identical to the sign of the product $(\alpha-1)(\beta-1)(\alpha-\beta)$.  We formalize this observation as Conjecture~\ref{thm:B4bound} in the appendix, where we provide strong supporting evidence and a partial proof.

Simple as it is, $\nu_{\textrm{B}4}$ is usable over the whole domain:  it lies strictly between 0 and 1 for all
$\alpha, \beta > 0$.  Numerically, it is within $0.30\%$ of the
median when both shape parameters are at least 1 (the worst case is the limit $\beta \to \infty$, where its error
becomes that of $\nu_{\Gamma 4}$), within $1.6\%$ when both are at least $\tfrac{1}{2}$, and within $3.7\%$ when
both are at least $\tfrac{1}{3}$, below which Kerman's approximation goes negative.  Its large errors are confined to small $\alpha$, where the median is smaller than any
power of $\alpha$:  as $\alpha \to 0$ with $\beta$ fixed,
\[
\frac{\nu_{\textrm{B}4}(\alpha,\beta)}{\nu_\textrm{B}(\alpha,\beta)} \to (2^{1/\beta} - 1)\,e^{\gamma + \psi_0(\beta)} \, ,
\]
where $\gamma$ is the Euler--Mascheroni constant and $\psi_0$ the digamma function (Theorem~\ref{thm:B4over}).
This limit tends to $0$ as $\beta \to 0$; it is about $0.28$ at $\beta = 0.2$, $\tfrac{3}{4}$ at $\beta = \tfrac{1}{2}$, and
$1$ at $\beta = 1$; and it tends to $e^\gamma\log 2 \approx 1.2345$ as $\beta \to \infty$.  So $\nu_{\textrm{B}4}$ falls far
below the median only where both shape parameters are small, and it never exceeds the median by much:
Theorem~\ref{thm:B4over} proves $\nu_{\textrm{B}4} < 1.246\,\nu_\textrm{B}$ for $\alpha \le \beta$, and numerically the
largest overestimate is the limiting $23.45\%$.

\subsection{Non-existence of \texorpdfstring{$g$}{g}}

Figures~\ref{fig:rel3}, \ref{fig:rel4}, and~\ref{fig:contour1} show that $\nu_{\textrm{B}4}(\alpha,\beta)$ has nonzero error relative to the true median, which implies that \emph{no function $g$ of one variable leads to the true median via that relation in general}.  But the plots rely on someone's numerical implementation of the inverse incomplete beta function, so we formalize the non-existence without relying on such things, in Theorem~\ref{thm:non} (see appendix).

\subsection{An interpolated approximation}
\label{sec:interp}

Since no ratio form is exact, we can do better by combining two.  Where both shape parameters are at least 1,
two of the ratio forms above bracket the median.  One is
$\nu_{\textrm{B}4}$, which is exact on the line $\alpha = 1$ and, by Conjecture~\ref{thm:B4bound}, below the median
for $1 < \alpha < \beta$.  The other is the form of Section~\ref{sec:asymexp} with $c = \tfrac{1}{30}$,
\[
\nu_{1/30}(\alpha,\beta) = \frac{\alpha - \frac{1}{3} + \frac{1}{30\alpha}}
{\alpha + \beta - \frac{2}{3} + \frac{1}{30\alpha} + \frac{1}{30\beta}} \, ,
\]
which is above the median there:  the form increases with $c$ for $\alpha < \beta$, and numerically the value
$c^*$ that makes it exact for $1 \le \alpha \le \beta$ is at most $c^*(1,1) \approx 0.03177$.  The relative error
of $\nu_{1/30}$ is between $0$ and $+0.99\%$, largest on the line $\alpha = 1$, and it is the more accurate of the
two for large shape parameters, since its $g$ has the right constant term, $-\tfrac{1}{3}$, which
$\nu_{\Gamma 4}$ lacks.  Weighting $\nu_{\textrm{B}4}$ by $1/\alpha$ and $\nu_{1/30}$ by $(\alpha - 1)/\alpha$ gives
the interpolated approximation
\begin{equation}
\nu_\textrm{I}(\alpha,\beta) = \frac{\alpha - 1}{\alpha} \cdot
\frac{\alpha - \frac{1}{3} + \frac{1}{30\alpha}}{\alpha + \beta - \frac{2}{3} + \frac{1}{30\alpha} + \frac{1}{30\beta}}
+ \frac{1}{\alpha} \cdot \frac{2^{1/\beta} - 1}{2^{1/\alpha} + 2^{1/\beta} - 2}
\qquad \textrm{for } 1 \le \alpha \le \beta \, ,
\label{eq:nuI}
\end{equation}
extended to $1 \le \beta < \alpha$ by $\nu_\textrm{I}(\alpha,\beta) = 1 - \nu_\textrm{I}(\beta,\alpha)$.  It is exact
on the line $\alpha = 1$, where all the weight is on $\nu_{\textrm{B}4}$, and on the diagonal, where both forms
give $\tfrac{1}{2}$; as $\alpha$ grows, the weight shifts to $\nu_{1/30}$.  It has no fitted constant, and
numerically its relative error is between $-0.003\%$ and $+0.016\%$ for $1 \le \alpha \le \beta \le 10^6$:  the
same worst case as the edge-exact $\nu_\textrm{E}$ of Equation~\eqref{eq:nuE}, from a much simpler formula.  Fitting the constant in place of
$\tfrac{1}{30}$, at about $0.0328$, lowers the worst error to $0.011\%$, but it makes the limit $\beta \to \infty$,
an approximation to the gamma median for $\alpha \ge 1$ (Section~\ref{sec:gamma}), worse:  from within $0.003\%$
to within $0.011\%$.  Section~\ref{sec:whole} uses $\nu_\textrm{I}$ for $\alpha \ge 1$ in an approximation over the
whole domain.

\section{Approach from the low end}\label{sec:lowend}

When the low-$\alpha$ region is of interest, we need a different approach.  Figure~\ref{fig:loglog} illustrates the interesting behavior via a log-log plot, revealing that the new $\nu_{\textrm{B}4}(\alpha,\beta)$ approximation has roughly the right form, but is not asymptotically accurate except when $\beta=1$.  So we consider a new approach.

\begin{figure}[!ht]
  \centering
  \includegraphics[width=0.8\linewidth]{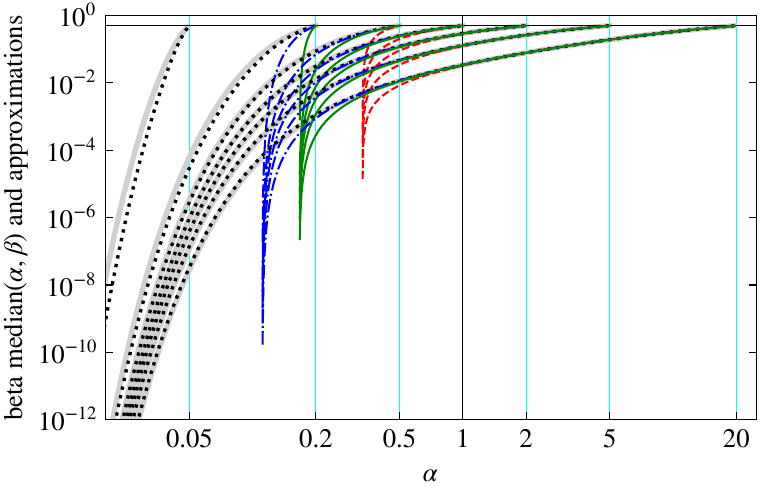}
    \caption{{\bf Medians and approximations on log-log axes.}
    Here the true median and four approximations are plotted on log-log axes, each using the same curve styles as previous figures, for the same set of $\beta$ values, where $\alpha \le \beta$.
    Notice that $\nu_{\textrm{B}4}(\alpha,\beta)$ (black dotted) follows the general shape of the true median (wide light gray) at low $\alpha$ as the values get to be extraordinarily close to 0, though not with a particularly small relative error, while the other ratio-form approximations go to zero long before reaching such low $\alpha$ values.}
  \label{fig:loglog}
\end{figure}

Instead of the ratio-type formula, we can solve directly for a low-$\alpha$ asymptote, as \citet{lyon2021closed} did for the gamma distribution median.  When the median is near 0, the factor $(1-x)^{\beta-1}$ in the PDF region being integrated is very close to 1.  So we get a useful asymptote by replacing it with 1, which makes the integral equation simple to solve:
\[
\frac{1}{2} = \frac{1}{B(\alpha,\beta)} \int_0^{\nu_{\textrm{B}5}(\alpha,\beta)}x^{\alpha-1}dx
\]
\[
\frac{1}{2} = \frac{1}{B(\alpha,\beta)} \frac{\nu_{\textrm{B}5}(\alpha,\beta)^{\alpha}}{\alpha}
\]
\[
\nu_{\textrm{B}5}(\alpha,\beta) = \left(\frac{\alpha B(\alpha,\beta)}{2} \right)^{1/\alpha}
\]

We use the known expression for the beta function in terms of the gamma function to be able to absorb the $\alpha$ factor; that is, use:
\[
B(\alpha,\beta) = \frac{\Gamma(\alpha) \Gamma(\beta)}{\Gamma(\alpha+\beta)} \quad \textrm{and} \quad \alpha \Gamma(\alpha) = \Gamma(\alpha+1)
\]
to rewrite the approximation as
\[
\nu_{\textrm{B}5}(\alpha,\beta) = 2^{-1/\alpha} \left(\frac{\Gamma(\alpha+1) \Gamma(\beta)}{\Gamma(\alpha+\beta)} \right)^{1/\alpha}
\]
which is the asymptote found by \citet{askitis2017asymptotic} via an asymptotic series expansion method (his Corollary 1, except that in his article he inverts the factor involving the gamma function when first stating the result, before getting it right in the proof).  We prefer to write it factored this way, with the $2^{-1/\alpha}$ factor, similar to how we expressed bounds and approximations to the gamma distribution median \citep{lyon2021closed, lyon2023tight}, though writing that factor as $(\tfrac{1}{2})^{1/\alpha}$ might be seen as more logical, with the $\frac{1}{2}$ coming from the CDF value at the median.

Theorem~\ref{thm:low} shows that $\nu_{\textrm{B}5}(\alpha,\beta)$ is a lower bound for $\beta > 1$
and an upper bound for $\beta < 1$, and is exact at $\beta = 1$.

When $\alpha$ and $\beta$ are both low (i.e.~in the bimodal case), there is a concentration of probability near 0 proportional to $\frac{1}{\alpha}$ and a concentration near 1 proportional to $\frac{1}{\beta}$ (that is, with the same constant of proportionality).  If $\alpha < \beta$ and we treat the PDF near 0 as approximately a gamma distribution with shape parameter $\alpha$, but scaled down by a factor $\frac{\beta}{\alpha+\beta}$ to account for the total distribution being split between the two modes, then we can approximate the median by finding where the unscaled gamma CDF reaches $\frac{1}{2} \frac{\alpha+\beta}{\beta}$ instead of just $\frac{1}{2}$.  That suggests writing $\nu_{\textrm{B}5}$ in the factored form
\[
\nu_{\textrm{B}5}(\alpha,\beta) = P(\alpha,\beta)\, F(\alpha,\beta) \, ,
\]
with
\[
P(\alpha,\beta) = \left(2\,\frac{\beta}{\alpha+\beta}\right)^{-1/\alpha} , \qquad
F(\alpha,\beta) = \left(\frac{\Gamma(\alpha+1) \Gamma(\beta+1)}{\Gamma(\alpha+\beta+1)} \right)^{1/\alpha} .
\]
The first factor carries the rapid variation of the median at low $\alpha$; the second is well behaved.
As $\alpha \to 0$, $F(\alpha,\beta)$ tends to $e^{-\gamma - \psi_0(\beta+1)}$, which is finite for all
$\beta \ge 0$, and $\log F$ is concave in $\alpha$, with slope $\frac{\pi^2}{12} - \frac{1}{2}\psi_1(\beta+1)$
at $\alpha = 0$ (Theorem~\ref{thm:lowasym} and Lemma~\ref{lem:logconc}, where $F$ is written $M_\beta$).  The
limit is equivalent to \citeauthor{askitis2017asymptotic}'s (\citeyear{askitis2017asymptotic}) finding that
$2^{1/\alpha} \nu_{\textrm{B}}(\alpha,\beta) \to e^{- \gamma -\psi_0(\beta)}$.

Figure~\ref{fig:rel5} shows the relative error of $\nu_{\textrm{B}5}$.  It is excellent for $\alpha \ll \beta$,
but for $\beta < 1$ it is too high near $\alpha = \beta$, by nearly $100\%$ when $\beta$ is small, and for
large $\beta$ it is too low as $\alpha \to 1$, by up to $28\%$.

\begin{figure}[!ht]
  \centering
  \includegraphics[width=0.8\linewidth]{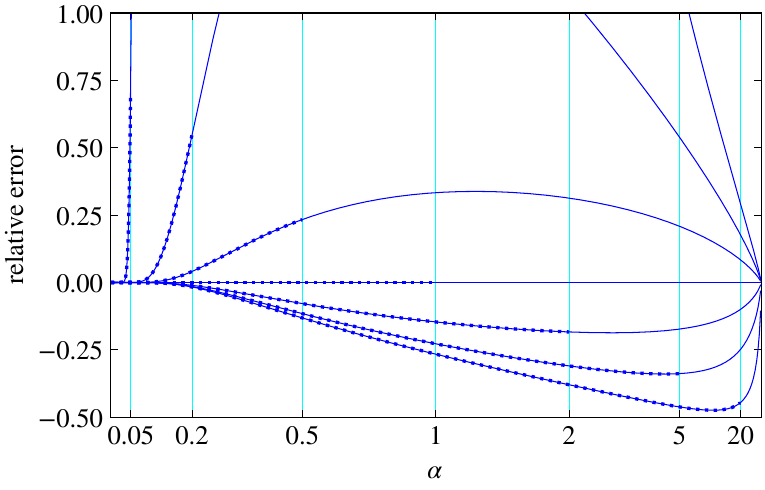}
    \caption{{\bf Relative error of the low-\texorpdfstring{$\alpha$}{alpha} asymptote.}
    The relative error of $\nu_{\textrm{B}5}(\alpha,\beta)$ (blue dotted for $\alpha \le \beta$), which is
    exact at $\beta = 1$.  The curves are extended (thin solid blue) to $\alpha > \beta$ to illustrate the full
    domain of Theorem~\ref{thm:low}.}
  \label{fig:rel5}
\end{figure}

\subsection{Elementary bounds}
\label{sec:lowrefactor}

Replacing $\log F$ by its first- or second-order Taylor polynomial at $\alpha = 0$, and the polygamma functions in the coefficients by
one-sided bounds \citep{alzer1997inequalities}, gives elementary
functions that provably bound the median.  Here elementary means built from arithmetic operations, powers,
exponentials, and logarithms, with no gamma or polygamma functions.  Those bounds come from the asymptotic series of the polygamma
functions, which are loosest at small arguments, so we evaluate them at $w = \beta + 2$ rather than at
$z = \beta + 1$, through the recurrences $\psi_n(z) = \psi_n(w) - (-1)^n n!/z^{n+1}$.  With $z = \beta + 1$
and $w = \beta + 2$, define
\begin{align*}
\nu_{\textrm{B}6}(\alpha,\beta) &= \left(2\,\frac{\beta}{\alpha+\beta}\right)^{-1/\alpha}
\underline{c}_0(\beta)\, e^{\underline{k}(\beta)\,\alpha - \eta(\beta)\,\alpha^2} \, , \\
\nu_{\textrm{B}7}(\alpha,\beta) &= \left(2\,\frac{\beta}{\alpha+\beta}\right)^{-1/\alpha}
\overline{c}_0(\beta)\, e^{\overline{k}(\beta)\,\alpha} \, ,
\end{align*}
where
\begin{align*}
\overline{c}_0(\beta) &= \frac{e^{-\gamma + 1/z}}{w}\, e^{\frac{1}{2w} + \frac{1}{12w^2}} \, , \qquad
\underline{c}_0(\beta) = \overline{c}_0(\beta)\, e^{-\frac{1}{120w^4}} \, , \\
\overline{k}(\beta) &= \frac{\pi^2}{12} - \frac{1}{2z^2}
 - \frac{1}{2}\left(\frac{1}{w} + \frac{1}{2w^2} + \frac{1}{6w^3} - \frac{1}{30w^5}\right) , \\
\underline{k}(\beta) &= \overline{k}(\beta) - \frac{1}{84w^7} \, , \\
\eta(\beta) &= \frac{\zeta(3)}{3} - \frac{1}{6}\left(\frac{2}{z^3} + \frac{1}{w^2} + \frac{1}{w^3} + \frac{1}{2w^4} - \frac{1}{6w^6}\right) .
\end{align*}
Evaluating at $w$ makes these much tighter for small $\beta$:  as $\beta \to 0$, $\overline{c}_0$ exceeds the
limit of $F$ at $\alpha = 0$ by only $0.05\%$, where the same bound evaluated at $z$ exceeds it by $0.6\%$.
Theorem~\ref{thm:lowbounds} shows that $\nu_{\textrm{B}6} < \nu_{\textrm{B}5} < \nu_{\textrm{B}7}$ for all
$\alpha, \beta > 0$.  With Theorem~\ref{thm:low}, $\nu_{\textrm{B}6}$ is therefore a lower bound on the
median for $\beta > 1$, and $\nu_{\textrm{B}7}$ an upper bound for $\beta < 1$.  The quadratic term is
what makes $\nu_{\textrm{B}6}$ a lower bound:  the logarithm of the second factor of $\nu_{\textrm{B}5}$
has a positive third derivative, so its second-order Taylor polynomial lies below it and its first-order
one above it.  That is also why $\nu_{\textrm{B}7}$ stops at first order.

\subsection{Odds-corrected low-\texorpdfstring{$\alpha$}{alpha} bounds}
\label{sec:odds}

The bounds $\nu_{\textrm{B}6}$ and $\nu_{\textrm{B}7}$ keep the leading factor $P$ and approximate only the
second factor $F$, so they inherit the two failures of $\nu_{\textrm{B}5}$ (Figure~\ref{fig:rel5}):  for
$\beta < 1$, it is too high near $\alpha = \beta$, by nearly $100\%$ when $\beta$ is small, and for large
$\beta$, it is too low near $\alpha = 1$, by up to $28\%$.  Neither failure comes from approximating $F$.  This
section fixes both, in four steps:
\begin{enumerate}
\item Correct the odds of $\nu_{\textrm{B}5}$, which fixes most of both failures.
\item Replace $F$ by an elementary approximation $G$ that is exact on the lines $\beta = 1$ and $\alpha = 1$.
\item Choose the exponent of the odds correction to make the result exact on the line $\alpha = \beta$ as well.
\item Steer that exponent to make one-sided bounds.
\end{enumerate}
The lines $\alpha = 1$, $\beta = 1$, and $\alpha = \beta$ are where the median is known exactly:  it is
$1 - 2^{-1/\beta}$, $2^{-1/\alpha}$, and $\tfrac{1}{2}$ there.  So the result is exact on the same lines as
$\nu_{\textrm{B}4}$, and its errors are $0$ on the edges of the two low-$\alpha$ regions,
$\alpha \le \beta < 1$ and $\alpha \le 1 < \beta$.

\paragraph{Step 1:  the odds correction.}
The asymptote $\nu_{\textrm{B}5}$ comes from dropping the factor $(1-x)^{\beta-1}$ from the integrand.  That is
fine when the median is near $0$, but not near $\alpha = \beta$ for small $\beta$, where much of the
probability is near $1$.  Keeping the factor approximately, by replacing the incomplete beta function
$B_x(\alpha,\beta)$ with $(x^\alpha/\alpha)(1-x)^{-\alpha(1-\beta)}$, which is exact at $\beta = 1$ and correct
to first order in $x$ for small $\alpha$, turns the median equation into $x/(1-x)^{1-\beta} = \nu_{\textrm{B}5}$.
For small $\beta$, that says that the odds $x/(1-x)$ of the median equal $\nu_{\textrm{B}5}$, so that
\[
\nu_\textrm{B}(\alpha,\beta) \approx \frac{\nu_{\textrm{B}5}(\alpha,\beta)}{1 + \nu_{\textrm{B}5}(\alpha,\beta)} \, ,
\]
which is $\tfrac{1}{2}$ on the diagonal in the limit, where $\nu_{\textrm{B}5}$ tends to $1$.  Rather than solve
the equation for each $\beta$, we use the explicit form
\begin{equation}
\frac{v}{(1+v)^{e}} \, ,
\label{eq:oddsform}
\end{equation}
with $v$ equal to $\nu_{\textrm{B}5}$ or an approximation of it, and an exponent $e$ that is near $1$ for small
$\beta$ and $0$ at $\beta = 1$, where $\nu_{\textrm{B}5}$ is exact.  For $\beta > 1$, the exponent is negative,
and the correction raises $v$, which is what $\nu_{\textrm{B}5}$ needs near $\alpha = 1$.  With
$v = \nu_{\textrm{B}5}$ and $e$ chosen to make the form exact on the edge of each region, the relative error is
already within $\pm 2\%$ over the whole region $0 < \alpha \le \min(\beta,1)$.

\paragraph{Step 2:  an elementary replacement for $F$.}
Since $\nu_{\textrm{B}5}$ is exact at $\beta = 1$, the factor $F(\alpha,1)$ is $(1+\alpha)^{-1/\alpha}$, so we write
$\log F(\alpha,\beta) = -\log(1+\alpha)/\alpha + R(\alpha,\beta)$, where $R$ is $0$ at $\beta = 1$.  Three facts
about $R$ are known:
\begin{itemize}
\item the value at $\alpha = 1$ is $R(1,\beta) = \log 2 - \log(1+\beta)$, since $F(1,\beta) = 1/(1+\beta)$;
\item the value at $\alpha = 0$ is $R(0,\beta) = 1 - H_\beta$, where $H_\beta = \gamma + \psi_0(\beta+1)$ is the harmonic
number;
\item the slope in $\alpha$ at $\alpha = 0$ is
$R'(0,\beta) = \pi^2/12 - \psi_1(\beta+1)/2 - \tfrac{1}{2}$ (Lemma~\ref{lem:logconc}).
\end{itemize}
The quadratic in $\alpha$ that matches all three is
$(1-\alpha^2)\,R(0,\beta) + \alpha^2\,R(1,\beta) + \alpha(1-\alpha)\,R'(0,\beta)$.  We use it in place of
$R$, with elementary approximations of $R(0,\beta)$ and $R'(0,\beta)$:
\begin{equation}
G(\alpha,\beta) = (1+\alpha)^{-1/\alpha} \left(\frac{2}{1+\beta}\right)^{\alpha^2}
\left(\frac{3}{2+\beta}\right)^{1-\alpha^2} \exp\left((1-\alpha^2)\,\sigma(\beta) + \alpha(1-\alpha)\,\tau(\beta)\right) ,
\label{eq:Gdef}
\end{equation}
where $\log\tfrac{3}{2+\beta} + \sigma(\beta)$ approximates $R(0,\beta)$, $\tau(\beta)$ approximates
$R'(0,\beta)$, and
\begin{align*}
\sigma(\beta) &= \frac{1}{\beta+1} + \frac{1}{2(\beta+2)} + \frac{1}{12(\beta+2)^2} - \frac{73}{108} \, , \\
\tau(\beta) &= \frac{9403}{29160} - \frac{1}{2(\beta+1)^2} - \frac{1}{2(\beta+2)} - \frac{1}{4(\beta+2)^2}
- \frac{1}{12(\beta+2)^3} + \frac{1}{60(\beta+2)^5} \, .
\end{align*}
These come from the leading terms of the asymptotic series of $\psi_0$ and $\psi_1$, evaluated at $\beta + 2$ as
for $\nu_{\textrm{B}6}$ and $\nu_{\textrm{B}7}$, with constants that make them exact at $\beta = 1$, where
$\sigma(1) = \tau(1) = 0$.  So $G$ is exact at $\beta = 1$ for every $\alpha$, and at $\alpha = 1$ for every
$\beta$.  Elsewhere, it is within $-0.26\%$ to $+0.04\%$ of $F$ for $\beta < 1$, and within $-0.01\%$ to $+0.18\%$
for $\beta > 1$.

\paragraph{Step 3:  exponents that make the form exact.}
With $v = P\,G$, the form~\eqref{eq:oddsform} is exact at $\beta = 1$ when the exponent is $0$ there.  On the
diagonal, $P$ is $1$ and the median is $\tfrac{1}{2}$; on the line $\alpha = 1$, $v$ is $1/(2\beta)$ and the median is
$1 - 2^{-1/\beta}$.  Solving for the exponent that makes the form exact on each of those lines gives
\[
e_\textrm{D}(\beta) = \frac{\log\left(2\,G(\beta,\beta)\right)}{\log\left(1 + G(\beta,\beta)\right)} \, , \qquad
e_1(\beta) = -\frac{\log\left(2\beta\left(1 - 2^{-1/\beta}\right)\right)}{\log\left(1 + \frac{1}{2\beta}\right)} \, ,
\]
with D for the diagonal and 1 for the line $\alpha = 1$; both are $0$ at $\beta = 1$.  With $e = e_\textrm{D}(\beta)$ for $\beta < 1$ and $e = e_1(\beta)$ for
$\beta > 1$, the form is exact on all three lines, and continuous across $\beta = 1$.  With no fitted
constants, its relative error is between $-0.70\%$ and $+0.13\%$ for $\alpha \le \beta < 1$, and between
$-1.74\%$ and $0$ for $\alpha \le 1 < \beta$, where it is a lower bound.

\paragraph{Step 4:  steering for one-sided bounds.}
To make one-sided bounds, we multiply the exponent by a steering factor that is $1$ on the edges of each
region, so that exactness on the three lines is kept:
\begin{align*}
\nu_{\textrm{B}8}(\alpha,\beta) &= \frac{v}{(1+v)^{e_\textrm{D}(\beta)\,\left(1 - 0.22\,(1-\beta)(1 - \alpha/\beta)\right)}}
&& \textrm{for } \alpha \le \beta < 1 \, , \\
\nu_{\textrm{B}9}(\alpha,\beta) &= \frac{v}{(1+v)^{e_1(\beta)\,\left(1 + 0.2\,(1-\alpha)\,\beta/(\beta+2)\right)}}
&& \textrm{for } \alpha \le 1 < \beta \, ,
\end{align*}
with $v = P(\alpha,\beta)\,G(\alpha,\beta)$.  Both are elementary functions of $\alpha$ and $\beta$.  Lowering the
exponent raises the form, so a larger first coefficient raises $\nu_{\textrm{B}8}$ (where
$e_\textrm{D} > 0$), and a larger second one raises $\nu_{\textrm{B}9}$ (where $e_1 < 0$).  On the grids below,
$\nu_{\textrm{B}8}$ stays an upper bound for a first coefficient of at least $0.215$, and $\nu_{\textrm{B}9}$ a
lower bound for a second coefficient of at most $0.209$; we use $0.22$ and $0.2$, with a little margin.

\begin{figure}[!ht]
  \centering
  \includegraphics[width=0.8\linewidth]{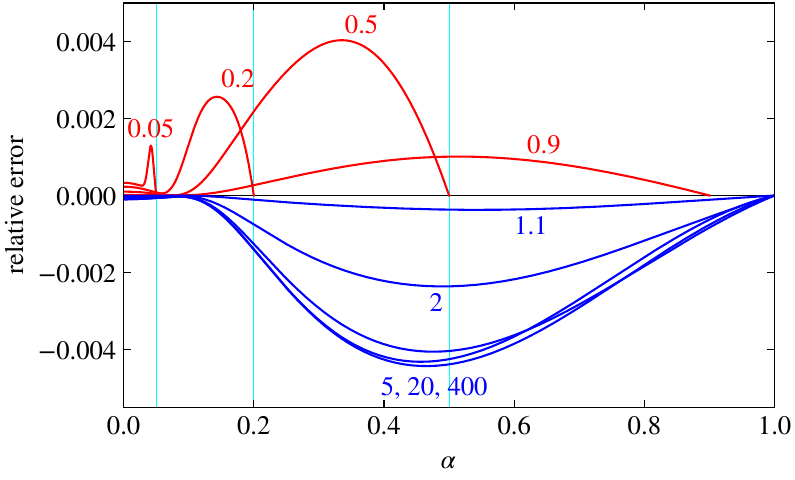}
    \caption{{\bf Relative errors of the odds-corrected low-\texorpdfstring{$\alpha$}{alpha} bounds.}
    The upper bound $\nu_{\textrm{B}8}(\alpha,\beta)$ (red) for $\beta = 0.05$, $0.2$, $0.5$, and $0.9$,
    each plotted up to $\alpha = \beta$, where it is exact, and the lower bound $\nu_{\textrm{B}9}(\alpha,\beta)$ (blue) for
    $\beta = 1.1$, $2$, $5$, $20$, and $400$, plotted up to $\alpha = 1$, where it is exact.  At $\alpha = 0.6$,
    the magnitude of the error of $\nu_{\textrm{B}9}$ increases with $\beta$; the curves for $5$, $20$, and
    $400$ lie close together.  Compare the errors of $\nu_{\textrm{B}5}$ in Figure~\ref{fig:rel5}, on a scale
    $150$ times larger.}
  \label{fig:relodds}
\end{figure}

\begin{figure}[!ht]
  \centering
  \includegraphics[width=0.57\linewidth]{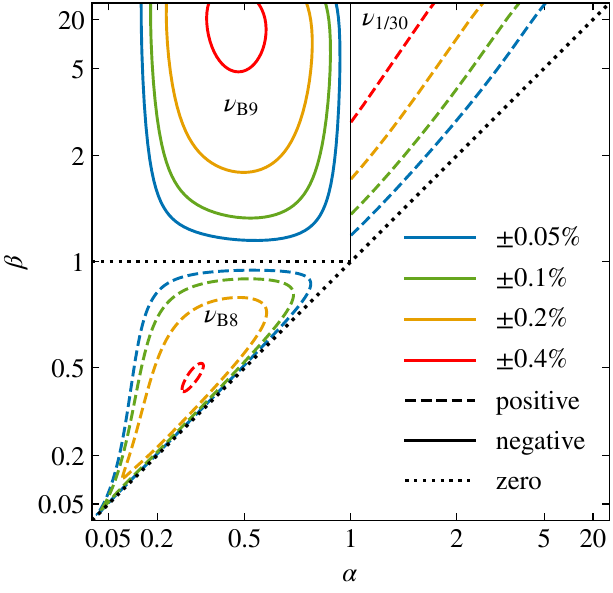}
    \caption{{\bf Relative error contours of the bounds \texorpdfstring{$\nu_{\textrm{B}8}$, $\nu_{\textrm{B}9}$, and $\nu_{1/30}$}{nu B8, nu B9, and nu 1/30}.}
    The upper bound $\nu_{\textrm{B}8}$ for $\alpha \le \beta < 1$, the lower bound $\nu_{\textrm{B}9}$ for
    $\alpha \le 1 < \beta$, and the upper bound $\nu_{1/30}$ of Section~\ref{sec:interp} for
    $1 \le \alpha \le \beta$.  Since these are one-sided bounds for $\alpha \le \beta$, the errors are shown
    only there, without reflection.  Contours at $\pm 0.05\%$, $\pm 0.1\%$, $\pm 0.2\%$, and $\pm 0.4\%$, colored as
    in the legend, dashed for positive and solid for negative errors, and dotted where the error is 0, as in
    Figure~\ref{fig:contour1}:  the bounds are exact on the region edges $\beta = 1$ and $\alpha = \beta$.  On the
    third edge, the solid line $\alpha = 1$, $\nu_{\textrm{B}9}$ is exact but $\nu_{1/30}$ is not, so the two do not
    join there, and the contours are drawn separately in each region.}
  \label{fig:contourodds}
\end{figure}

Numerically, $\nu_{\textrm{B}8}$ is an upper bound on the median for $0 < \alpha \le \beta < 1$, with
relative error between $0$ and $+0.41\%$, and $\nu_{\textrm{B}9}$ is a lower bound for
$0 < \alpha \le 1 < \beta$, with relative error between $-0.44\%$ and $0$ (Figures~\ref{fig:relodds}
and~\ref{fig:contourodds}).  We checked them on logarithmic grids of $300$ values of $\beta$, from $0.001$ to
$0.99999$ and from $1.00001$ to $10^6$, with about $450$ values of $\alpha$ for each, concentrated toward
$\alpha = 0$ and toward the edge of the region, and confirmed the ranges with 30-digit medians
(Section~\ref{sec:methods}), including at $3{,}000$ random points in each region with $\beta$ from $10^{-6}$
to $1$ and from $1$ to $10^9$.  The worst error of $\nu_{\textrm{B}8}$ is near $(\alpha,\beta) = (0.32, 0.47)$;
that of $\nu_{\textrm{B}9}$ is near $\alpha \approx 0.47$ for $\beta \approx 16$.  For comparison, the provable
bounds $\nu_{\textrm{B}6}$ and $\nu_{\textrm{B}7}$ have worst errors of $-38\%$ and about $+100\%$ on the same
regions.  The bounding behavior of $\nu_{\textrm{B}8}$ and $\nu_{\textrm{B}9}$ rests on the numerical checks.
Figure~\ref{fig:contourodds} also shows the upper bound $\nu_{1/30}$ of Section~\ref{sec:interp}, for
$1 \le \alpha \le \beta$.

\section{An approximation over the whole domain}
\label{sec:whole}

The pieces of Sections~\ref{sec:interp} and~\ref{sec:odds} combine into one elementary-function
approximation, $\nu_{\textrm{B}10}$, that is accurate over the whole domain.  As in Figure~\ref{fig:contourodds},
we split $\alpha \le \beta$ into three regions, $\alpha \le \beta < 1$, $\alpha < 1 \le \beta$, and
$1 \le \alpha \le \beta$, and use in each a form that is exact on the region's edges.  Since the edges are the
lines $\alpha = 1$, $\beta = 1$, and $\alpha = \beta$, where the median is known exactly, the pieces join
continuously, though not smoothly:  the partial derivatives jump across the lines, by up to about $1\%$ of their
values across $\beta = 1$, $0.4\%$ across $\alpha = 1$, and much less across the diagonal.  For $\alpha > \beta$, we use the symmetry $\nu_{\textrm{B}10}(\alpha,\beta) =
1 - \nu_{\textrm{B}10}(\beta,\alpha)$.

\paragraph{The two low-$\alpha$ regions.}
For $\alpha < 1$, we use the odds-corrected form of Section~\ref{sec:odds}, but steer it for accuracy rather than
for a bound:  the coefficients $0.22$ and $0.2$ of $\nu_{\textrm{B}8}$ and $\nu_{\textrm{B}9}$ become
functions $\kappa_1(\beta)$ and $\kappa_2(\alpha)$, with two constants each, fitted to minimize the worst
relative error.  Any $\kappa$ keeps the form exact on the three lines, so $\kappa$ plays the role that $c$ plays
in the ratio forms.  (The steering factor for $\alpha \le 1 < \beta$ uses $\beta/(\beta+1)$, where
$\nu_{\textrm{B}9}$ uses $\beta/(\beta+2)$; each was chosen for its own purpose, the best fit here and a bound
with some margin there.)

\paragraph{The high-$\alpha$ region.}
For $1 \le \alpha \le \beta$, we use the interpolated approximation $\nu_\textrm{I}$ of
Section~\ref{sec:interp}, Equation~\eqref{eq:nuI}, which is exact on the region's edges $\alpha = 1$ and
$\alpha = \beta$ and has no fitted constant.  The edge-exact $\nu_\textrm{E}$ of Section~\ref{sec:asymexp} is
about as accurate in this region, but $\nu_\textrm{I}$ is simpler and has no removable $0/0$ at $\alpha = 1$.

\paragraph{The combined approximation.}
In terms of the pieces, the combined approximation is, for $\alpha \le \beta$,
\begin{align*}
\nu_{\textrm{B}10}(\alpha,\beta) &= \frac{v}{(1+v)^{e_\textrm{D}(\beta)\,\left(1 - \kappa_1(\beta)\,(1-\beta)(1 - \alpha/\beta)\right)}}
&& \textrm{for } \alpha \le \beta < 1 \, , \\
\nu_{\textrm{B}10}(\alpha,\beta) &= \frac{v}{(1+v)^{e_1(\beta)\,\left(1 + \kappa_2(\alpha)\,(1-\alpha)\,\beta/(\beta+1)\right)}}
&& \textrm{for } \alpha < 1 \le \beta \, , \\
\nu_{\textrm{B}10}(\alpha,\beta) &= \nu_\textrm{I}(\alpha,\beta)
&& \textrm{for } 1 \le \alpha \le \beta \, ,
\end{align*}
with $\kappa_1(\beta) = 0.173 - 0.568\,\beta^2$, $\kappa_2(\alpha) = 0.328 - 0.131\,\alpha$, $v = P\,G$, and
the exponents $e_\textrm{D}$ and $e_1$ of Section~\ref{sec:odds}.

Written out in full, with $P$, $G$, $e_\textrm{D}$, $e_1$, $\kappa_1$, $\kappa_2$, and $\nu_\textrm{I}$
expanded in place, the combined approximation is
\begin{equation}
\nu_{\textrm{B}10}(\alpha,\beta) =
\begin{cases}
\begin{aligned}[t]
&v \exp\biggl(-\dfrac{\log\left(2v_\beta\right)\log\left(1+v\right)}{\log\left(1+v_\beta\right)} \\
&\qquad\quad \times \left(1 - \left(0.173 - 0.568\,\beta^2\right)(1-\beta)\left(1-\dfrac{\alpha}{\beta}\right)\right)\biggr) ,
\end{aligned}
& \alpha \le \beta < 1 \, , \\[44pt]
\begin{aligned}[t]
&v \exp\biggl(\dfrac{\log\left(2\beta\left(1-2^{-1/\beta}\right)\right)\log\left(1+v\right)}{\log\left(1+\frac{1}{2\beta}\right)} \\
&\qquad\quad \times \left(1 + \left(0.328 - 0.131\,\alpha\right)(1-\alpha)\dfrac{\beta}{\beta+1}\right)\biggr) ,
\end{aligned}
& \alpha < 1 \le \beta \, , \\[40pt]
\dfrac{\alpha-1}{\alpha} \cdot \dfrac{\alpha-\frac{1}{3}+\frac{1}{30\alpha}}{\alpha+\beta-\frac{2}{3}+\frac{1}{30\alpha}+\frac{1}{30\beta}}
+ \dfrac{2^{1/\beta}-1}{\alpha\left(2^{1/\alpha}+2^{1/\beta}-2\right)} \, , & 1 \le \alpha \le \beta \, , \\[16pt]
1 - \nu_{\textrm{B}10}(\beta,\alpha) \, , & \alpha > \beta \, ,
\end{cases}
\label{eq:B10full}
\end{equation}
where $v = v(\alpha,\beta)$ and $v_\beta = v(\beta,\beta) = G(\beta,\beta)$, since $P(\beta,\beta) = 1$, with
\begin{align*}
v(\alpha,\beta) &= \left(\frac{\alpha+\beta}{2\beta(1+\alpha)}\right)^{1/\alpha}
\left(\frac{2}{1+\beta}\right)^{\alpha^2} \left(\frac{3}{2+\beta}\right)^{1-\alpha^2} \\*
&\qquad \times \exp\left((1-\alpha^2)\,\sigma(\beta) + \alpha(1-\alpha)\,\tau(\beta)\right) , \\*[6pt]
\sigma(\beta) &= \frac{1}{\beta+1} + \frac{1}{2(\beta+2)} + \frac{1}{12(\beta+2)^2} - \frac{73}{108} \, , \\*[6pt]
\tau(\beta) &= \frac{9403}{29160} - \frac{1}{2(\beta+1)^2} - \frac{1}{2(\beta+2)} - \frac{1}{4(\beta+2)^2}
- \frac{1}{12(\beta+2)^3} + \frac{1}{60(\beta+2)^5} \, .
\end{align*}
Here $v = P\,G$, with $\sigma$ and $\tau$ repeated from Equation~\eqref{eq:Gdef} so that
Equation~\eqref{eq:B10full} stands alone.  At $\beta = 1$, both exponentials are $1$, and $v(\alpha,1) = \left(\tfrac{1}{2}\right)^{1/\alpha}$.

\begin{figure}[!ht]
  \centering
  \includegraphics[width=0.8\linewidth]{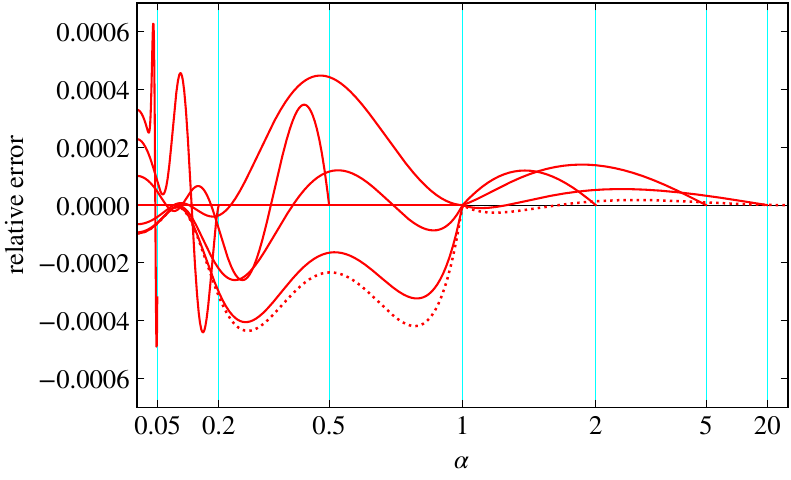}
    \caption{{\bf Relative error of the combined approximation \texorpdfstring{$\nu_{\textrm{B}10}$}{nu B10}.}
    For each $\beta$ marked on the axis, the error is plotted for $\alpha \le \beta$ on an
    $\arctan(\alpha)$ axis as in Figure~\ref{fig:rel3}; each curve is $0$ at $\alpha = 1$ and at $\alpha = \beta$,
    and very small for $\alpha > 1$.  The dotted curve is the limit $\beta \to \infty$:  the
    relative error of $\nu_{\Gamma 5}(\alpha) = \lim_{\beta\to\infty} \beta\,\nu_{\textrm{B}10}(\alpha,\beta)$ as an
    approximation to the gamma median $\nu_\Gamma(\alpha)$, for all $\alpha$ (Section~\ref{sec:gamma}), which
    is $0$ at $\alpha = 1$.  The spike at the end of the $\beta = 0.05$ curve is the abrupt approach to
    $\tfrac{1}{2}$ near the diagonal.}
  \label{fig:relwhole}
\end{figure}

\begin{figure}[!ht]
  \centering
  \includegraphics[width=0.58\linewidth]{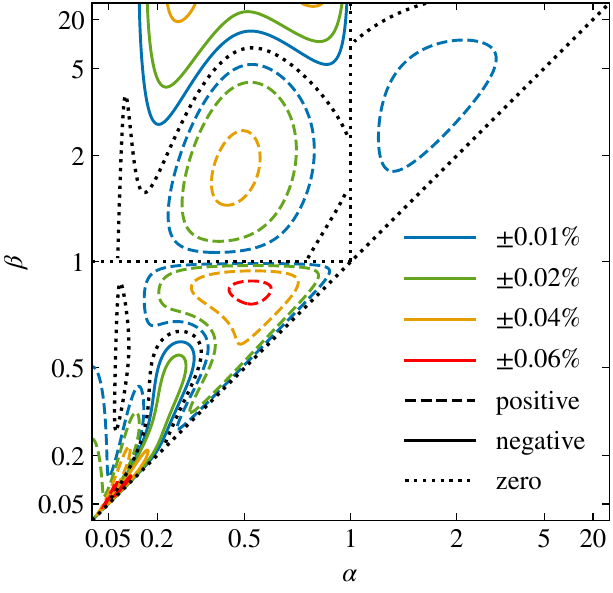}
    \caption{{\bf Relative error contours of \texorpdfstring{$\nu_{\textrm{B}10}$}{nu B10}.}
    Contours at $\pm 0.01\%$, $\pm 0.02\%$, $\pm 0.04\%$, and $\pm 0.06\%$, colored as in the legend,
    dashed for positive and solid for negative errors, and dotted where the error is 0, as in
    Figure~\ref{fig:contour1}.  Errors are shown for $\alpha \le \beta$ only:  since
    $\nu_{\textrm{B}10}(\beta,\alpha) = 1 - \nu_{\textrm{B}10}(\alpha,\beta)$, a point with $\alpha > \beta$
    has its mirror point's absolute error but, with a larger median, a smaller relative error.  The straight dotted lines are the region edges, where
    $\nu_{\textrm{B}10}$ is exact; the curved ones are sign changes within the regions.}
  \label{fig:contourwhole}
\end{figure}

Table~\ref{tab:kappa} shows how the worst error falls with the number of fitted constants in $\kappa$.
Splitting $n$ constants between the two low-$\alpha$ regions as well as possible, the worst error over the
whole domain is $1.74\%$, $0.70\%$, $0.23\%$, $0.14\%$, and $0.065\%$ for $n = 0$ to $4$; $\nu_{\textrm{B}10}$
is the member with four.

\begin{table}[!ht]
\centering
\caption{Worst relative error of the steered odds-corrected form by the number of fitted constants in
$\kappa$, on the grids of Section~\ref{sec:odds}.}
\label{tab:kappa}
\small
\begin{tabular}{@{}clrlr@{}}
\toprule
Constants & $\kappa_1(\beta)$, for $\alpha \le \beta < 1$ & Error & $\kappa_2(\alpha)$, for $\alpha \le 1 < \beta$ & Error \\
\midrule
0 & $0$ & $0.70\%$ & $0$ & $1.74\%$ \\
1 & $0.119$ & $0.23\%$ & $0.259$ & $0.14\%$ \\
2 & $0.173 - 0.568\,\beta^2$ & $0.065\%$ & $0.328 - 0.131\,\alpha$ & $0.046\%$ \\
\bottomrule
\end{tabular}
\end{table}

Over the same grids as before, from $\beta = 0.001$ to $10^6$, the relative error of $\nu_{\textrm{B}10}$ is
between $-0.064\%$ and $+0.065\%$ for $\alpha \le \beta < 1$, between $-0.044\%$ and $+0.046\%$ for
$\alpha \le 1 < \beta$, and between $-0.003\%$ and $+0.016\%$ for $1 \le \alpha \le \beta$
(Figures~\ref{fig:relwhole} and~\ref{fig:contourwhole}).  As is typical of a minimax fit, the largest errors, of about $\pm 0.065\%$,
occur at several places in the region $\alpha \le \beta < 1$.  For smaller $\beta$, the error stays
between $0$ and $+0.037\%$.  As both shape parameters tend to $0$ with $1/\alpha - 1/\beta$ fixed, $\alpha/\beta$
tends to $1$, so the steering factor tends to $1$ and the exponent to $e_\textrm{D}(0) \approx 1.0003$; the logit
of the median and $\log P$ both tend to $(1/\beta - 1/\alpha)/2$; and $G$ tends to
$G(0,0) = \tfrac{3}{2}\,e^{\sigma(0) - 1} \approx 1.00037$.  The error is $0$ on the diagonal and approaches
$G(0,0) - 1 \approx +0.037\%$ as $1/\alpha - 1/\beta$ grows and $v$ tends to $0$.  So $\nu_{\textrm{B}10}$ is within $0.065\%$ of the median over the whole domain.

\section{New gamma median approximations}\label{sec:gamma}

In this section, we examine how the study of the median of the beta distribution contributes back to the problem of finding good bounds and approximations to the median of the gamma distribution.

Every ratio form has a corresponding gamma median approximation.  If $X \sim \textrm{Beta}(\alpha,\beta)$,
then $\beta X$ converges in distribution to $\textrm{Gamma}(\alpha,1)$ as $\beta \to \infty$, so
$\beta\,\nu_\textrm{B}(\alpha,\beta) \to \nu_\Gamma(\alpha)$; and if $g(x)/x \to 1$ as $x \to \infty$, then
$\beta\,g(\alpha)/(g(\alpha)+g(\beta)) \to g(\alpha)$.  A ratio form that is accurate for large $\beta$ therefore
comes with $g$ itself as an approximation to the gamma median.  We have already met an especially good one:
$\nu_{\Gamma 3}(x) = (x - \tfrac{1}{6})^2/x$, the $g$ of $\nu_{\textrm{B}3}$, overestimates the gamma median by
at most $0.212\%$ for $x \ge 1$.  But it falls to zero at $x = \tfrac{1}{6}$ and grows
without bound as $x \to 0$, while the true median goes to zero.  We are more interested here in an
approximation that has the right behavior over the whole range $0 < \alpha < \infty$, and that can serve as
the basis for still better approximations.

The beta median approximation $\nu_{\textrm{B}4}(\alpha,\beta)$, based on Lemma~\ref{lem:special}'s constraint on $g$ that makes it accurate at $\beta = 1$, combined with the known exact value of the median of the gamma distribution at $\alpha = 1$ (where it is the exponential distribution), $\nu_{\Gamma}(1) = \log 2$, suggests this new approximation for the median of the gamma distribution:
\[
\nu_{\Gamma 4}(\alpha)= \frac{\log 2}{2^{1/\alpha} - 1} \, ,
\]
which is consistent with writing $\nu_{\textrm{B}4}$ in terms of this gamma median approximation as
\[
\nu_{\textrm{B}4}(\alpha,\beta) = \frac{\nu_{\Gamma 4}(\alpha)}{\nu_{\Gamma 4}(\alpha) + \nu_{\Gamma 4}(\beta)} \, .
\]

Notice that the numerator value $\log 2$, besides making the approximation exact at $\alpha = 1$, is also the unique value that makes the slope approach 1 at high $\alpha$, like the slope of the true median, since $2^{1/\alpha}$ approaches $1 + \tfrac{\log 2}{\alpha}$.
Graphically, we observe that $\nu_{\Gamma 4}(\alpha)$ crosses the true value only once, at $\alpha = 1$.  Theorem~\ref{thm:gamma} proves that this gamma median approximation is a lower bound on the standard gamma distribution median $\nu_{\Gamma}(\alpha)$ where $\alpha > 1$, and an upper bound where $\alpha < 1$:
\[
\nu_{\Gamma 4}(\alpha) > \nu_{\Gamma}(\alpha) \quad \textrm{for} \quad 0 < \alpha < 1 \, ,
\quad \textrm{and} \quad \nu_{\Gamma 4}(\alpha) < \nu_{\Gamma}(\alpha) \quad \textrm{for} \; \alpha > 1 \, .
\]

At both ends, and at $\alpha = 1$, $\nu_{\Gamma 4}$ compares with the true median as follows.
At low $\alpha$, it is about 23\% high:
\[
\frac{\lim_{\alpha \to 0} 2^{1/\alpha} \nu_{\Gamma 4}(\alpha)}{\lim_{\alpha \to 0} 2^{1/\alpha} \nu_{\Gamma}(\alpha)} = \frac{\log 2}{e^{-\gamma}} \approx \frac{0.693147}{0.561459}\approx 1.2345 \, .
\]
At high $\alpha$, $\nu_{\Gamma 4}$ is asymptotic to $\alpha - \tfrac{\log 2}{2}$, about $0.0132403$ below the true $\alpha - \frac{1}{3}$ asymptote:
\[
\lim_{\alpha \to \infty} \left(\nu_{\Gamma 4}(\alpha) - \nu_{\Gamma}(\alpha) \right)=  -\frac{\log 2}{2} + \frac{1}{3} \approx -0.0132403 \,.
\]
At $\alpha = 1$, where the value is exact, the slope is just slightly below the correct slope \citep{lyon2023tight}:
\[
\frac{\nu_{\Gamma 4}^{\prime}(1)}{\nu_{\Gamma}^{\prime}(1)} = \frac{2 (\log 2)^2}{\gamma - 2 \textrm{Ei}(-\log 2) + \log \log 2} \approx \frac{0.9609}{0.9680} \approx 0.9926 \, .
\]

So $\nu_{\Gamma 4}$ is exact at $\alpha = 1$, is a bound on each side of it, and has the right form of
behavior at both ends:  a constant ratio to the true median as $\alpha \to 0$ and a constant offset from it
as $\alpha \to \infty$.  Its relative error is largest at low $\alpha$, and its absolute error is largest at
high $\alpha$.  That makes it a good base on which to build more accurate approximations.  A companion paper
\citep{lyon2026gamma} develops interpolated approximations built on such forms, first with every constant
set analytically from the median's known behavior at the two ends and at $\alpha = 1$, and then with fitted
constants that minimize the maximum of the absolute and relative errors, so that they work well at both low
and high $\alpha$.

The approximation $\nu_{\textrm{B}10}$ of Section~\ref{sec:whole} also yields a gamma median approximation,
$\nu_{\Gamma 5}(\alpha) = \lim_{\beta\to\infty} \beta\,\nu_{\textrm{B}10}(\alpha,\beta)$.  It has two pieces:
\[
\nu_{\Gamma 5}(\alpha) =
\begin{cases}
V(\alpha)\,(2\log 2)^{2V(\alpha)\left(1 + \kappa_2(\alpha)(1-\alpha)\right)} \, , & 0 < \alpha < 1 \, , \\[4pt]
\dfrac{\alpha - 1}{\alpha}\left(\alpha - \dfrac{1}{3} + \dfrac{1}{30\alpha}\right) + \dfrac{1}{\alpha}\,\nu_{\Gamma 4}(\alpha) \, ,
& \alpha \ge 1 \, ,
\end{cases}
\]
with $\kappa_2(\alpha) = 0.328 - 0.131\,\alpha$ from Section~\ref{sec:whole} and
\[
V(\alpha) = \left(2(1+\alpha)\right)^{-1/\alpha} \exp\left((1-\alpha)\,a + \alpha\log 2 + \alpha(1-\alpha)\,b\right) \, ,
\]
where $a = \log 3 - \tfrac{73}{108} \approx 0.422686$ and
$b = \log\tfrac{3}{2} - \tfrac{10307}{29160} \approx 0.052001$ come from the constants in $\sigma$ and $\tau$
(Section~\ref{sec:odds}).  The piece for $\alpha \ge 1$ is the limit of $\nu_\textrm{I}$, the interpolation between $\nu_{1/30}$
and $\nu_{\textrm{B}4}$, since each ratio form tends to its $g$.  For $\alpha < 1$, the product $\beta v = \beta P G$ tends
to $V(\alpha)$, and $e_1(\beta) \sim -2\beta\log(2\log 2)$, so the odds factor $(1+v)^{-e}$ tends to
$(2\log 2)^{2V(1 + \kappa_2(1-\alpha))}$.

At $\alpha = 1$, $V$ is $\tfrac{1}{2}$, so both pieces equal $\log 2$, and $\nu_{\Gamma 5}$ is continuous and
exact there, though its slope jumps from about $0.9713$ to $0.9678$ across the true $0.9680$.  As
$\alpha \to 0$, $\nu_{\Gamma 5}(\alpha) \sim 2^{-1/\alpha} e^{a-1}$; since $a$ is close to $1 - \gamma$,
its ratio to the true median tends to $e^{a - 1 + \gamma} \approx 0.99990$, against $1.2345$ for
$\nu_{\Gamma 4}$.  As $\alpha \to \infty$, its absolute error vanishes, like
$\left(\tfrac{11}{30} - \tfrac{\log 2}{2} - \tfrac{8}{405}\right)/\alpha \approx 0.00034/\alpha$, by the
expansion $\nu_\Gamma(\alpha) = \alpha - \tfrac{1}{3} + \tfrac{8}{405\alpha} + O(\alpha^{-2})$ (Section~\ref{sec:asymexp}).  Its relative
error is between $-0.044\%$, near $\alpha = 0.27$, and $+0.002\%$, near $\alpha = 2.6$ (the dotted curve in
Figure~\ref{fig:relwhole}), and its absolute error is between $-0.00022$ and $+0.00005$, so the maximum of
its absolute and relative errors is $0.00044$ over all $\alpha$.

The constants $a$ and $b$ come from the beta median construction, and approximate
$1 - \gamma \approx 0.422784$ and $\pi^2/12 + \tfrac{1}{2} - \log 2 - \gamma \approx 0.052104$, the values
that match the true median's expansion at $\alpha = 0$ to first order.  With those exact values, and with
$\kappa(\alpha) = 0.33530 - 0.13485\,\alpha$ refitted for the gamma median, the maximum of the absolute and
relative errors falls to $0.00020$.

\section{Summary of contributions}

This article contributes the following potentially useful new results:

\begin{enumerate}
\item A rigorously proved tight lower bound and asymptote to the median of the beta distribution
in terms of the median of the gamma distribution (Theorems~\ref{thm:asymptote}
and~\ref{thm:lower_bound}):
\begin{itemize}
\item $\nu_\textrm{B}(\alpha,\beta) \ge \nu_{\textrm{B}0}(\alpha,\beta) =
\frac{\nu_\Gamma(\alpha)}{\nu_\Gamma(\alpha)+\nu_\Gamma(\beta)}$ for all $0 < \alpha \le \beta$
\item $\lim\limits_{\beta\to\infty} \frac{\nu_{\textrm{B}0}(\alpha,\beta)}{\nu_\textrm{B}(\alpha,\beta)} = 1$
for all $\alpha > 0$
\end{itemize}
\item Elementary-function approximations to the median of the beta distribution, of
the form $g(\alpha)/(g(\alpha)+g(\beta))$, generalizing Kerman's approximation, where $g$ is an
approximation to the median of the gamma distribution:
\begin{itemize}
\item $\nu_{\textrm{B}2}(\alpha,\beta) = \frac{\alpha(1-\frac{1}{9\alpha})^3}
{\alpha(1-\frac{1}{9\alpha})^3 + \beta(1-\frac{1}{9\beta})^3}$ based on the Wilson--Hilferty
approximation, implicit in \citet{paulson1942approximate} but little known as an approximation to the
beta median
\item $\nu_{\textrm{B}3}(\alpha,\beta) = \frac{(\alpha - \frac{1}{6})^2/\alpha}
{(\alpha - \frac{1}{6})^2/\alpha + (\beta - \frac{1}{6})^2/\beta}$ based on the gamma median approximation $\nu_{\Gamma3}$, the member
$c = \tfrac{1}{36}$ of the family $g(x) = x - \tfrac{1}{3} + c/x$ (called $\nu_{1/36}$ in
\citet{lyon2026kerman}), with relative error numerically between $-0.171\%$ and $+0.212\%$ for $1 \le \alpha \le \beta$
\item $\nu_{\textrm{B}4}(\alpha,\beta) = \frac{2^{1/\beta}-1}{2^{1/\alpha}+2^{1/\beta}-2}$ giving rise to
the new gamma median approximation $\nu_{\Gamma4}$, exact when $\alpha = 1$, $\beta = 1$, or $\alpha = \beta$,
numerically within $0.30\%$ of the median when both shape parameters are at least 1, and proved to exceed it by
less than $24.6\%$ for $\alpha \le \beta$ (Theorem~\ref{thm:B4over})
\end{itemize}
\item Approximations to the median of the gamma distribution, as limits of beta median approximations
for large $\beta$ (Section~\ref{sec:gamma}):
\begin{itemize}
\item $\nu_{\Gamma3}(\alpha) = \frac{(\alpha - \frac{1}{6})^2}{\alpha}$, which numerically overestimates the
gamma median by at most $0.212\%$ for $\alpha \ge 1$
\item $\nu_{\Gamma4}(\alpha) = \frac{\log 2}{2^{1/\alpha} - 1}$, exact at $\alpha = 1$,
rigorously proved to provide tight lower and upper bounds above and below $\alpha = 1$,
respectively (Theorem~\ref{thm:gamma}), with the right form of behavior at both ends
\item $\nu_{\Gamma5}(\alpha) = \lim\limits_{\beta\to\infty} \beta\,\nu_{\textrm{B}10}(\alpha,\beta)$, a two-piece
elementary-function approximation, exact at $\alpha = 1$ and numerically within $0.044\%$ of the gamma median for all
$\alpha > 0$
\end{itemize}
\item Proofs of the bounding behavior of the Kerman and Wilson--Hilferty forms:  $\nu_{\textrm{B}1}$
is a lower bound for $\tfrac{1}{3} < \alpha < \beta$ (Theorem~\ref{thm:kerman}), and $\nu_{\textrm{B}2}$ is an upper
bound on the line $\alpha = 1$ (Theorem~\ref{thm:wh_edge}) and along every ray $\beta = r\alpha$ with $r > 1$
for large enough $\alpha$ (Theorem~\ref{thm:whray}).  We conjecture that it is an upper bound for all
$1 \le \alpha < \beta$ (Conjecture~\ref{conj:wh}).
\item An explicit second-order asymptotic expansion of the logit of the beta median for large shape
parameters (Theorem~\ref{thm:largeexp}), also given in \citet{lyon2026kerman}.  It gives the second-order errors of all these
approximations, and shows that no form $g(\alpha)/(g(\alpha)+g(\beta))$ is accurate to second order
for all large shape parameters.
\item An edge-exact approximation $\nu_\textrm{E}$, of the form $g_c(\alpha)/(g_c(\alpha)+g_c(\beta))$ with
$g_c(x) = x - \tfrac{1}{3} + c/x$, where $c = c_\textrm{E}(\alpha,\beta)$ depends on both shape parameters and is
built from its exact values on the lines $\alpha = 1$ and $\beta = 1$ and its large-parameter asymptote.  Like
$\nu_{\textrm{B}4}$, it is exact on the lines $\alpha = 1$, $\beta = 1$, and $\alpha = \beta$, but it is much more
accurate where both shape parameters exceed 1:  numerically within $0.016\%$ of the median for $1 \le \alpha \le \beta$, with
no fitted constant, against $0.30\%$ for $\nu_{\textrm{B}4}$ (Section~\ref{sec:asymexp}).
\item A proof that no function $g$ of one variable used in the form $g(\alpha)/(g(\alpha)+g(\beta))$ will yield the
beta distribution median exactly (Theorem~\ref{thm:non}).
\item An interpolated approximation $\nu_\textrm{I} = \frac{\alpha - 1}{\alpha}\,\nu_{1/30} + \frac{1}{\alpha}\,\nu_{\textrm{B}4}$
between two ratio forms that bracket the median where both shape parameters are at least 1, exact on the lines
$\alpha = 1$ and $\alpha = \beta$, with no fitted constant:  numerically within $0.016\%$ of the median for
$1 \le \alpha \le \beta$, as accurate as $\nu_\textrm{E}$ and much simpler (Section~\ref{sec:interp}).
\item An analysis and partial proof of the regional over- and under-estimation of the
$\nu_{\textrm{B}4}(\alpha,\beta)$ approximation (Conjecture~\ref{thm:B4bound}), proved for large shape
parameters (Theorem~\ref{thm:B4large}) and wherever Lemma~\ref{lem:B4partial} applies.
\item A low-$\alpha$ asymptote $\nu_{\textrm{B}5}(\alpha,\beta) = \left(\alpha B(\alpha,\beta)/2\right)^{1/\alpha}$
of the beta median, proved to be a lower bound for $\beta > 1$ and an upper bound for $\beta < 1$
(Theorems~\ref{thm:low} and~\ref{thm:lowasym}).  Written as $\nu_{\textrm{B}5} = P\,F$, with $P$
elementary, the factor $F$ can be replaced to give provable elementary-function bounds:
$\nu_{\textrm{B}6}$, a lower bound for $\beta > 1$, and $\nu_{\textrm{B}7}$, an upper bound for $\beta < 1$
(Lemma~\ref{lem:logconc} and Theorem~\ref{thm:lowbounds}).
\item Odds-corrected elementary-function bounds, exact on the lines $\alpha = 1$, $\beta = 1$, and $\alpha = \beta$, and
built on an approximation to $F$ that is exact on the first two:
an upper bound $\nu_{\textrm{B}8}$ for $\alpha \le \beta < 1$ and a lower bound $\nu_{\textrm{B}9}$ for
$\alpha \le 1 < \beta$, within $0.41\%$ and $0.44\%$ of the median, verified numerically
(Section~\ref{sec:odds}).
\item An elementary-function approximation $\nu_{\textrm{B}10}$ over the whole domain, with four fitted
constants that steer the odds-corrected form for $\alpha < 1$, and $\nu_\textrm{I}$ for $\alpha \ge 1$ (for $\alpha \le \beta$, and by symmetry
for $\alpha > \beta$).  It is exact on the lines
$\alpha = 1$, $\beta = 1$, and $\alpha = \beta$, and numerically within $0.065\%$ of the median over the whole domain
(Section~\ref{sec:whole}).
\item A Lean~4 formalization that machine-checks most of the theorems above
(Section~\ref{app:lean} of the appendix).
\end{enumerate}

Several of these results are developed with more focus in papers spun out during this work.
\citet{lyon2026kerman} treats the family $g(x) = x - \tfrac{1}{3} + c/x$, its asymptotic bounds for
large shape parameters (Theorem~\ref{thm:cfamily}), the large-shape expansion, and $\nu_{\textrm{B}3}$
and $\nu_{\Gamma 3}$ in detail, and \citet{lyon2026gamma} builds analytic and optimized approximations to the
gamma median, including forms based on $\nu_{\Gamma 4}$, with constants set analytically or fitted.

\section{Open questions}

The main open questions are the three conjectures of the appendix:  that the Wilson--Hilferty form
$\nu_{\textrm{B}2}$ is an upper bound for $1 \le \alpha < \beta$ (Conjecture~\ref{conj:wh}), the uniform
bound on the remainder of the large-shape expansion that would imply it (Conjecture~\ref{conj:largeR}),
and the sign pattern of the errors of $\nu_{\textrm{B}4}$ (Conjecture~\ref{thm:B4bound}), which would
follow from a triangle inequality for the median odds (Lemma~\ref{lem:triangle}).  A proof of the last
would also confirm that $\nu_{\textrm{B}4}$ and $\nu_{1/30}$ bracket the median for $1 \le \alpha \le \beta$
(Section~\ref{sec:interp}).  The bounds $\nu_{\textrm{B}8}$ and $\nu_{\textrm{B}9}$ are verified only
numerically, and proofs of them would need bounds on the steered odds correction.

For approximation, the region $\alpha \le \beta < 1$ remains the hardest, because there the median moves
from near $0$ to $\tfrac{1}{2}$ over a range of $\alpha$ near $\beta$ that narrows as $\beta$ decreases.
More constants in the steering function $\kappa_1$ help only slowly there:  a third,
$\kappa_1(\beta) = 0.167 + 0.087\,\beta - 0.795\,\beta^2$, lowers the worst error only from $0.065\%$ to
$0.064\%$.  Other forms of the exponent may do better.

\section{Conclusions}

A simple formula that approximates the median of the beta distribution well over its whole domain
was the goal that started this work.  We reached it only in part.  The approximation $\nu_{\textrm{B}10}$,
written out in full in Equation~\eqref{eq:B10full} and built from the low-$\alpha$ asymptote and the ratio forms, is exact
on the lines $\alpha = 1$, $\beta = 1$, and $\alpha = \beta$, like $\nu_{\textrm{B}4}$, and within $0.065\%$ of
the median everywhere, but it is not simple:  it has three pieces, two of them odds-corrected forms with
their own exponent functions, and four fitted constants.  The simplest formula we found that is
usable over the whole domain is $\nu_{\textrm{B}4}$ itself, which stays between 0 and 1 everywhere and is within
$0.30\%$ of the median when both shape parameters are at least 1, but falls far below it where both are
small.  Where both are at least 1, the interpolated $\nu_\textrm{I}$ comes much closer to the goal:  a few
rational terms and powers of 2, with no fitted constant, within $0.016\%$ of the median.  Along the way, we have shown how to make
excellent approximations for high $\beta$, for $\beta$ near 1, and for low $\alpha$.
While standard numerical methods, such as piecewise rational (minimax or Padé)
approximations, will be preferred in a computational context, the asymptotes, bounds, and
approximations presented herein may be of interest in shedding some light on the
behavior of the beta distribution median and its relationship to the gamma
distribution median.

\begin{appendix}

\section{Theorems, proofs, and formal verification}
\label{app:main}

The author originally proved several of the bound theorems by hand, and then used several AI systems to check, simplify, and improve the proofs, and to prove more bounds, over a period of several months.  Theorem~\ref{thm:lower_bound} was particularly resistant to proof, until Claude (Anthropic's AI desktop app) discovered the quantile pairing approach, which settled it.  As Claude describes it, ``The proof pairs the tails of one gamma variable at the quantiles of the other, a variant of the classical symmetric-quantile device \citep{vanzwet1979mean}.  Earlier attempts paired at points symmetric about 1, which is the classical form.''  This proof approach may be useful elsewhere in relation to medians.
Asked about novelty, Google DeepMind's AI co-mathematician (Gemini) replied:  ``To my knowledge,
while pairing symmetric quantiles of a single distribution is a standard trick (van Zwet), directly
using the symmetric quantiles of one variable as the integration support to probe the CDF of an
independent second variable to bound their ratio's median is indeed highly non-standard and
represents a genuinely novel structural approach.''  Neither system's literature knowledge is
exhaustive, so we offer this as an observation rather than a claim of priority.
Gemini also proposed a major simplification to Claude's proof of Lemma~\ref{lem:Lambda}, which Claude agreed was better and adopted.
The collaborative behavior of different AI systems was genuinely productive throughout the project.

The rest of this appendix is mostly as Claude drafted it.  The author has carefully reviewed it and made light edits, and has run the Lean~4 compilations to verify the formalizations.

This appendix collects the formal statements made in the paper, with proofs.  It covers the
high-$\beta$ asymptote, the lower bound $\nu_{\textrm{B}0} \le \nu_\textrm{B}$, the gamma median bounds,
the special case and non-existence results, the Kerman and Wilson--Hilferty bounds, large-shape asymptotics, the status
of the $\nu_{\textrm{B}4}$ conjecture, and the low-$\alpha$ asymptote with elementary bounds derived from it.  It also states three conjectures with the evidence for them.
Section~\ref{app:lean} describes the Lean~4 formalization, which machine-checks every theorem and
lemma here, with the exceptions named there.

\subsection{Notation and basic facts}
\label{app:notation}

For shape parameter $\alpha > 0$, $X_\alpha$ denotes a standard gamma variable, $\textrm{Gamma}(\alpha,1)$, and
variables with different subscripts are independent.  The gamma median is $\nu_\Gamma(\alpha)$, and
$\nu_\textrm{B}(\alpha,\beta)$ is the median of $\textrm{Beta}(\alpha,\beta)$.  Both distributions have
continuous, strictly increasing CDFs on their supports, so the medians are unique.  The link used
throughout is that
\[
\frac{X_\alpha}{X_\alpha + X_\beta} \sim \textrm{Beta}(\alpha,\beta) \, ,
\]
so for $s > 0$ the event $\{X_\alpha \le s X_\beta\}$ has the probability of the beta CDF at $s/(1+s)$.
In particular, the median odds
\[
m(\alpha,\beta) = \frac{\nu_\textrm{B}(\alpha,\beta)}{1 - \nu_\textrm{B}(\alpha,\beta)}
\]
is the median of the ratio $X_\alpha/X_\beta$.  We write $\operatorname{logit} p = \log(p/(1-p))$.
Each ratio-form approximation $\nu_{\textrm{B}k} = g(\alpha)/(g(\alpha)+g(\beta))$ has odds
$g(\alpha)/g(\beta)$.  Since $t \mapsto t/(1+t)$ is strictly increasing, comparing two such medians is
the same as comparing their odds.

The approximations used in this appendix are the following.  The ratio forms are
$\nu_{\textrm{B}0}$, with $g = \nu_\Gamma$; Kerman's $\nu_{\textrm{B}1}$, with $g(x) = x - \tfrac{1}{3}$;
the Wilson--Hilferty form $\nu_{\textrm{B}2}$, implicit in \citet{paulson1942approximate}, with $g(x) = \nu_{\Gamma2}(x) = x(1 - \tfrac{1}{9x})^3$;
$\nu_{\textrm{B}3}$, with $g(x) = \nu_{\Gamma3}(x) = (x - \tfrac{1}{6})^2/x$;
and $\nu_{\textrm{B}4}$, with $g(x) = \nu_{\Gamma4}(x) = \log 2/(2^{1/x} - 1)$.  The low-$\alpha$ asymptote is
\[
\nu_{\textrm{B}5}(\alpha,\beta) = \left(\frac{\alpha B(\alpha,\beta)}{2}\right)^{1/\alpha} ,
\]
and the elementary bounds $\nu_{\textrm{B}6}$ and $\nu_{\textrm{B}7}$ are defined in
Section~\ref{app:lowalpha}.  Here $\gamma$ is the Euler--Mascheroni constant and $\psi_n$ the polygamma
functions.  The odds-corrected bounds $\nu_{\textrm{B}8}$ and $\nu_{\textrm{B}9}$ and the combined
approximation $\nu_{\textrm{B}10}$ of the body are numerical results, not used here.

\begin{Lemma}[Basic facts]
\label{lem:basic}
For $\alpha, \beta > 0$:
$\nu_\textrm{B}(\beta,\alpha) = 1 - \nu_\textrm{B}(\alpha,\beta)$; $\nu_\textrm{B}(\alpha,\alpha) = \tfrac{1}{2}$;
$\nu_\textrm{B}(\alpha,1) = 2^{-1/\alpha}$; $\nu_\textrm{B}(1,\beta) = 1 - 2^{-1/\beta}$; and
$\nu_\Gamma(1) = \log 2$.
\end{Lemma}

\begin{proof}
If $X \sim \textrm{Beta}(\alpha,\beta)$, then $1 - X \sim \textrm{Beta}(\beta,\alpha)$, which gives the
first identity.  With $\beta = \alpha$, uniqueness of the median gives the second.  The
$\textrm{Beta}(\alpha,1)$ CDF is $x^\alpha$, which equals $\tfrac{1}{2}$ at $x = 2^{-1/\alpha}$; the fourth
identity follows from the first.  $\textrm{Gamma}(1,1)$ is exponential, with median $\log 2$.
\end{proof}

\subsection{The high-\texorpdfstring{$\beta$}{beta} asymptote}
\label{app:asymptote}

\begin{Theorem}[High-$\beta$ asymptote]
\label{thm:asymptote}
For every $\alpha > 0$,
\[
\lim_{\beta \to \infty} \frac{\nu_{\textrm{B}0}(\alpha,\beta)}{\nu_\textrm{B}(\alpha,\beta)} = 1 \, ,
\qquad \nu_{\textrm{B}0}(\alpha,\beta) = \frac{\nu_\Gamma(\alpha)}{\nu_\Gamma(\alpha) + \nu_\Gamma(\beta)} \, .
\]
\end{Theorem}

\begin{proof}
Let $Z = X_\alpha/(X_\alpha + X_\beta)$, so the median of $\beta Z$ is $\beta\,\nu_\textrm{B}(\alpha,\beta)$, and
\[
\beta Z = \frac{X_\alpha}{(X_\alpha + X_\beta)/\beta} \, .
\]
The denominator has mean $(\alpha+\beta)/\beta \to 1$ and variance $(\alpha+\beta)/\beta^2 \to 0$, so by
Chebyshev's inequality it converges to $1$ in probability.  By Slutsky's theorem $\beta Z$ converges
in distribution to $X_\alpha$.  The limit CDF $G$ is continuous and strictly increasing on $(0,\infty)$,
so for any $\varepsilon > 0$, $G(\nu_\Gamma(\alpha) - \varepsilon) < \tfrac{1}{2} < G(\nu_\Gamma(\alpha) + \varepsilon)$, and the same
strict inequalities hold for the CDF of $\beta Z$ once $\beta$ is large.  Hence
$\beta\,\nu_\textrm{B}(\alpha,\beta) \to \nu_\Gamma(\alpha)$.

For the approximation, Chebyshev's inequality for $X_\beta$, with mean and variance $\beta$, gives
$\mathbb{P}(|X_\beta - \beta| \ge 2\sqrt\beta) \le \tfrac{1}{4}$.  So
$\nu_\Gamma(\beta) \in [\beta - 2\sqrt\beta,\, \beta + 2\sqrt\beta]$, and $\nu_\Gamma(\beta)/\beta \to 1$.  Then
\[
\beta\,\nu_{\textrm{B}0}(\alpha,\beta) = \frac{\nu_\Gamma(\alpha)\,\beta}{\nu_\Gamma(\alpha) + \nu_\Gamma(\beta)} \to \nu_\Gamma(\alpha) \, ,
\]
and the ratio of the two medians tends to~1.
\end{proof}

The Lean proof is organized differently, to avoid weak convergence.  One side is
$\nu_{\textrm{B}0} < \nu_\textrm{B}$ (Theorem~\ref{thm:lower_bound}).  The other side uses Chernoff
bounds for the gamma tails, valid for $0 < \delta < 1$.  The lower-tail bound
$\mathbb{P}(X_\beta \le \beta(1-\delta)) \le e^{-\beta\delta^2/4}$ keeps $X_\beta$ from being too small, which shows that
eventually $\nu_\textrm{B}(\alpha,\beta) \le (1+\varepsilon)\,\nu_\Gamma(\alpha)/\beta$, and the upper-tail bound
$\mathbb{P}(X_\beta > \beta(1+\delta)) \le e^{-\beta\delta^2/8}$ shows that eventually
$\nu_\Gamma(\beta) \le (1+\varepsilon)\beta$.

\subsection{The lower bound \texorpdfstring{$\nu_{\textrm{B}0} \le \nu_\textrm{B}$}{nu B0 <= nu B}}
\label{app:lower}

\begin{Theorem}[Lower bound]
\label{thm:lower_bound}
For $0 < \alpha \le \beta$,
\[
\nu_{\textrm{B}0}(\alpha,\beta) = \frac{\nu_\Gamma(\alpha)}{\nu_\Gamma(\alpha) + \nu_\Gamma(\beta)} \le \nu_\textrm{B}(\alpha,\beta) \, ,
\]
with equality if and only if $\alpha = \beta$.
\end{Theorem}

\begin{proof}
At $\alpha=\beta$, both sides equal $\tfrac{1}{2}$ (Lemma~\ref{lem:basic}).
So let $0<\alpha<\beta$; we prove the strict inequality.

\emph{Reformulation.}
Let $X = X_\alpha$ and $Y = X_\beta$, with CDFs $F_X,F_Y$ and PDFs $f_X,f_Y$, and let
$Q_Y=F_Y^{-1}:(0,1)\to(0,\infty)$ be the quantile function of $Y$, so that
$Q_Y(\tfrac{1}{2})=\nu_\Gamma(\beta)$.  Put
\[
s=\frac{\nu_\Gamma(\alpha)}{\nu_\Gamma(\beta)} \,.
\]
Since $\{X/(X+Y)\le\nu_{\textrm{B}0}\}=\{X\le sY\}$, and the beta CDF is strictly increasing on $(0,1)$,
the theorem is equivalent to
\begin{equation}\label{eq:t2goal}
\mathbb{P}(X\le sY)<\tfrac{1}{2} \,.
\end{equation}

\emph{Step 1:  pairing the quantiles of $Y$.}
Conditioning on $Y$, and writing $Y=Q_Y(r)$ with $r$ uniform on $(0,1)$,
\[
\mathbb{P}(X\le sY)=\int_0^1 F_X\big(s\,Q_Y(r)\big)\,dr \,.
\]
Pairing $r=\tfrac{1}{2} + t$ with $r=\tfrac{1}{2} - t$ gives
\begin{equation}\label{eq:t2pair}
\mathbb{P}(X\le sY)=\tfrac{1}{2} - \int_0^{1/2}D(t)\,dt \,,
\end{equation}
where
\[
D(t)=1-F_X(s\,x_t)-F_X(s\,y_t) \,,
\]
with $x_t=Q_Y(\tfrac{1}{2} + t)$ and $y_t=Q_Y(\tfrac{1}{2} - t)$, so that $0<y_t<x_t$ for $0<t<\tfrac{1}{2}$.
$D$ is continuous on $[0,\tfrac{1}{2})$ and differentiable on $(0,\tfrac{1}{2})$.
At the left end,
$D(0)=1-2F_X\big(s\,\nu_\Gamma(\beta)\big)=1-2F_X\big(\nu_\Gamma(\alpha)\big)=0$.
At the right end, $x_t\to\infty$ and $y_t\to0$, so $D(t)\to0$ as $t\to\tfrac{1}{2}$.
By \eqref{eq:t2pair}, it suffices to show $D(t)>0$ for $0<t<\tfrac{1}{2}$.

\emph{Step 2:  the sign of $D'$.}
Since $dx_t/dt=1/f_Y(x_t)$ and $dy_t/dt=-1/f_Y(y_t)$,
\[
D'(t)=\varrho(y_t)-\varrho(x_t) \,,
\]
where
\[
\varrho(y)=\frac{s\,f_X(sy)}{f_Y(y)}=K\,y^{\alpha-\beta}e^{(1-s)y}
\]
with a constant $K>0$.  With $d_t=\log(x_t/y_t)>0$, this gives
\begin{equation}\label{eq:t2rho}
\log\frac{\varrho(y_t)}{\varrho(x_t)}=(\beta-\alpha)\,d_t-(1-s)(x_t-y_t) \,.
\end{equation}

First, $s<1$; that is, $\nu_\Gamma(\alpha)<\nu_\Gamma(\beta)$.
For if $s\ge1$, the right side of \eqref{eq:t2rho} is positive for every $t$.
Then $D$ is strictly increasing on $[0,\tfrac{1}{2})$, so $D(t)>D(0)=0$ for $t>0$.
That contradicts $D(t)\to0$ as $t\to\tfrac{1}{2}$.

With $s<1$, the right side of \eqref{eq:t2rho} equals $(1-s)\,d_t\,\big(q_0-\ell(t)\big)$, where
\[
q_0=\frac{\beta-\alpha}{1-s} \,, \qquad
\ell(t)=\frac{x_t-y_t}{\log(x_t/y_t)}
\]
($\ell$ is the logarithmic mean of $x_t$ and $y_t$).  Hence
$\operatorname{sign}D'(t)=\operatorname{sign}\big(q_0-\ell(t)\big)$.
By Lemma~\ref{lem:logmean} below, $\ell$ is strictly increasing on $(0,\tfrac{1}{2})$.
Now fix $t\in(0,\tfrac{1}{2})$.
\begin{itemize}
\item If $\ell(t)\le q_0$, then $\ell(u)<q_0$ for $0<u<t$.  So $D$ is strictly increasing on $[0,t]$,
and $D(t)>D(0)=0$.
\item If $\ell(t)>q_0$, then $\ell(u)>q_0$ for $t<u<\tfrac{1}{2}$.  So $D$ is strictly decreasing on
$[t,\tfrac{1}{2})$, and $D(t)>\lim_{u\to1/2}D(u)=0$.
\end{itemize}
Either way $D(t)>0$, which proves \eqref{eq:t2goal}.
\end{proof}

The proof rests on two lemmas:  one about a single gamma distribution, and one about an elementary
function.

\begin{Lemma}[Log-mean of symmetric quantiles]
\label{lem:logmean}
Let $\beta>0$, and let $Y\sim\textrm{Gamma}(\beta,1)$ have CDF $F$, PDF $f$ and quantile function $Q$.
Then
\[
\ell(t)=\frac{Q(\tfrac{1}{2} + t)-Q(\tfrac{1}{2} - t)}{\log\big(Q(\tfrac{1}{2} + t)/Q(\tfrac{1}{2} - t)\big)}
\]
is strictly increasing on $(0,\tfrac{1}{2})$.
\end{Lemma}

\begin{proof}
Write $x=Q(\tfrac{1}{2} + t)$, $y=Q(\tfrac{1}{2} - t)$, and $d=\log(x/y)>0$, so that $x=e^dy$ and
$F(x)+F(y)=1$.  Define
\[
a(d)=e^d(d-1)+1 \,, \quad b(d)=e^d-1-d \,, \quad \Lambda(d)=\log\frac{a(d)}{b(d)} \,,
\]
where $a,b>0$ for $d>0$.
Differentiating $\ell=(x-y)/d$, with $x'=1/f(x)$ and $y'=-1/f(y)$, gives
\[
d^2\,\ell'(t)=\frac{e^{-d}a(d)}{f(x)}-\frac{b(d)}{f(y)} \,.
\]
So $\ell'(t)>0$ if and only if
\[
f(x)/f(y)<e^{-d}a(d)/b(d)=e^{\Lambda(d)-d} \, .
\]
Since $f(x)/f(y)=(x/y)^{\beta-1}e^{-(x-y)}=e^{(\beta-1)d-(x-y)}$, this condition is
\begin{equation}\label{eq:t2star}
x-y>\beta d-\Lambda(d) \,.
\end{equation}

If $\beta d\le\Lambda(d)$, \eqref{eq:t2star} holds because $x>y$.
Otherwise $d$ lies in $J=\{\delta>0:\beta\delta>\Lambda(\delta)\}$.
By Lemma~\ref{lem:Lambda}, $\Lambda(\delta)/\delta$ is strictly decreasing, so $J$ is an interval
$(d_\beta,\infty)$.
For $\delta\in J$, define
\[
y_c(\delta)=\frac{\beta\delta-\Lambda(\delta)}{e^\delta-1}>0 \,, \qquad
x_c(\delta)=e^\delta y_c(\delta) \,,
\]
so that $x_c-y_c=\beta\delta-\Lambda(\delta)$.
Since $x-y=(e^d-1)y$, \eqref{eq:t2star} is equivalent to $y>y_c(d)$.
Let
\[
S(\delta)=F\big(y_c(\delta)\big)+F\big(x_c(\delta)\big) \,.
\]
If instead $y\le y_c(d)$, then, because $F$ is increasing, $1=F(y)+F(e^dy)\le S(d)$.
So it suffices to show $S<1$ on $J$.

\emph{$S$ is increasing on $J$.}
On $J$,
\[
f(x_c)=f(y_c)\,e^{(\beta-1)\delta-(x_c-y_c)}=f(y_c)\,e^{-\delta}a/b \,.
\]
Hence
\[
S'=f(y_c)\,y_c'+f(x_c)\,e^\delta(y_c+y_c')
=\frac{f(y_c)}{b}\Big[(a+b)\,y_c'+a\,y_c\Big] \,.
\]
Substitute $a+b=\delta(e^\delta-1)$, $\delta e^\delta-a=e^\delta-1$, and
\[
y_c'=\frac{(\beta-\Lambda')(e^\delta-1)-(\beta\delta-\Lambda)\,e^\delta}{(e^\delta-1)^2} \,.
\]
The bracket then reduces to $\delta(\beta-\Lambda')-(\beta\delta-\Lambda)=\Lambda-\delta\Lambda'$, so
\[
S'(\delta)=\frac{f\big(y_c(\delta)\big)}{b(\delta)}\,\big(\Lambda(\delta)-\delta\Lambda'(\delta)\big)>0
\]
by Lemma~\ref{lem:Lambda}.

\emph{Limit.}
As $\delta\to\infty$, $\Lambda(\delta)/\delta\to0$ (Lemma~\ref{lem:Lambda}).
Hence $y_c\le\beta\delta/(e^\delta-1)\to0$ and
$x_c\ge x_c-y_c=\delta\big(\beta-\Lambda(\delta)/\delta\big)\to\infty$, so $S(\delta)\to0+1=1$.
An increasing function with limit $1$ is less than $1$ everywhere on $J$.
\end{proof}

\begin{Lemma}[An elementary concavity]
\label{lem:Lambda}
With $a,b,\Lambda$ as in Lemma~\ref{lem:logmean}, for $d>0$:
$\Lambda(d)>d\,\Lambda'(d)$, and $\Lambda(d)/d$ is strictly decreasing, with limit $0$ as
$d\to\infty$.
\end{Lemma}

\begin{proof}
From $a'=de^d$ and $b'=e^d-1$, a direct expansion gives $\Lambda'=a'/a-b'/b=N/M$, where
\begin{align*}
N(d)&=e^{-d}(a'b-ab')=2(\cosh d-1)-d^2 \\
&=\sum_{n\ge2}\frac{2\,d^{2n}}{(2n)!} \,, \\
M(d)&=e^{-d}a\,b=2d\sinh d-2(\cosh d-1)-d^2 \\
&=\sum_{n\ge2}\frac{2(2n-1)\,d^{2n}}{(2n)!} \,.
\end{align*}
In particular $0<\Lambda'\le\tfrac{1}{3}$.

For concavity, note the symmetries $a(-d)=e^{-d}b(d)$ and $b(-d)=e^{-d}a(d)$.  The first gives
$a(d)=e^d\,b(-d)$, so with $L=\log b$,
\[
\Lambda(d)=d+L(-d)-L(d) \, .
\]
Since $e^d b-(e^d-1)^2=-a$, we have $L''=-a/b^2$, and by the symmetries,
\[
L''(-d)=-\frac{a(-d)}{b(-d)^2}=-\frac{e^d\,b(d)}{a(d)^2} \, .
\]
Hence
\[
\Lambda''(d)=L''(-d)-L''(d)=\frac{a(d)^3-e^d\,b(d)^3}{a(d)^2\,b(d)^2} \, ,
\]
and $\Lambda''<0$ exactly when $h(d)=e^{d/3}b(d)-a(d)>0$.  Here
$h(d)=e^{4d/3}-(1+d)\,e^{d/3}-(d-1)\,e^d-1=\sum_{n\ge1}c_n d^n/n!$ with
\[
3^n c_n=4^n-(n-1)\,3^n-3n-1 \, .
\]
This vanishes for $n=1,\dots,4$ and equals $36$ for $n=5$.  If $4^n>(n-1)3^n+3n+1$ for some $n\ge5$, then
$4^{n+1}>4(n-1)3^n+12n+4=n\,3^{n+1}+3(n+1)+1+(n-4)3^n+9n$, so by induction $c_n>0$ for all $n\ge5$.
Hence $h>0$ for $d>0$, and $\Lambda$ is strictly concave.

Since $a-b$ vanishes at $0$ and has derivative $a>0$, $b<a$.
Also $(1+d)\,b-a=2\big(e^d-1-d-\tfrac{1}{2} d^2\big)>0$.
Hence $0<\Lambda(d)<\log(1+d)$, so $\Lambda(0^+)=0$ and
$\Lambda(d)=\int_0^d\Lambda'(u)\,du$.
Because $\Lambda'$ is strictly decreasing, $\Lambda(d)>d\,\Lambda'(d)$.
Then $(\Lambda/d)'=(d\Lambda'-\Lambda)/d^2<0$, and
$\Lambda(d)/d<\log(1+d)/d\to0$.
\end{proof}

\subsection{Gamma median bounds}
\label{app:gamma}

\begin{Theorem}[Gamma median bounds]
\label{thm:gamma}
Let $\nu_{\Gamma4}(\alpha) = \log 2/(2^{1/\alpha} - 1)$.  Then $\nu_\Gamma(\alpha) < \nu_{\Gamma4}(\alpha)$ for
$0 < \alpha < 1$, $\nu_{\Gamma4}(\alpha) < \nu_\Gamma(\alpha)$ for $\alpha > 1$, and the two are equal at
$\alpha = 1$.
\end{Theorem}

\begin{proof}
This is the special case of Theorem~\ref{thm:lower_bound} in which one shape parameter is $1$,
using the exact values in Lemma~\ref{lem:basic}.  Write $q=2^{-1/\alpha}\in(0,1)$, so that
\[
\nu_{\Gamma 4}(\alpha)=\frac{\log 2}{2^{1/\alpha}-1}=\frac{q\log 2}{1-q} \,.
\]

\emph{Case $0<\alpha<1$.}
Theorem~\ref{thm:lower_bound} with shape parameters $(\alpha,1)$ gives the strict inequality
\[
\frac{\nu_\Gamma(\alpha)}{\nu_\Gamma(\alpha)+\log 2}<q \,.
\]
The map $t\mapsto t/(t+\log 2)$ is strictly increasing on $(0,\infty)$, and it equals $q$ at
$t=\nu_{\Gamma 4}(\alpha)$.  Hence $\nu_\Gamma(\alpha)<\nu_{\Gamma 4}(\alpha)$.

\emph{Case $\alpha>1$.}
Theorem~\ref{thm:lower_bound} with shape parameters $(1,\alpha)$ gives
\[
\frac{\log 2}{\log 2+\nu_\Gamma(\alpha)}<1-q \,.
\]
The map $t\mapsto \log 2/(\log 2+t)$ is strictly decreasing on $(0,\infty)$, and it equals $1-q$
at $t=\nu_{\Gamma 4}(\alpha)$.  Hence $\nu_\Gamma(\alpha)>\nu_{\Gamma 4}(\alpha)$.

At $\alpha=1$, both sides equal $\log 2$.
\end{proof}

Besides $\nu_\Gamma(1) = \log 2$, the gamma median has a closed form at $\alpha = 2$, where the CDF is
$1 - (1+x)e^{-x}$.  Setting it to $\tfrac{1}{2}$ gives $-(1+x)e^{-(1+x)} = -1/(2e)$, so
\[
\nu_\Gamma(2) = -1 - W_{-1}\big(-1/(2e)\big) \approx 1.678347 \, ,
\]
with $W_{-1}$ the lower real branch of the Lambert $W$ function \citep{corless1996lambert}.  At
$\alpha = \tfrac{1}{2}$, the median is $\nu_\Gamma(\tfrac{1}{2}) = \big(\operatorname{erf}^{-1}(\tfrac{1}{2})\big)^2 \approx 0.227468$
in terms of the inverse error function.  For integers $\alpha \ge 3$, the equation
$e^{-x}\sum_{k<\alpha} x^k/k! = \tfrac{1}{2}$ is beyond the Lambert $W$ function.

\subsection{The ratio form:  special case and non-existence}
\label{app:ratio}

\begin{Lemma}[Special case]
\label{lem:special}
If a function $g$ satisfies $\nu_\textrm{B}(\alpha,\beta) = g(\alpha)/(g(\alpha) + g(\beta))$ for all
$\alpha, \beta > 0$, then
\[
\frac{g(\alpha)}{g(1)} = \frac{1}{2^{1/\alpha} - 1} \quad \textrm{for all} \quad \alpha > 0 \, .
\]
The same conclusion holds if the relation is assumed only for $\beta = 1$.
\end{Lemma}

\begin{proof}
The relation implies $g(\alpha)/g(\beta) = \nu_\textrm{B}(\alpha,\beta)/(1 - \nu_\textrm{B}(\alpha,\beta))$.  At
$\beta = 1$, Lemma~\ref{lem:basic} gives $\nu_\textrm{B}(\alpha,1) = 2^{-1/\alpha}$, so
$g(\alpha)/g(1) = 2^{-1/\alpha}/(1 - 2^{-1/\alpha}) = 1/(2^{1/\alpha} - 1)$.  (The relation forces
$g(1) \ne 0$, since otherwise the median at $\beta = 1$ would be $0$ or $1$.)
\end{proof}

\begin{Theorem}[Non-existence]
\label{thm:non}
No function $g$ of one variable satisfies
\[
\nu_\textrm{B}(\alpha,\beta) = \frac{g(\alpha)}{g(\alpha) + g(\beta)} \quad \textrm{for all} \quad 0 < \alpha < \beta \, .
\]
\end{Theorem}

\begin{proof}
Suppose such a $g$ exists.  The relation for the pairs $(\tfrac{1}{2}, 1)$ and $(1, 2)$ gives, as in
Lemma~\ref{lem:special} (using $\nu_\textrm{B}(1,2) = 1 - 2^{-1/2}$ for the second),
\[
\frac{g(\tfrac{1}{2})}{g(1)} = \frac{1}{2^2 - 1} = \frac{1}{3} \, , \qquad \frac{g(2)}{g(1)} = \frac{1}{\sqrt2 - 1} \, .
\]
So the relation at $(\tfrac{1}{2}, 2)$ predicts the median
\[
\nu = \frac{1/3}{1/3 + 1/(\sqrt2 - 1)} = \frac{\sqrt2 - 1}{\sqrt2 + 2} = \frac{3\sqrt2 - 4}{2} \, .
\]
The density of $\textrm{Beta}(\tfrac{1}{2}, 2)$ is $\tfrac{3}{4}  x^{-1/2}(1-x)$, with CDF
$\tfrac{3}{2}  x^{1/2} - \tfrac{1}{2} x^{3/2}$.  At the median this equals $\tfrac{1}{2}$, so
$\sqrt\nu\,(3 - \nu) = 1$, and squaring gives
\[
\nu(3-\nu)^2 = 1 \, .
\]
For the predicted value,
\[
\nu(3-\nu)^2 = \frac{3\sqrt2-4}{2}\left(\frac{10-3\sqrt2}{2}\right)^2 = \frac{297\sqrt2 - 416}{4} \, ,
\]
which equals $1$ only if $\sqrt2 = 420/297$, contradicting the irrationality of $\sqrt2$.
\end{proof}

Solving the cubic, the true median is found to be $\nu_\textrm{B}(\tfrac{1}{2},2) = 4\cos^2(4\pi/9) \approx 0.120615$, while the ratio form
of Lemma~\ref{lem:special}, which is $\nu_{\textrm{B}4}$, gives $\approx 0.121320$.

The median equation is algebraic more generally.  If $\beta = n$ is a positive integer and
$\alpha = p/q$ is rational in lowest terms, the substitution $x = y^q$ turns
$I_x(\alpha,n) = \tfrac{1}{2}$ into a polynomial equation in $y$ of degree $p + q(n-1)$, and by symmetry
the same holds with the shape parameters exchanged (with the substitution $1 - x = y^q$ when $\alpha = n$ and
$\beta = p/q$).  For $\textrm{Beta}(1/q, 2)$ it is the trinomial
$2y^{q+1} - 2(q+1)y + q = 0$.  The case $q = 2$ is the cubic above, $y^3 - 3y + 1 = 0$ with
$y = \sqrt{\nu}$; its Galois group is cyclic, which is why its root is a cosine,
$y = 2\cos(4\pi/9)$.  We computed the Galois groups of all such equations with $q \le 9$,
$\alpha \le 2$, $2 \le n \le 5$, and degree at most $30$, $163$ cases in all.  Besides $(\tfrac{1}{2},2)$,
the median is expressible in radicals in only two of them, both quartics.  For $(\tfrac{1}{3},2)$, the
equation is $2y^4 - 8y + 3 = 0$ with $\nu = y^3$, and Ferrari's method gives
\[
y = \tfrac{1}{2}\left(\sqrt{2m} - \sqrt{4\sqrt{2m}/m - 2m}\right) , \qquad
m = \sqrt[3]{1 + \sqrt{7/8}} + \sqrt[3]{1 - \sqrt{7/8}} \, ,
\]
so $\nu_\textrm{B}(\tfrac{1}{3},2) \approx 0.0549696$.  For $(2,3)$, the median satisfies $1 - \nu = 1/z$ with $z^4 - 8z + 6 = 0$,
and
\[
z = \tfrac{1}{2}\left(\sqrt{2m} + \sqrt{8\sqrt{2m}/m - 2m}\right) , \qquad
m = \sqrt[3]{4 + 2\sqrt{2}} + \sqrt[3]{4 - 2\sqrt{2}} \, ,
\]
so $\nu_\textrm{B}(2,3) \approx 0.385728$.  One more case is a curiosity:  for $(\tfrac{1}{6},2)$, the
equation is $y^7 - 7y + 3 = 0$ with $\nu = y^6 \approx 0.0062296$.  This is the Trinks polynomial, whose Galois
group is the simple group $\textrm{PSL}(2,7)$ of order $168$ \citep{erbach1979polynomials}, so it is not
solvable by radicals.  In each of the remaining $159$ cases, the factorization of the polynomial modulo a
small prime has a cycle type that, by Jordan's theorem, puts the alternating group of its degree in the
Galois group, so none of them is solvable by radicals either.  When neither shape parameter is an
integer, the CDF involves inverse trigonometric or other transcendental functions, and the median
equation is in general not algebraic.

\subsection{Kerman and Wilson--Hilferty bounds}
\label{app:kerman_wh}

\begin{Theorem}[Kerman lower bound]
\label{thm:kerman}
For $\tfrac{1}{3} < \alpha \le \beta$,
\[
0 < \nu_{\textrm{B}1}(\alpha,\beta) = \frac{\alpha - \frac{1}{3}}{\alpha + \beta - \frac{2}{3}} \le \nu_\textrm{B}(\alpha,\beta) \, ,
\]
with equality only when $\alpha = \beta$.
\end{Theorem}

For $\alpha \le \tfrac{1}{3}$ and $\alpha + \beta > \tfrac{2}{3}$ the inequality holds trivially, since then
$\nu_{\textrm{B}1} \le 0$.  For $\alpha + \beta < \tfrac{2}{3}$, both numerator and denominator are negative,
$\nu_{\textrm{B}1}$ is positive, and it is not a bound.  \citet{kerman2011closed} observed that the approximation
consistently underestimates the median when $\alpha < \beta$, but did not prove it.

\begin{proof}
If $\alpha = \beta$, both sides equal $\tfrac{1}{2}$.  Let $\tfrac{1}{3} < \alpha < \beta$.

Let $D(\alpha) = \alpha - \nu_\Gamma(\alpha)$, the mean--median difference of the gamma distribution.
\citet{berg2006chen} proved $0 < \nu_\Gamma'(\alpha) < 1$, so $D$ is strictly increasing on
$(0,\infty)$, and $D(\alpha) \to \tfrac{1}{3}$ as $\alpha \to \infty$.  Hence
$\delta(\alpha) = \tfrac{1}{3} - D(\alpha) = \nu_\Gamma(\alpha) - (\alpha - \tfrac{1}{3})$ is positive and
strictly decreasing.  (The rest of the proof uses only that $\delta$ is positive and nonincreasing, the
form in which the Lean proof takes the two facts; see Section~\ref{app:lean}.)  For $\alpha > \tfrac{1}{3}$,
\[
\frac{\nu_\Gamma(\alpha)}{\alpha - \frac{1}{3}} = 1 + \frac{\delta(\alpha)}{\alpha - \frac{1}{3}}
\]
has a positive decreasing numerator over a positive increasing denominator, so
$R(\alpha) = (\alpha - \tfrac{1}{3})/\nu_\Gamma(\alpha)$ is strictly increasing.

Then $R(\alpha) < R(\beta)$ gives $(\alpha - \tfrac{1}{3})\,\nu_\Gamma(\beta) < (\beta - \tfrac{1}{3})\,\nu_\Gamma(\alpha)$.
Adding $(\alpha - \tfrac{1}{3})\,\nu_\Gamma(\alpha)$ to both sides and dividing by the positive quantities
$\alpha + \beta - \tfrac{2}{3}$ and $\nu_\Gamma(\alpha) + \nu_\Gamma(\beta)$ yields
\[
\frac{\alpha - \frac{1}{3}}{\alpha + \beta - \frac{2}{3}} < \frac{\nu_\Gamma(\alpha)}{\nu_\Gamma(\alpha) + \nu_\Gamma(\beta)}
= \nu_{\textrm{B}0}(\alpha,\beta) \, ,
\]
and $\nu_{\textrm{B}0}(\alpha,\beta) < \nu_\textrm{B}(\alpha,\beta)$ by Theorem~\ref{thm:lower_bound}.
\end{proof}

On the two lines $\alpha = 1$ and $\beta = 1$ the beta median is known exactly (Lemma~\ref{lem:basic}), and
there the Wilson--Hilferty form $\nu_{\textrm{B}2}$ can be settled completely.  Here
$\nu_{\Gamma2}(\alpha) = \alpha(1 - \tfrac{1}{9\alpha})^3$, which is positive for $\alpha > \tfrac{1}{9}$.

\begin{Theorem}[Wilson--Hilferty bound on the lines $\alpha = 1$ and $\beta = 1$]
\label{thm:wh_edge}
(a) For $\beta > 1$, $\nu_\textrm{B}(1,\beta) < \nu_{\textrm{B}2}(1,\beta)$.

(b) There is $\alpha_0 \in (\tfrac{2}{5} , \tfrac{1}{2})$, numerically $\alpha_0 \approx 0.4076$, such that
$\nu_\textrm{B}(\alpha,1) < \nu_{\textrm{B}2}(\alpha,1)$ for $\alpha_0 < \alpha < 1$, and
$\nu_{\textrm{B}2}(\alpha,1) < \nu_\textrm{B}(\alpha,1)$ for $\tfrac{1}{9} < \alpha < \alpha_0$.

So on these lines, with $\alpha < \beta$, $\nu_{\textrm{B}2}$ is an upper bound exactly when the smaller
shape parameter exceeds $\alpha_0$.
\end{Theorem}

\begin{proof}
We have $\nu_{\Gamma2}(1) = (8/9)^3 = 512/729$.  Let $u$ be the reciprocal of the shape parameter that differs
from~1, so $u \in (0,1)$ on line (a) and $u \in (1,9)$ on line (b); in both cases that shape parameter has
$\nu_{\Gamma2} = (1 - u/9)^3/u$.  Comparing odds, both lines reduce to comparing $2^u$ with
\[
R(u) = 1 + \frac{512}{729}\,\frac{u}{(1 - u/9)^3} \, .
\]
On line (a), $\nu_\textrm{B} < \nu_{\textrm{B}2}$ exactly when $2^u < R(u)$.  On line (b),
$\nu_{\textrm{B}2}(\alpha,1) - \nu_\textrm{B}(\alpha,1)$ is a positive multiple of $2^u - R(u)$, so
$\nu_\textrm{B} < \nu_{\textrm{B}2}$ exactly when $2^u > R(u)$.

Let $D(u) = \log R(u) - u \log 2$ on $[0,9)$.  With $w = 1 - u/9$ and $c = 512/729$,
\[
R' = \frac{c\,(1 + 2u/9)}{w^4} \, , \qquad R'' = \frac{c\,(2/3 + 2u/27)}{w^5} \, ,
\]
and
\[
R''R - R'^2 = -\frac{c\,P(u)}{59049\,w^8} \, ,
\]
where
\[
P(u) = 2106 - 468u - 1024u^2 - 108u^3 + 6u^4 \, ,
\]
so $D'' = (R''R - R'^2)/R^2$ has the sign of $-P$.  On $[0,9]$, the derivative $P'(u) = -468 - 2048u - 324u^2 + 24u^3$ is negative,
because $24u^3 \le 1944u$ there.  So $P$ is strictly decreasing, with a single root
$u^* \in (1.1, 1.2)$, since $P(1.1) > 0 > P(1.2)$.  Hence $D$ is strictly concave on $[0,u^*]$ and
strictly convex on $[u^*, 9)$.

Since $D(0) = D(1) = 0$, concavity gives $D > 0$ on $(0,1)$, which proves (a), and $D < 0$ on
$(1, u^*]$.  Next, $R(2) = 1367/343 < 4 = 2^2$ and $R(5/2)^2 = (12437/2197)^2 > 32 = (2^{5/2})^2$, so
$D(2) < 0 < D(5/2)$.  Convexity on $[u^*, 9)$, with $D(u^*) < 0$, then gives exactly one zero
$u_1 \in (2, \tfrac{5}{2} )$, with $D < 0$ on $(u^*, u_1)$ and $D > 0$ on $(u_1, 9)$.  On line (b), $D < 0$ means $2^u > R(u)$, and so
$\nu_\textrm{B} < \nu_{\textrm{B}2}$; thus $\nu_\textrm{B} < \nu_{\textrm{B}2}$ for $u \in (1, u_1)$ and
$\nu_{\textrm{B}2} < \nu_\textrm{B}$ for $u \in (u_1, 9)$, which is statement (b) with $\alpha_0 = 1/u_1$.
\end{proof}

\begin{Conjecture}[Wilson--Hilferty upper bound]
\label{conj:wh}
For $1 \le \alpha < \beta$, $\nu_\textrm{B}(\alpha,\beta) < \nu_{\textrm{B}2}(\alpha,\beta)$.
\end{Conjecture}

\emph{Evidence.}  Theorem~\ref{thm:wh_edge}(a) is the line $\alpha = 1$.  Numerically the bound holds
for all $\beta > \alpha$ once $\alpha \gtrsim 0.615$.  Along the line $\beta = 1$, it holds down to
$\alpha_0 \approx 0.4076$, so the binding constraint lies at some $\beta > 1$.

For each $\beta > 1$, there is a crossing in $\alpha$.  As $\alpha \to \tfrac{1}{9}^+$,
$\nu_{\Gamma2}(\alpha) \to 0$, so $\nu_{\textrm{B}2}(\alpha,\beta) \to 0 < \nu_\textrm{B}(\tfrac{1}{9},\beta)$.
At $\alpha = 1$, Theorem~\ref{thm:wh_edge}(a) gives $\nu_{\textrm{B}2} > \nu_\textrm{B}$.  Since both are
continuous in $\alpha$, the intermediate value theorem gives a crossing $\alpha_0(\beta) \in (\tfrac{1}{9}, 1)$.
Numerically the crossing is unique, and $\alpha_0(1) = \alpha_0$ of Theorem~\ref{thm:wh_edge}(b).

In logit terms, the conjecture says that the excess of the true median over the lower bound of
Theorem~\ref{thm:lower_bound},
\[
T(\alpha,\beta) = \operatorname{logit} \nu_\textrm{B}(\alpha,\beta) - \log\frac{\nu_\Gamma(\alpha)}{\nu_\Gamma(\beta)} > 0 \, ,
\]
is at most $\rho_2(\alpha) - \rho_2(\beta)$, where $\rho_2 = \log(\nu_{\Gamma2}/\nu_\Gamma)$.  A necessary
condition is that $\nu_{\Gamma2}/\nu_\Gamma$ be decreasing, which holds numerically for $\alpha \gtrsim 0.40$.
Over $1 \le \alpha \le \beta$, $T$ is at most about a third of $\rho_2(\alpha) - \rho_2(\beta)$, largest
near $(\alpha,\beta) = (1,2)$.  Near the diagonal, the $\alpha$-slope of $\operatorname{logit}\nu_\textrm{B}$
at $\alpha = \beta$ exceeds that of $\log \nu_{\Gamma2}$.  At $\beta = 1$ the two slopes are
$2\log 2 > \tfrac{11}{8}$; numerically the difference is about $0.03/\beta^3$ for large $\beta$.

For large shape parameters the conjecture holds along every ray $\beta = r\alpha$ with $r > 1$ (Theorem~\ref{thm:whray}),
and it would follow everywhere from a uniform bound on the remainder of the large-shape expansion
(Conjecture~\ref{conj:largeR}); see Section~\ref{app:large}.  The margin is of second order, since
$\nu_{\Gamma2}(\alpha) - \nu_\Gamma(\alpha) \approx 7/(405\alpha)$.  This is what makes the conjecture harder
to prove than Theorem~\ref{thm:lower_bound}:  a proof needs an upper bound on $\nu_\textrm{B}$ that is
accurate to second order, and the pairing arguments used here give upper bounds only to first order,
such as $\nu_\textrm{B} < \alpha/(\alpha+\beta)$ (Lemma~\ref{lem:bracket}).

\subsection{Large-shape asymptotics}
\label{app:large}

When both shape parameters are large, each of the ratio forms $\nu_{\textrm{B}0}$, $\nu_{\textrm{B}1}$, and $\nu_{\textrm{B}2}$ agrees with $\nu_\textrm{B}$ to first order, so they differ only at second order;
$\nu_{\textrm{B}4}$ errs at first order.  An expansion of the median logit makes the comparison explicit.
Write $\Delta = 1/\alpha - 1/\beta$, which is positive for
$\alpha < \beta$, and
\begin{equation}\label{eq:largeexp}
\Lambda_\infty(\alpha,\beta) = \log\frac{\alpha}{\beta} - \frac{\Delta}{3}
- \frac{\Delta\,(29\alpha^2 + 50\alpha\beta + 29\beta^2)}{810\,\alpha\beta\,(\alpha+\beta)} \, ,
\end{equation}
with remainder $R(\alpha,\beta) = \operatorname{logit}\nu_\textrm{B}(\alpha,\beta) - \Lambda_\infty(\alpha,\beta)$.

\begin{Theorem}[Large-shape expansion]
\label{thm:largeexp}
Fix $r \ge 1$.  As $\alpha \to \infty$ with $\beta = r\alpha$, $R(\alpha,\beta) = O(\alpha^{-3})$.
\end{Theorem}

Since $\Delta = (1 - 1/r)/\alpha$, the three terms of \eqref{eq:largeexp} are of orders $0$, $1$, and $2$
in $1/\alpha$, so the theorem gives $\operatorname{logit}\nu_\textrm{B}$ to second order along every ray.
By symmetry the same holds for $r < 1$.  This theorem and Theorem~\ref{thm:cfamily}, with their proofs,
also appear in \citet{lyon2026kerman}.

\begin{proof}
Let $n = \alpha + \beta$, $p = \alpha/n$, and $q = \beta/n$, so $p$ and $q$ are fixed.  The logit
$W = \operatorname{logit} X$ of $X \sim \textrm{Beta}(\alpha,\beta)$ has density proportional to
$e^{\alpha w}(1+e^w)^{-\alpha-\beta} = e^{n\varphi(w)}$, with $\varphi(w) = pw - \log(1+e^w)$, and its
median is $\operatorname{logit}\nu_\textrm{B}$.  Here $\varphi' = p - \sigma$ and
$\varphi'' = -\sigma(1-\sigma) < 0$, so $\varphi$ is strictly concave with maximum at
$w^* = \log(p/q) = \log(\alpha/\beta)$, where $\sigma = p$.  Put
$\chi(x) = \varphi(w^*+x) - \varphi(w^*)$.  Differentiating $\sigma' = \sigma(1-\sigma)$ repeatedly gives
\[
\chi(x) = -\frac{A}{2}x^2 + \frac{B}{6}x^3 + \frac{C}{24}x^4 + \frac{D}{120}x^5 + O(x^6) \, ,
\]
with $A = pq$, $B = -pq(q-p)$, $C = -pq(1-6pq)$, and $D = -pq(q-p)(1-12pq)$.

The median $x_m$ of $x = W - w^*$ satisfies $\int_{-\infty}^{x_m} e^{n\chi} = \int_{x_m}^\infty e^{n\chi}$,
that is,
\[
2\int_0^{x_m} e^{n\chi(x)}\,dx = \int_0^\infty \left(e^{n\chi(x)} - e^{n\chi(-x)}\right) dx \, .
\]
We expand both sides by Laplace's method \citep{olver1974asymptotics}.  By concavity,
$\chi(x) \le \chi(\pm\delta)\,|x|/\delta$ for $|x| \ge \delta$, so the parts of the integrals with
$|x| \ge \delta$ are exponentially small in $n$.  For $|x| < \delta$, put $x = y/\sqrt n$.  Then
$n\chi = -Ay^2/2 + \varepsilon(y)$, where
$\varepsilon = By^3/(6\sqrt n) + Cy^4/(24n) + Dy^5/(120 n^{3/2}) + O(y^6 n^{-2})$, and for small
$\delta$ we have $|\varepsilon| \le Ay^2/4$, which gives the usual Gaussian domination of the remainders.
Expanding $e^\varepsilon$, the terms of order $n^{-1/2}$ and $n^{-3/2}$ are odd in $y$, and the terms of
order $1$, $n^{-1}$, and $n^{-2}$ are even.  Only odd terms survive on the right.  With
$\int_0^\infty y^{2j+1}e^{-Ay^2/2}\,dy = j!\,2^j/A^{j+1}$, the right side is $n^{-1/2}$ times
\[
\frac{2B}{3A^2}\,n^{-1/2} + \left(\frac{2D}{15A^3} + \frac{2BC}{3A^4} + \frac{16B^3}{27A^5}\right)n^{-3/2}
+ O(n^{-5/2}) \, ,
\]
where the $n^{-3/2}$ coefficient collects $Dy^5/120$, $BCy^7/144$, and $B^3y^9/1296$ from $e^\varepsilon$.
In particular the right side is $O(n^{-1})$.  On the left the integrand is close to $1$ near $0$, so
$y_m = \sqrt n\,x_m = O(n^{-1/2})$; then $\varepsilon = O(n^{-2})$ on $[0, y_m]$, and the left side is $n^{-1/2}$ times
$2y_m - Ay_m^3/3 + O(n^{-5/2})$.  Matching powers gives
$y_m = \mu_1 n^{-1/2} + \mu_3 n^{-3/2} + O(n^{-5/2})$, with
\[
\mu_1 = \frac{B}{3A^2} \, , \qquad
\mu_3 = \frac{D}{15A^3} + \frac{BC}{3A^4} + \frac{49B^3}{162A^5} \, ,
\]
where $\frac{49}{162} = \frac{8}{27} + \frac{1}{162}$ combines half the last coefficient on the right with
the term $A\mu_1^3/6$ from the left, so $x_m = \mu_1/n + \mu_3/n^2 + O(n^{-3})$.

It remains to evaluate the coefficients.  First, $\mu_1/n = -(q-p)/(3npq) = -\Delta/3$, since
$(q-p)/(npq) = (\beta - \alpha)/(\alpha\beta) = \Delta$.  Next, using $(q-p)^2 = 1 - 4pq$,
\begin{align*}
\mu_3 &= \frac{q-p}{p^2q^2}\left[-\frac{1-12pq}{15} + \frac{1-6pq}{3} - \frac{49(1-4pq)}{162}\right] \\
&= -\frac{(q-p)(29 - 8pq)}{810\,p^2q^2} \, ,
\end{align*}
so $\mu_3/n^2 = -\Delta(29 - 8pq)/(810\,npq)$.  As $n^2(29 - 8pq) = 29\alpha^2 + 50\alpha\beta + 29\beta^2$
and $n^3pq = \alpha\beta(\alpha+\beta)$, this is the last term of \eqref{eq:largeexp}.  Finally
$\operatorname{logit}\nu_\textrm{B} = w^* + x_m$, and $O(n^{-3}) = O(\alpha^{-3})$ for fixed $r$.
\end{proof}

The same expansion comes out formally from the Cornish--Fisher expansion \citep{cornish1938moments}
of the median of $W$, whose cumulants are $\kappa_n = \psi_{n-1}(\alpha) + (-1)^n\psi_{n-1}(\beta)$.
It can also be derived in the original variable, by Temme's quadratic change of variable, which makes the
exponent exactly Gaussian, followed by Watson's lemma, the route of Temme's uniform asymptotic expansions
of the incomplete beta function \citep{temme1992asymptotic}; see \citet{lyon2026kerman}.  As
$\beta \to \infty$, \eqref{eq:largeexp} becomes $\log(\alpha/\beta) - 1/(3\alpha) - 29/(810\alpha^2)$, which
is $\log(\nu_\Gamma(\alpha)/\beta)$ to second order by
$\nu_\Gamma(\alpha) = \alpha - \tfrac{1}{3} + 8/(405\alpha) + O(\alpha^{-2})$ \citep{berg2006chen}, as
Theorem~\ref{thm:asymptote} requires.

\emph{Consequences.}  To second order, the logit errors of the approximations are
\begin{align*}
\operatorname{logit}\nu_\textrm{B} - \operatorname{logit}\nu_{\textrm{B}0}
&\approx \frac{4\Delta}{405\,(\alpha+\beta)} \, , \\
\operatorname{logit}\nu_\textrm{B} - \operatorname{logit}\nu_{\textrm{B}1}
&\approx \frac{8\Delta\,(2\alpha+\beta)(\alpha+2\beta)}{405\,\alpha\beta\,(\alpha+\beta)} \, , \\
\operatorname{logit}\nu_{\textrm{B}2} - \operatorname{logit}\nu_\textrm{B}
&\approx \frac{\Delta\,(7\alpha^2+10\alpha\beta+7\beta^2)}{405\,\alpha\beta\,(\alpha+\beta)} \, , \\
\operatorname{logit}\nu_\textrm{B} - \operatorname{logit}\nu_{\textrm{B}4}
&\approx \left(\tfrac{1}{2} \log 2 - \tfrac{1}{3}\right)\Delta \, .
\end{align*}
All four are positive for $\alpha < \beta$.  For $\nu_{\textrm{B}0}$ and $\nu_{\textrm{B}1}$ this agrees with
Theorems~\ref{thm:lower_bound} and~\ref{thm:kerman}, and for $\nu_{\textrm{B}4}$ with
Theorem~\ref{thm:B4large}.  For $\nu_{\textrm{B}2}$ the margin over $T$ is large:  $T$ is asymptotically a
fraction $4\alpha\beta/(7(\alpha+\beta)^2) \le \tfrac{1}{7}$ of $\rho_2(\alpha) - \rho_2(\beta)$.  Near the
diagonal the margin is $4(\beta - \alpha)/(135\beta^3)$ to leading order, matching the slope difference
of about $0.03/\beta^3$ found numerically.

The second-order term of \eqref{eq:largeexp} splits into $-\tfrac{29}{810}(\alpha^{-2} - \beta^{-2})$
and the $\nu_{\textrm{B}0}$ error $4\Delta/(405(\alpha+\beta))$.  The first part has the form
$f(\alpha) - f(\beta)$.  The second does not:  it vanishes when either shape parameter tends to infinity without
vanishing identically, which a nonzero difference $f(\alpha) - f(\beta)$ cannot do.  A ratio form
$g(\alpha)/(g(\alpha) + g(\beta))$ has logit $\log g(\alpha) - \log g(\beta)$, so no ratio form is accurate
to second order for all large shape parameters.  This is an asymptotic counterpart of Theorem~\ref{thm:non}.  Up to
a constant factor, ratio forms with $g(x) = x - \tfrac{1}{3} + 8/(405x) + O(x^{-2})$, such as
$\nu_{\textrm{B}0}$, are the only ones whose second-order error vanishes as $\beta \to \infty$.  Their
second-order error is then exactly the non-separable part.

To turn the comparison with $\nu_{\textrm{B}2}$ into a statement about the median itself, we need
the gap between its logit and $\Lambda_\infty$.

\begin{Lemma}[Wilson--Hilferty margin]
\label{lem:whmargin}
Let $E(\alpha,\beta) = \operatorname{logit}\nu_{\textrm{B}2}(\alpha,\beta) - \Lambda_\infty(\alpha,\beta)$.  For
$1 \le \alpha < \beta$,
\[
E(\alpha,\beta) > \frac{\Delta}{65\,\alpha} \, .
\]
\end{Lemma}

\begin{proof}
Let $a = 1/\alpha \le 1$, $b = 1/\beta < a$, and $s = b/a \in (0,1)$.  Since
$\operatorname{logit}\nu_{\textrm{B}2} = \log(\alpha/\beta) + 3\log(1 - a/9) - 3\log(1 - b/9)$, expanding
the logarithms gives
\[
E = (a-b)\,\frac{29a^2 + 50ab + 29b^2}{810\,(a+b)} + \frac{a-b}{3}
- 3\sum_{k\ge1} \frac{a^k - b^k}{k\,9^k} \, .
\]
The $k = 1$ term cancels $(a-b)/3$, and the $k = 2$ term $(a^2-b^2)/54$ combines with the first term
to give
\[
\frac{E}{a-b} = \frac{7a^2 + 10ab + 7b^2}{405\,(a+b)} - 3\sum_{k\ge3}\frac{h_{k-1}}{k\,9^k} \, ,
\]
with $h_{k-1} = (a^k - b^k)/(a-b) = \sum_{j=0}^{k-1} a^j b^{k-1-j} \le k\,a^{k-1}$.  The $k = 3$ term is
$a^2(1+s+s^2)/729$.  The rest is at most $3\sum_{k\ge4} a^{k-1}/9^k = 3a^3/(9^4(1 - a/9)) \le a^3/1944$.
The first term equals $a\,F(s)$ with $F(s) = (7s + 3 + 4/(1+s))/405$.  So
\[
\frac{E}{a(a-b)} \ge F(s) - \frac{a(1+s+s^2)}{729} - \frac{a^2}{1944} =: G(a,s) \, .
\]
$G$ is decreasing in $a$, so $G(a,s) \ge G(1,s)$.  And $G(1,s)$ is increasing in $s$, since its
derivative is $(7 - 4/(1+s)^2)/405 - (1+2s)/729 \ge 3/405 - 3/729 > 0$.  Hence
$G(a,s) \ge G(1,0) = 7/405 - 1/729 - 1/1944 > 1/65$.  Since $a(a-b) = \Delta/\alpha$, the lemma follows.
\end{proof}

\begin{Theorem}[Wilson--Hilferty bound on rays]
\label{thm:whray}
For each $r > 1$, $\nu_\textrm{B}(\alpha, r\alpha) < \nu_{\textrm{B}2}(\alpha, r\alpha)$ for all
sufficiently large $\alpha$.
\end{Theorem}

\begin{proof}
$\operatorname{logit}\nu_{\textrm{B}2} - \operatorname{logit}\nu_\textrm{B} = E - R$.  By
Lemma~\ref{lem:whmargin}, $E > \Delta/(65\alpha) = (1 - 1/r)/(65\alpha^2)$, while $R = O(\alpha^{-3})$
by Theorem~\ref{thm:largeexp}.
\end{proof}

So Conjecture~\ref{conj:wh} holds asymptotically along every ray.  Extending this to the whole region
needs a bound on $R$ that is uniform in $r$, which the numerics strongly support.

\begin{Conjecture}[Uniform remainder]
\label{conj:largeR}
For all $0 < \alpha < \beta$,
\[
0 < R(\alpha,\beta) < \frac{\Delta}{72\,\alpha^2} \, .
\]
\end{Conjecture}

\emph{Evidence.}  Exact medians were computed in 30- to 40-digit arithmetic for $\alpha$ from $0.05$
to $80$ and $\beta/\alpha$ from $1 + 10^{-6}$ to $10^4$.  Over this range $R\,\alpha^2/\Delta$ lies between
$0.0012$ and $0.01371$.  It is largest near the diagonal at $\alpha \approx 0.96$, falls to about $0.0095$
on the diagonal and $0.0015$ for $\beta/\alpha \ge 10^3$ as $\alpha$ grows, and decreases toward
$0$ as $\alpha \to 0$.  The bound $\tfrac{1}{72} \approx 0.01389$ thus has about $1\%$ to spare.  Along each
ray, Theorem~\ref{thm:largeexp} shows that $R\,\alpha^2/\Delta$ stays bounded as $\alpha \to \infty$; the
conjecture asks for a bound uniform in $\beta/\alpha$ and valid for all $\alpha$.

Conjecture~\ref{conj:largeR} implies Conjecture~\ref{conj:wh} (an implication also checked in Lean,
Section~\ref{app:lean}):  for $1 \le \alpha < \beta$,
Lemma~\ref{lem:whmargin} gives
\[
\operatorname{logit}\nu_{\textrm{B}2} - \operatorname{logit}\nu_\textrm{B} = E - R
> \frac{\Delta}{65\alpha} - \frac{\Delta}{72\alpha^2} > 0 \, .
\]
Only its upper bound for $\alpha \ge 1$ is needed for this.  Its lower bound $R > 0$ says that the
logistic function of $\Lambda_\infty$ is a lower bound on $\nu_\textrm{B}$.  Its relative error is
$-2.6 \times 10^{-3}$ at $(\alpha,\beta) = (1,2)$, $-1.6 \times 10^{-5}$ at $(5,10)$ and
$-4.5 \times 10^{-7}$ at $(10,11)$, ten to fifty times smaller than that of $\nu_{\textrm{B}1}$.

The expansion also decides which ratio forms are bounds for large shape parameters.  Consider
$g(x) = x - \tfrac{1}{3} + c/x$ and functions that agree with it to that order.

\begin{Theorem}[Ratio forms with a $1/x$ correction, on rays]
\label{thm:cfamily}
Fix $r > 1$ and let
\[
c_r = \frac{8}{405} + \frac{4r}{405\,(1+r)^2} \, .
\]
Let $g(x) = x - \tfrac{1}{3} + c/x + O(x^{-2})$ as $x \to \infty$.  If $c < c_r$, then
$g(\alpha)/(g(\alpha) + g(r\alpha)) < \nu_\textrm{B}(\alpha, r\alpha)$ for all sufficiently large $\alpha$.
If $c > c_r$, the reverse inequality holds for all sufficiently large $\alpha$.
\end{Theorem}

Since $0 < 4r/(1+r)^2 < 1$ for $r > 1$, we have $\tfrac{8}{405} < c_r < \tfrac{9}{405} = \tfrac{1}{45}$.
So $c \le \tfrac{8}{405}$ gives a lower bound on every ray, $c \ge \tfrac{1}{45}$ an upper bound on every
ray, and $c$ strictly between gives one or the other depending on the ray.  The theorem covers several forms at once:
\begin{itemize}
\item $g = \nu_\Gamma$, with $c = \tfrac{8}{405}$:  a lower bound, consistent with
Theorem~\ref{thm:lower_bound};
\item every partial sum of the asymptotic series of $\nu_\Gamma$ beyond the first, also with
$c = \tfrac{8}{405}$:  lower bounds;
\item Kerman's $g(x) = x - \tfrac{1}{3}$, with $c = 0$:  a lower bound, consistent with
Theorem~\ref{thm:kerman};
\item the Wilson--Hilferty $\nu_{\Gamma2}(x) = x - \tfrac{1}{3} + \tfrac{1}{27x} - \tfrac{1}{729x^2}$,
with $c = \tfrac{1}{27}$:  an upper bound, which is Theorem~\ref{thm:whray};
\item $\nu_{\Gamma3}(x) = (x - \tfrac{1}{6})^2/x = x - \tfrac{1}{3} + \tfrac{1}{36x}$, with
$c = \tfrac{1}{36}$:  an upper bound, so $\nu_{\textrm{B}3}$ is an asymptotic upper bound on every ray;
\item displaced forms such as $x - \tfrac{1}{3} + c/(x + d)$, which differ from
$x - \tfrac{1}{3} + c/x$ only at order $x^{-2}$.
\end{itemize}

\begin{proof}
Write $a = 1/\alpha$ and $b = 1/\beta = a/r$, so $\Delta = a - b$.  For large $x$, $\log g(x)$ expands as
$\log x + \log\big(1 - \tfrac{1}{3x} + \tfrac{c}{x^2} + O(x^{-3})\big)
= \log x - \tfrac{1}{3x} + \big(c - \tfrac{1}{18}\big)\tfrac{1}{x^2} + O(x^{-3})$, so along the ray
\[
\log\frac{g(\alpha)}{g(\beta)} = \log\frac{\alpha}{\beta} - \frac{\Delta}{3}
+ \Big(c - \frac{1}{18}\Big)(a^2 - b^2) + O(\alpha^{-3}) \, .
\]
In terms of $a$ and $b$, the last term of \eqref{eq:largeexp} is
$-\Delta\,(29a^2 + 50ab + 29b^2)/(810(a+b))$, and $R = O(\alpha^{-3})$ by Theorem~\ref{thm:largeexp}.
Subtracting, and using $45(a+b)^2 - (29a^2 + 50ab + 29b^2) = 16(a+b)^2 + 8ab$,
\begin{align*}
&\operatorname{logit}\nu_\textrm{B} - \log\frac{g(\alpha)}{g(\beta)} \\
&\quad = \Delta\left[\Big(\frac{8}{405} - c\Big)(a+b) + \frac{4ab}{405\,(a+b)}\right] + O(\alpha^{-3}) \, .
\end{align*}
With $b = a/r$, we have $ab/(a+b)^2 = r/(1+r)^2$, so the bracket is $(a+b)(c_r - c)$, and the right
side is $a^2(1 - r^{-2})(c_r - c) + O(a^3)$.  For large $\alpha$, both $g(\alpha)$ and $g(\beta)$ are
positive, so the ratio form lies in $(0,1)$ and its odds are $g(\alpha)/g(\beta)$.  Its logit is below
that of $\nu_\textrm{B}$ when $c < c_r$, and above it when $c > c_r$, for all sufficiently large $\alpha$.
\end{proof}

At $c = \tfrac{1}{45}$ the second-order term vanishes as $r \to 1$, so near the diagonal the sign is
decided at third order; numerically, $c = \tfrac{1}{45}$ fails to be an upper bound in a thin region near
the diagonal ($\beta/\alpha < 1.51$ at $\alpha = 100$) that shrinks as $\alpha$ grows.

\subsection{The \texorpdfstring{$\nu_{\textrm{B}4}$}{nu B4} conjecture}
\label{app:nuB4}

Here $\nu_{\textrm{B}4}(\alpha,\beta) = (2^{1/\beta}-1)/(2^{1/\alpha} + 2^{1/\beta} - 2)$, the ratio form
of Lemma~\ref{lem:special}.  It equals $\nu_\textrm{B}$ on the lines $\alpha = 1$, $\beta = 1$, and $\alpha = \beta$.

\begin{Conjecture}[$\nu_{\textrm{B}4}$ bounding behavior]
\label{thm:B4bound}
For all $\alpha, \beta > 0$, the sign of $\nu_{\textrm{B}4}(\alpha,\beta) - \nu_\textrm{B}(\alpha,\beta)$ is
the sign of $(\alpha-1)(\beta-1)(\alpha-\beta)$.  Equivalently, for $\alpha < \beta$:
$\nu_{\textrm{B}4} > \nu_\textrm{B}$ if $\alpha < 1 < \beta$, and $\nu_{\textrm{B}4} < \nu_\textrm{B}$ if
$1 < \alpha < \beta$ or $\alpha < \beta < 1$.
\end{Conjecture}

\emph{Evidence.}  The sign pattern is visible in Figure~\ref{fig:contour1}.  On a logarithmic grid of
300 values over $[0.02, 200]$, all $44{,}850$ pairs $\alpha < \beta$ give the conjectured sign.
For large shape parameters the conjecture is a theorem, proved next.

\begin{Theorem}[$\nu_{\textrm{B}4}$ conjecture for large shape parameters]
\label{thm:B4large}
Let $x_1 = (\log 2)^2/(3\log 2 - 2) \approx 6.048$.  For $x_1 \le \alpha < \beta$,
\[
\nu_{\textrm{B}4}(\alpha,\beta) < \nu_{\textrm{B}1}(\alpha,\beta) < \nu_\textrm{B}(\alpha,\beta) \, .
\]
So Conjecture~\ref{thm:B4bound} holds whenever $\min(\alpha,\beta) \ge x_1$.
\end{Theorem}

\begin{proof}
The second inequality is Theorem~\ref{thm:kerman}.  For the first, compare odds.  The odds of
$\nu_{\textrm{B}1}$ are $(\alpha - \tfrac{1}{3})/(\beta - \tfrac{1}{3})$ and those of $\nu_{\textrm{B}4}$ are
$\nu_{\Gamma4}(\alpha)/\nu_{\Gamma4}(\beta)$, with $\nu_{\Gamma4}(x) = \log 2/(2^{1/x} - 1)$.  So it
suffices that
\[
g(x) = \log \nu_{\Gamma4}(x) - \log\left(x - \tfrac{1}{3}\right)
\]
is strictly increasing on $[x_1, \infty)$.  Its derivative is
\[
g'(x) = \frac{\log 2}{x^2\,(1 - 2^{-1/x})} - \frac{1}{x - \frac{1}{3}} \, .
\]
With $t = (\log 2)/x$, the condition $g'(x) > 0$ becomes
\[
\frac{1 - e^{-t}}{t} < 1 - \frac{t}{3\log 2} \, .
\]
For $t > 0$, $e^{-t} > 1 - t + t^2/2 - t^3/6$, so the left side is less than $1 - t/2 + t^2/6$.  That is
at most the right side when $t/6 \le \tfrac{1}{2} - 1/(3\log 2)$, that is when $t \le 3 - 2/\log 2$, which
is positive.  In terms of $x$, this is $x \ge x_1$.
\end{proof}

The first inequality involves only elementary functions, not the beta median.  All the analysis is in
Theorem~\ref{thm:kerman}, which rests on Theorem~\ref{thm:lower_bound} and the two facts from
\citet{berg2006chen}.  The reason it works is a difference in order.  As $x \to \infty$,
\[
\nu_{\Gamma4}(x) = x - \tfrac{1}{2} \log 2 + \frac{(\log 2)^2}{12x} + O(x^{-3}) \, ,
\]
while $\nu_\Gamma(x) = x - \tfrac{1}{3} + O(x^{-1})$, and $\tfrac{1}{2} \log 2 > \tfrac{1}{3}$.  So
$\nu_{\textrm{B}4}$ errs at first order, by about $(\tfrac{1}{2} \log 2 - \tfrac{1}{3})(1/\alpha - 1/\beta)
\approx 0.0132\,(1/\alpha - 1/\beta)$ in the logit, while Kerman's $\nu_{\textrm{B}1}$ errs only at
second order (Section~\ref{app:large}).  For large shape parameters, then, $\nu_{\textrm{B}4}$ lies below a bound that is already proved.
Numerically $g$ decreases up to $x \approx 5.874$ and increases after it, so the exact threshold for
the first inequality is about $5.874$; $x_1$ is the one the two-term estimate gives.  The condition
$g(\alpha) < g(\beta)$ also holds for every $\alpha > 3.021$ once $\beta$ is large enough, since
$g < 0 = g(\infty)$ there.

The theorem does not reach the regions near shape parameter~1, where the conjecture's sign changes, or the
region $\alpha < \beta < 1$.  In particular it does not settle the inequality left open in Lean (see
Section~\ref{app:lean}), which is stated for all shape parameters.

The median odds give the conjecture a simple form.  By Lemma~\ref{lem:basic},
$m(\alpha,1) = 1/(2^{1/\alpha}-1) = \nu_{\Gamma4}(\alpha)/\log 2$, $m(y,x) = 1/m(x,y)$, and $m(x,x) = 1$.
So the odds of the approximation are a product of two exact median odds through the intermediate
shape parameter~1:
\[
\frac{\nu_{\textrm{B}4}(\alpha,\beta)}{1-\nu_{\textrm{B}4}(\alpha,\beta)} = m(\alpha,1)\, m(1,\beta) \, .
\]

\begin{Lemma}[Triangle form]
\label{lem:triangle}
Conjecture~\ref{thm:B4bound} holds if and only if
\[
m(x,z) < m(x,y)\, m(y,z)
\]
for all $0 < x < y < z$ with $1 \in \{x,y,z\}$.
\end{Lemma}

\begin{proof}
Write $L(x,y) = \log m(x,y)$, so that $L(y,x) = -L(x,y)$ and $L(x,x) = 0$.  The error
$\nu_{\textrm{B}4}(\alpha,\beta) - \nu_\textrm{B}(\alpha,\beta)$ has the sign of the difference of the
odds, $m(\alpha,1)\,m(1,\beta) - m(\alpha,\beta)$, and so the sign of
\[
D(\alpha,\beta) = L(\alpha,1) + L(1,\beta) - L(\alpha,\beta) \, .
\]
On the lines $\alpha = 1$, $\beta = 1$, and $\alpha = \beta$ we have $D = 0$, matching the zero of
$(\alpha-1)(\beta-1)(\alpha-\beta)$.  Swapping $\alpha$ and $\beta$ changes the sign of both $D$ and the
product, so it suffices to take $\alpha < \beta$.  With the triangle defect
$K(x,y,z) = L(x,y) + L(y,z) - L(x,z)$ for $x < y < z$:
\begin{itemize}
\item if $\alpha < 1 < \beta$, then $D = K(\alpha,1,\beta)$, and the product is positive;
\item if $1 < \alpha < \beta$, then $D = -K(1,\alpha,\beta)$, and the product is negative;
\item if $\alpha < \beta < 1$, then $D = -K(\alpha,\beta,1)$, and the product is negative.
\end{itemize}
So the conjecture holds exactly when $K > 0$ for every increasing triple containing 1, which is the
stated inequality.
\end{proof}

The three cases $y = 1$, $x = 1$, and $z = 1$ are the regions $\alpha < 1 < \beta$, $1 < \alpha < \beta$,
and $\alpha < \beta < 1$ of the conjecture.  Numerically, $m(x,z) < m(x,y)\, m(y,z)$ appears to hold for
all $0 < x < y < z$, not only when one of them is~1:  there was no exception in 400 random triples
spread log-uniformly over $[0.01, 200]$.  That stronger statement would in turn follow if
$z \mapsto m(u_2,z)/m(u_1,z)$ were increasing on $z \ge u_2$, for each $0 < u_1 < u_2$, since that ratio
equals $1/m(u_1,u_2)$ at $z = u_2$.

Theorem~\ref{thm:lower_bound} settles part of the conjecture.

\begin{Lemma}[Partial result]
\label{lem:B4partial}
If $0 < \alpha < \beta$ and
\[
\frac{\nu_{\Gamma4}(\alpha)}{\nu_\Gamma(\alpha)} \le \frac{\nu_{\Gamma4}(\beta)}{\nu_\Gamma(\beta)} \, ,
\]
then $\nu_{\textrm{B}4}(\alpha,\beta) < \nu_\textrm{B}(\alpha,\beta)$.
\end{Lemma}

\begin{proof}
By Theorem~\ref{thm:lower_bound} and the rearranged hypothesis,
\[
m(\alpha,\beta) > \frac{\nu_\Gamma(\alpha)}{\nu_\Gamma(\beta)} \ge \frac{\nu_{\Gamma4}(\alpha)}{\nu_{\Gamma4}(\beta)} \, ,
\]
and the right side is the odds of $\nu_{\textrm{B}4}(\alpha,\beta)$.  So the odds of $\nu_\textrm{B}$
exceed those of $\nu_{\textrm{B}4}$.
\end{proof}

Numerically, $\nu_{\Gamma4}/\nu_\Gamma$ decreases on $0 < \alpha \lesssim 2.08$ and increases after.  So
Lemma~\ref{lem:B4partial} covers the region $2.08 \lesssim \alpha < \beta$ and, for each $\alpha$ in
$(1, 2.08)$, all sufficiently large $\beta$ (for example $\beta \ge 7.44$ when $\alpha = 1.2$).

Where $\nu_{\textrm{B}4}$ overestimates the median, by Conjecture~\ref{thm:B4bound} only for $\alpha < 1 < \beta$, the
overestimate is bounded, whether or not the conjecture holds.

\begin{Theorem}[Bounded overestimate]
\label{thm:B4over}
Let $c_0 = (\tfrac{1}{2}\log 2 - \tfrac{1}{3})\log 2 \approx 0.0091774$.  For $0 < \alpha \le \beta$,
\[
\nu_{\textrm{B}4}(\alpha,\beta) < \frac{e^\gamma \log 2}{1 - c_0}\,\nu_\textrm{B}(\alpha,\beta)
\approx 1.24598\,\nu_\textrm{B}(\alpha,\beta) \, .
\]
The factor cannot be reduced below $e^\gamma\log 2 \approx 1.23455$:  as $\alpha \to 0$ with $\beta$ fixed,
$\nu_{\textrm{B}4}/\nu_\textrm{B} \to (2^{1/\beta} - 1)\,e^{\gamma + \psi_0(\beta)}$, which tends to $e^\gamma\log 2$ as
$\beta \to \infty$.
\end{Theorem}

\begin{proof}
Write $a = \nu_\Gamma(\alpha)$, $a_4 = \nu_{\Gamma4}(\alpha)$, $b = \nu_\Gamma(\beta)$, and $b_4 = \nu_{\Gamma4}(\beta)$.  By
Theorem~\ref{thm:lower_bound}, $\nu_\textrm{B} \ge a/(a + b)$, and $\nu_{\textrm{B}4} = a_4/(a_4 + b_4)$, so
\[
\frac{\nu_{\textrm{B}4}}{\nu_\textrm{B}} \le \frac{a_4}{a} \cdot \frac{a + b}{a_4 + b_4}
\le \frac{a_4}{a} \max\left(\frac{a}{a_4}, \frac{b}{b_4}\right) = \max\left(1, \frac{a_4}{a} \cdot \frac{b}{b_4}\right) ,
\]
since a ratio of two sums lies between the ratios of their terms.  Theorem~2 of \citet{lyon2026gamma}, which is
machine-checked there, gives
\[
\frac{e^{-\gamma}}{2^{1/x} - 1 - c_0/x^2} < \nu_\Gamma(x) < \frac{\log 2}{2^{1/x} - 1 - c_0/x^2}
\]
for all $x > 0$, with a positive denominator.  Since $c_0 > 0$, the lower bound exceeds
$e^{-\gamma}/(2^{1/x} - 1) = (e^{-\gamma}/\log 2)\,\nu_{\Gamma4}(x)$, so $a_4/a < e^\gamma\log 2$.  If $\beta \le 1$,
then $b \le b_4$ by Theorem~\ref{thm:gamma}.  If $\beta > 1$, the upper bound gives
\[
\frac{b}{b_4} < \frac{1}{1 - c_0/\left(\beta^2(2^{1/\beta} - 1)\right)} < \frac{1}{1 - c_0} \, ,
\]
because $\beta^2(2^{1/\beta} - 1) > 1$:  with $x = 1/\beta \in (0,1)$, the function $2^x - 1 - x^2$ is concave (its
second derivative is $2^x\log^2 2 - 2 < 0$) and vanishes at $x = 0$ and $x = 1$.  In either case
$(a_4/a)(b/b_4) < e^\gamma\log 2/(1 - c_0)$, which exceeds 1.

For the limit, as $\alpha \to 0$, $2^{1/\alpha}\,\nu_{\textrm{B}4}(\alpha,\beta) \to 2^{1/\beta} - 1$, and
$2^{1/\alpha}\,\nu_\textrm{B}(\alpha,\beta) \to e^{-\gamma - \psi_0(\beta)}$ by Theorem~\ref{thm:lowasym}.  As
$\beta \to \infty$, $2^{1/\beta} - 1 \sim (\log 2)/\beta$ and $e^{\psi_0(\beta)} \sim \beta$.
\end{proof}

The proof also shows that $\nu_{\Gamma4} < e^\gamma\log 2\,\nu_\Gamma$ for all $\alpha > 0$, which is sharp as
$\alpha \to 0$.  Numerically, $\nu_{\textrm{B}4}/\nu_\textrm{B}$ stays below its limit $e^\gamma\log 2$, so the largest
overestimate is about $23.45\%$, approached as $\alpha \to 0$ and $\beta \to \infty$.

A second partial result brackets the beta median between two elementary expressions.

\begin{Lemma}[Logit bracket]
\label{lem:bracket}
For $0 < \alpha < \beta$,
\[
\frac{e^{\psi_0(\alpha)}}{e^{\psi_0(\alpha)} + e^{\psi_0(\beta)}} < \nu_\textrm{B}(\alpha,\beta) < \frac{\alpha}{\alpha+\beta} \, .
\]
\end{Lemma}

\begin{proof}
Let $W = \log(X_\alpha/X_\beta)$.  Then $X = X_\alpha/(X_\alpha+X_\beta)$ has the beta distribution and
$W = \operatorname{logit} X$, so the median of $W$ is $\mu = \operatorname{logit} \nu_\textrm{B}$, and
$\mathbb{E}W = \psi_0(\alpha) - \psi_0(\beta)$ since $\mathbb{E}\log X_\alpha = \psi_0(\alpha)$.  The density
of $W$ is $f(w) = e^{\alpha w}(1+e^w)^{-\alpha-\beta}/B(\alpha,\beta)$; let $F$ be its CDF and
$\sigma(w) = 1/(1+e^{-w})$.  In terms of $W$ the two claims are
\[
\psi_0(\alpha) - \psi_0(\beta) < \mu < \log(\alpha/\beta) \, .
\]

For a center $c$, pair the points $c \pm t$ through
\[
Q_c(t) = F(c+t) + F(c-t) - 1 \, , \quad t \ge 0 \, .
\]
Then $Q_c(0) = 2F(c) - 1$, $Q_c(t) \to 0$ as $t \to \infty$, and $Q_c'(t) = f(c+t) - f(c-t)$ has the
sign of
\begin{align*}
h_c(t) &= \log\frac{f(c+t)}{f(c-t)} \\
&= 2\alpha t - (\alpha+\beta)\left[\log(1+e^{c+t}) - \log(1+e^{c-t})\right] .
\end{align*}
We have $h_c(0) = 0$ and
\begin{align*}
h_c'(t) &= 2\alpha - (\alpha+\beta)\left[\sigma(c+t) + \sigma(c-t)\right] , \\
h_c''(t) &= -(\alpha+\beta)\left[\sigma'(c+t) - \sigma'(c-t)\right] .
\end{align*}
The function $\sigma' = \sigma(1-\sigma)$ is even and decreasing in $|w|$, and for $c < 0 < t$ we have
$|c+t| < t - c = |c-t|$, so $h_c'' < 0$:  for $c < 0$, $h_c$ is strictly concave on $[0,\infty)$.

\emph{Upper bound.}  Take $c = \log(\alpha/\beta) < 0$.  Then $\sigma(c) = \alpha/(\alpha+\beta)$, so
$h_c'(0) = 0$, and concavity gives $h_c' < 0$ and $h_c < 0$ on $(0,\infty)$.  Thus $Q_c$ is strictly
decreasing to the limit $0$, so $Q_c(0) > 0$, that is $F(c) > \tfrac{1}{2}$, and $\mu < c$.

\emph{Lower bound.}  Take $c = \mu$.  By the upper bound $\mu < \log(\alpha/\beta) < 0$, so $h_\mu$ is
strictly concave on $[0, \infty)$, and $\sigma(\mu) < \alpha/(\alpha+\beta)$ gives
$h_\mu'(0) = 2\alpha - 2(\alpha+\beta)\sigma(\mu) > 0$.  Also $h_\mu(t) = (\alpha - \beta)t + O(1) \to -\infty$.
So $h_\mu$ has exactly one zero $t_0 > 0$, with $h_\mu > 0$ on $(0,t_0)$ and $h_\mu < 0$ on
$(t_0,\infty)$.  Hence $Q_\mu$ rises on $[0,t_0]$ and falls on $[t_0,\infty)$.  As $Q_\mu(0) = 2F(\mu) - 1 = 0$
and $Q_\mu(\infty) = 0$, we get $Q_\mu(t) > 0$ for every $t > 0$:
\[
\mathbb{P}(W \le \mu - t) > \mathbb{P}(W > \mu + t) \quad \textrm{for all} \quad t > 0 \, ,
\]
so the logit is skewed to the left about its median.  Integrating,
\begin{align*}
\mu - \mathbb{E}W &= \int_0^\infty \left[ \mathbb{P}(W \le \mu - t) - \mathbb{P}(W > \mu + t) \right] dt \\
&= \int_0^\infty Q_\mu(t)\,dt > 0 \, ,
\end{align*}
so $\mu > \psi_0(\alpha) - \psi_0(\beta)$.  Applying $\sigma$ to both bounds gives the lemma.
\end{proof}

In logit terms, the median of $W$ lies strictly between its mean and its mode, $\log(\alpha/\beta)$.
For $1 < \alpha < \beta$ the upper bound is the classical mode--median--mean ordering for the beta
distribution itself \citep{groeneveld1977mode,payton1989bounds}, as noted by
\citet{kerman2011closed}; the proof here covers all $0 < \alpha < \beta$.  Numerically,
$\nu_{\textrm{B}4}$ also lies strictly inside this bracket, so the bracket alone cannot decide
Conjecture~\ref{thm:B4bound}.

\subsection{The low-\texorpdfstring{$\alpha$}{alpha} asymptote and elementary bounds}
\label{app:lowalpha}

Write
\begin{gather*}
P(\alpha,\beta) = \left(\frac{2\beta}{\alpha+\beta}\right)^{-1/\alpha} , \\
M_\beta(\alpha) = \left(\frac{\Gamma(\alpha+1)\,\Gamma(\beta+1)}{\Gamma(\alpha+\beta+1)}\right)^{1/\alpha} ,
\end{gather*}
so that
\begin{align*}
\nu_{\textrm{B}5}(\alpha,\beta) &= \left(\frac{\alpha B(\alpha,\beta)}{2}\right)^{1/\alpha} \\
&= 2^{-1/\alpha}\left(\frac{\Gamma(\alpha+1)\,\Gamma(\beta)}{\Gamma(\alpha+\beta)}\right)^{1/\alpha}
= P(\alpha,\beta)\, M_\beta(\alpha) \, .
\end{align*}

\begin{Theorem}[Beta median bounds]
\label{thm:low}
For $\alpha > 0$:  $\nu_{\textrm{B}5}(\alpha,\beta) < \nu_\textrm{B}(\alpha,\beta)$ if $\beta > 1$,
$\nu_{\textrm{B}5}(\alpha,\beta) > \nu_\textrm{B}(\alpha,\beta)$ if $\beta < 1$, and the two are equal if
$\beta = 1$.
\end{Theorem}

\begin{proof}
Let $\nu = \nu_\textrm{B}(\alpha,\beta)$, so $\int_0^\nu x^{\alpha-1}(1-x)^{\beta-1}\,dx = \tfrac{1}{2} B(\alpha,\beta)$,
while by definition $\nu_{\textrm{B}5}^\alpha/\alpha = \tfrac{1}{2} B(\alpha,\beta)$.  On $(0,\nu)$,
$(1-x)^{\beta-1}$ is less than $1$ if $\beta > 1$, greater than $1$ if $\beta < 1$, and equal to $1$
if $\beta = 1$.  So $\int_0^\nu x^{\alpha-1}\,dx = \nu^\alpha/\alpha$ is greater than, less than, or equal to
$\tfrac{1}{2} B(\alpha,\beta) = \nu_{\textrm{B}5}^\alpha/\alpha$ in the three cases, and $y \mapsto y^\alpha$ is
strictly increasing.
\end{proof}

The same comparison, with $(1-x)^{\beta-1}$ bounded on $(0,\nu)$ by its value at $x = \nu$, gives the
size of the error:
\begin{equation}\label{eq:lowsand}
\left|\log\frac{\nu_\textrm{B}(\alpha,\beta)}{\nu_{\textrm{B}5}(\alpha,\beta)}\right|
\le \frac{|\beta - 1|}{\alpha}\,\log\frac{1}{1 - \nu_\textrm{B}(\alpha,\beta)} \, .
\end{equation}

\begin{Theorem}[Low-$\alpha$ asymptote]
\label{thm:lowasym}
For fixed $\beta > 0$, as $\alpha \to 0^+$, $\nu_\textrm{B}(\alpha,\beta)/\nu_{\textrm{B}5}(\alpha,\beta) - 1$ tends to
$0$ faster than any power of $\alpha$, and
\[
\lim_{\alpha \to 0^+} 2^{1/\alpha}\,\nu_\textrm{B}(\alpha,\beta) = e^{-\gamma - \psi_0(\beta)} \, , \qquad
\lim_{\alpha \to 0^+} M_\beta(\alpha) = e^{-\gamma - \psi_0(\beta+1)} \, .
\]
\end{Theorem}

\begin{proof}
Write $\nu_{\textrm{B}5} = 2^{-1/\alpha} G(\alpha)$, with
\[
\log G(\alpha) = \frac{\log\Gamma(1+\alpha) + \log\Gamma(\beta) - \log\Gamma(\alpha+\beta)}{\alpha} \, ,
\]
which has a finite limit as $\alpha \to 0^+$; see below.  So $\nu_{\textrm{B}5} = O(2^{-1/\alpha})$.
For $\alpha < \beta$, Lemma~\ref{lem:bracket} gives $\nu_\textrm{B} < \alpha/(\alpha+\beta) < \tfrac{1}{2}$, and
since $x/(1-x)$ is increasing,
\[
\log\frac{1}{1 - \nu_\textrm{B}} \le \frac{\nu_\textrm{B}}{1 - \nu_\textrm{B}} \le \min\left(\frac{\alpha}{\beta}, \, 2\nu_\textrm{B}\right) .
\]
With the first of these, \eqref{eq:lowsand} gives $|\log(\nu_\textrm{B}/\nu_{\textrm{B}5})| \le |\beta - 1|/\beta$, so
$\nu_\textrm{B} \le e^{|\beta-1|/\beta}\,\nu_{\textrm{B}5} = O(2^{-1/\alpha})$.  With the second, it then gives
$|\log(\nu_\textrm{B}/\nu_{\textrm{B}5})| \le 2|\beta - 1|\,\nu_\textrm{B}/\alpha = O(2^{-1/\alpha}/\alpha)$, which is
smaller than any power of $\alpha$.

Next, the numerator of $\log G$ vanishes at $\alpha = 0$, so Taylor's theorem gives
$\log G(\alpha) = \psi_0(1) - \psi_0(\beta) + \tfrac{1}{2} \left(\psi_1(1) - \psi_1(\beta)\right)\alpha + O(\alpha^2)$,
with $\psi_0(1) = -\gamma$ and $\psi_1(1) = \pi^2/6$.  Hence $2^{1/\alpha}\nu_{\textrm{B}5} = G(\alpha) \to e^{-\gamma-\psi_0(\beta)}$.
The same computation for $\log M_\beta$, with $\beta + 1$ in place of $\beta$, gives its limit.
\end{proof}

\begin{Lemma}[Shape of the log factor]
\label{lem:logconc}
For $\beta > 0$, the function $\ell(\alpha) = \log M_\beta(\alpha)$ satisfies $\ell'' < 0$ and $\ell''' > 0$ on $[0,\infty)$,
with
\begin{gather*}
\ell(0) = -\gamma - \psi_0(\beta+1) \, , \qquad
\ell'(0) = \frac{\pi^2}{12} - \frac{\psi_1(\beta+1)}{2} \, , \\
\ell''(0) = \frac{\psi_2(1) - \psi_2(\beta+1)}{3} = -\frac{2\zeta(3) - |\psi_2(\beta+1)|}{3} \, .
\end{gather*}
\end{Lemma}

\begin{proof}
With $\phi(t) = \psi_0(1+t) - \psi_0(1+\beta+t)$, we have $\ell(\alpha) = \frac{1}{\alpha} \int_0^\alpha \phi(t)\,dt
= \int_0^1 \phi(\alpha s)\,ds$, which is smooth on $[0,\infty)$, and
$\ell^{(n)}(\alpha) = \int_0^1 s^n \phi^{(n)}(\alpha s)\,ds$.  Here
$\phi''(t) = \psi_2(1+t) - \psi_2(1+\beta+t) < 0$ and $\phi'''(t) = \psi_3(1+t) - \psi_3(1+\beta+t) > 0$,
because $\psi_2$ is negative and increasing and $\psi_3$ is positive and decreasing.  The values at $0$ are
$\ell^{(n)}(0) = \phi^{(n)}(0)/(n+1)$, with $\psi_0(1) = -\gamma$, $\psi_1(1) = \pi^2/6$, and $\psi_2(1) = -2\zeta(3)$.
\end{proof}

Because $\ell''' > 0$, the second-order Taylor polynomial of $\ell$ is a lower bound:
\begin{equation}\label{eq:taylor2}
\ell(\alpha) \ge \ell(0) + \ell'(0)\,\alpha + \tfrac{1}{2} \ell''(0)\,\alpha^2 \quad \textrm{for} \quad \alpha \ge 0 \, ,
\end{equation}
since $\ell''(t) \ge \ell''(0)$ in the remainder $\int_0^\alpha (\alpha - t)\,\ell''(t)\,dt$.  Because $\ell'' < 0$, the
first-order polynomial is an upper bound.  With Theorem~\ref{thm:low}, and elementary one-sided bounds on
$\psi_0$, $\psi_1$, $\psi_2$, these give elementary bounds for the beta median.  The partial sums of the
asymptotic series of the polygamma functions alternate around them \citep{alzer1997inequalities}.  For $z > 0$,
we use
\begin{gather*}
\log z - \tfrac{1}{2z} - \tfrac{1}{12z^2} < \psi_0(z) < \log z - \tfrac{1}{2z} - \tfrac{1}{12z^2} + \tfrac{1}{120z^4} \, , \\
\tfrac{1}{z} + \tfrac{1}{2z^2} + \tfrac{1}{6z^3} - \tfrac{1}{30z^5} < \psi_1(z) \, , \\
\psi_1(z) < \tfrac{1}{z} + \tfrac{1}{2z^2} + \tfrac{1}{6z^3} - \tfrac{1}{30z^5} + \tfrac{1}{42z^7} \, , \\
|\psi_2(z)| > \tfrac{1}{z^2} + \tfrac{1}{z^3} + \tfrac{1}{2z^4} - \tfrac{1}{6z^6} \, .
\end{gather*}
We apply them at $w = \beta + 2$ and move to $z = \beta + 1$ with the recurrences
$\psi_0(z) = \psi_0(w) - 1/z$, $\psi_1(z) = \psi_1(w) + 1/z^2$, and $\psi_2(z) = \psi_2(w) - 2/z^3$; the
bounds are much tighter at $w$ than at $z$ when $\beta$ is small.  With $z = \beta + 1$ and $w = \beta + 2$,
define
\begin{align*}
\underline{c}_0(\beta) &= \frac{e^{-\gamma + 1/z}}{w}\, e^{\frac{1}{2w} + \frac{1}{12w^2} - \frac{1}{120w^4}} \, , \\
\underline{k}(\beta) &= \frac{\pi^2}{12} - \frac{1}{2z^2} - \frac{1}{2}\left(\frac{1}{w} + \frac{1}{2w^2} + \frac{1}{6w^3} - \frac{1}{30w^5}\right) - \frac{1}{84w^7} \, , \\
\eta(\beta) &= \frac{\zeta(3)}{3} - \frac{1}{6}\left(\frac{2}{z^3} + \frac{1}{w^2} + \frac{1}{w^3} + \frac{1}{2w^4} - \frac{1}{6w^6}\right) ,
\end{align*}
and
\[
\nu_{\textrm{B}6}(\alpha,\beta) = P(\alpha,\beta)\,\underline{c}_0(\beta)\,e^{\underline{k}(\beta)\,\alpha - \eta(\beta)\,\alpha^2} \, .
\]
For the upper bound let
\[
\overline{c}_0(\beta) = \frac{e^{-\gamma + 1/z}}{w}\,e^{\frac{1}{2w} + \frac{1}{12w^2}} \, , \qquad
\overline{k}(\beta) = \frac{\pi^2}{12} - \frac{1}{2z^2} - \frac{1}{2} \left(\frac{1}{w} + \frac{1}{2w^2} + \frac{1}{6w^3} - \frac{1}{30w^5}\right) ,
\]
and
\[
\nu_{\textrm{B}7}(\alpha,\beta) = P(\alpha,\beta)\,\overline{c}_0(\beta)\,e^{\overline{k}(\beta)\,\alpha} \, .
\]

\begin{Theorem}[Elementary low-$\alpha$ bounds]
\label{thm:lowbounds}
(a) For all $\alpha, \beta > 0$, $\nu_{\textrm{B}6}(\alpha,\beta) < \nu_{\textrm{B}5}(\alpha,\beta)$.  Hence
$\nu_{\textrm{B}6}(\alpha,\beta) < \nu_\textrm{B}(\alpha,\beta)$ for $\alpha > 0$, $\beta > 1$.

(b) For all $\alpha, \beta > 0$, $\nu_{\textrm{B}5}(\alpha,\beta) < \nu_{\textrm{B}7}(\alpha,\beta)$.  Hence
$\nu_\textrm{B}(\alpha,\beta) < \nu_{\textrm{B}7}(\alpha,\beta)$ for $\alpha > 0$, $0 < \beta < 1$.
\end{Theorem}

\begin{proof}
(a) The upper bounds on $\psi_0(w)$ and $\psi_1(w)$, with the recurrences, give $\ell(0) > \log \underline{c}_0$ and $\ell'(0) > \underline{k}$, and the lower bound on
$|\psi_2(w)|$, with $|\psi_2(z)| = |\psi_2(w)| + 2/z^3$, gives $-\tfrac{1}{2} \ell''(0) < \eta$.  By \eqref{eq:taylor2},
$\ell(\alpha) > \log \underline{c}_0 + \underline{k}\alpha - \eta\alpha^2$, so $\nu_{\textrm{B}6} < Pe^{\ell} = \nu_{\textrm{B}5}$, and
$\nu_{\textrm{B}5} < \nu_\textrm{B}$ for $\beta > 1$ by Theorem~\ref{thm:low}.
(b) By concavity, $\ell(\alpha) \le \ell(0) + \ell'(0)\,\alpha$, and the lower bounds on $\psi_0(w)$ and $\psi_1(w)$ give
$\ell(0) < \log\overline{c}_0$ and $\ell'(0) < \overline{k}$.  So $\nu_{\textrm{B}5} < \nu_{\textrm{B}7}$, and
$\nu_\textrm{B} < \nu_{\textrm{B}5}$ for $\beta < 1$.
\end{proof}

\emph{Notes and evidence.}  The coefficient $\eta(\beta)$ rises from about $0.333$ at $\beta = 1$ to
$\zeta(3)/3 \approx 0.4007$ as $\beta \to \infty$.  On logarithmic grids over $0.002 \le \alpha \le 0.995$
and $1.005 \le \beta \le 400$, the relative error of $\nu_{\textrm{B}6}$ ranges from about $0$ (as
$\alpha \to 0$) to $-38.0\%$.  The worst case is at $\alpha \to 1$ with large $\beta$, where
$\nu_{\textrm{B}5}$ itself is $28\%$ low, since $\nu_{\textrm{B}5}(1,\beta) = 1/(2\beta)$ while
$\nu_\textrm{B}(1,\beta) \approx \log 2/\beta$.

The two bounds differ only slightly in their coefficients:  $\overline{c}_0 = \underline{c}_0\,e^{1/(120w^4)}$ and
$\overline{k} = \underline{k} + 1/(84w^7)$.  But $\nu_{\textrm{B}7}$ has no $\alpha^2$ term.  Since $\ell''' > 0$, a
second-order Taylor polynomial bounds $\ell$ from below, so only the first-order one gives an upper bound.

For $\alpha < \beta < 1$, $\nu_{\textrm{B}7}$ is poor near the diagonal when $\beta$ is small, because
on the diagonal $\nu_{\textrm{B}5}$ tends to~$1$ as $\beta \to 0$.  An upper-bound form that equals $\tfrac{1}{2}$ on the
diagonal must lie below $\nu_{\textrm{B}5}$ there, so it cannot be obtained by this comparison; the
odds-corrected bounds of Section~\ref{sec:odds} address this, with numerical rather than proved bounding.

\subsection{Formal verification in Lean}
\label{app:lean}

Most results of this appendix were formalized in the Lean~4 proof assistant \citep{lean4}, using
its mathematics library Mathlib \citep{mathlib}, and checked by machine;
the exceptions are listed at the end of this section.  The project, in the folder \texttt{lean} of the
supplementary material, has 29 source files and about 5{,}850 lines.  It builds with Lean
\texttt{v4.35.0-rc3} and the matching Mathlib release.  All results are in the namespace
\texttt{MedianBounds}.

\emph{What is being checked.}  The statements are about Mathlib's own probability distributions:
\texttt{gammaMeasure $\alpha$ 1} and \texttt{betaMeasure $\alpha$ $\beta$}, with CDFs \texttt{cdf}.  A median
appears either as a hypothesis \texttt{cdf $\mu$ m = 1/2} or as a function (\texttt{nuGamma},
\texttt{nuB}) defined as the inverse of the CDF at $\tfrac{1}{2}$.  Existence and uniqueness of the medians
are themselves proved, so the hypotheses are never vacuous.  For each main result, the build prints the
axioms it depends on (\texttt{\#print axioms}).  All but one depend only on Lean's three standard
axioms (\texttt{propext}, \texttt{Classical.choice}, \texttt{Quot.sound}), which means no step is
assumed or skipped.

\emph{Correspondence.}  Table~\ref{tab:lean} lists the Lean name of each result.

\begin{table}[!ht]
\centering
\caption{Lean names of the formalized results.}
\label{tab:lean}
\footnotesize
\begin{tabular}{@{}ll@{}}
\toprule
Result & Lean name \\
\midrule
Lemma~\ref{lem:basic} & \texttt{thm\_nuB\_symm}, \texttt{thm\_nuB\_equal}, \texttt{nuB\_one\_right}, \\
 & \texttt{nuB\_one\_left\_eq}, \texttt{cdf\_gamma\_one\_log\_two} \\
Theorem~\ref{thm:asymptote} & \texttt{thm\_nuB0\_asymptote} \\
Theorem~\ref{thm:lower_bound} & \texttt{thm\_nuB0\_bound\_le}, \texttt{thm\_nuB0\_bound} \\
Lemmas~\ref{lem:logmean}, \ref{lem:Lambda} & \texttt{lem\_logmean}, \texttt{lem\_Lambda} \\
Theorem~\ref{thm:gamma} & \texttt{thm\_nuGamma4\_upper}, \texttt{thm\_nuGamma4\_lower} \\
Lemma~\ref{lem:special} & \texttt{lem\_special} \\
Theorem~\ref{thm:non} & \texttt{thm\_no\_g} \\
Theorem~\ref{thm:kerman}$^\dagger$ & \texttt{thm\_nuB1\_bound}, \texttt{lem\_kerman\_ratio} \\
Theorem~\ref{thm:wh_edge} & \texttt{thm\_nuB2\_bound\_one}, \texttt{thm\_nuB2\_edge\_one} \\
Conjecture~\ref{conj:wh}$^\S$ & \texttt{NuB2Conjecture} \\
Lemma~\ref{lem:whmargin}$^\P$ & \texttt{lem\_wh\_margin} \\
Conjecture~\ref{conj:largeR}$^\S$ & \texttt{UniformRemainderConjecture} \\
Conj.~\ref{conj:largeR} $\Rightarrow$ Conj.~\ref{conj:wh} & \texttt{nuB2Conjecture\_of\_uniformRemainder} \\
Lemma~\ref{lem:triangle} & \texttt{nuB4\_conjecture\_of\_triangleAtOne} \\
 & \texttt{triangleAtOne\_of\_nuB4\_conjecture} \\
Lemma~\ref{lem:B4partial} & \texttt{lem\_nuB4\_lt\_nuB\_of\_rho} \\
Theorem~\ref{thm:B4large}$^\dagger$ & \texttt{thm\_nuB4\_large}, \texttt{thm\_nuB4\_bound\_large} \\
Lemma~\ref{lem:bracket} & \texttt{lem\_nuB\_lt\_mean}, \texttt{lem\_nuB\_gt\_digamma} \\
Conjecture~\ref{thm:B4bound}$^*$ & \texttt{thm\_nuB4\_bound} \\
Theorem~\ref{thm:low} & \texttt{thm\_nuB5\_bound} \\
Lemma~\ref{lem:logconc}$^\|$ & \texttt{lem\_logconc\_taylor} \\
Theorem~\ref{thm:lowasym} & \texttt{thm\_lowasym\_ratio}, \texttt{thm\_lowasym\_median}, \\
 & \texttt{thm\_lowasym\_M} \\
Theorem~\ref{thm:lowbounds} & \texttt{thm\_lowbounds\_a}, \texttt{thm\_lowbounds\_b}, \\
 & \texttt{thm\_nuB6\_bound}, \texttt{thm\_nuB7\_bound} \\
\bottomrule
\end{tabular}

\smallskip
\raggedright
$^\dagger$ with two facts from \citet{berg2006chen} as a hypothesis;
$^*$ conditional on one unproved inequality;
$^\S$ stated only;
$^\P$ in the weaker form $E > \Delta/(72\alpha^2)$;
$^\|$ in part (see below).
\end{table}

\emph{How the proofs were adapted.}  The Lean proofs follow the ones above, with a few changes of
route.
\begin{itemize}
\item Theorem~\ref{thm:lower_bound} is proved for any CDF on $(0,\infty)$ with density
$c\,y^{\beta-1}e^{-y}$, packaged as a hypothesis structure, and then applied to Mathlib's gamma CDF\@.
The identity $\mathbb{P}(X_\alpha \le sX_\beta) = I_{s/(1+s)}(\alpha,\beta)$, the link to Mathlib's
beta distribution, is proved by an explicit change of variables and Fubini's theorem.
\item The high-$\beta$ asymptote uses Chernoff tail bounds instead of Slutsky's theorem, as described
after its proof.
\item The logit bracket works in the variable $q = e^t$, in which the pair $c \pm t$ becomes
$x_1 = Mq/(1+Mq)$ and $x_2 = M/(M+q)$ in the beta variable, with $M = e^c$.  The sign of the derivative
of $F(x_1) + F(x_2)$ is then that of $(\alpha-\beta)\log q + (\alpha+\beta)[\log(M+q) - \log(1+Mq)]$.
\item Theorem~\ref{thm:wh_edge} checks the bounds $P(1.1) > 0 > P(1.2)$, $R(2) < 4$, and
$R(5/2)^2 > 32$ in exact rational arithmetic.
\item Lemma~\ref{lem:whmargin} is proved in the weaker form $E > \Delta/(72\alpha^2)$, which is all
the implication from Conjecture~\ref{conj:largeR} to Conjecture~\ref{conj:wh} needs.  Instead of the
power series it uses that $3\log(1 - t/9) + t/3 + t^2/54 + t^3/648$ is nondecreasing on $[0,1]$, since
its derivative is $t^2(1-t)/(1944(1 - t/9))$.
\item Lemma~\ref{lem:Lambda} is formalized with an earlier, longer proof of the concavity of $\Lambda$:
$(N/M)'=H/M^2$ with $H(d)=\sum_{k\ge4}c_k\,d^{2k+1}/(2k+1)!$ and $c_k=16k^3-24k^2+8k+16-4^{k+1}<0$.
The symmetry proof given above was found later.
\item Theorem~\ref{thm:B4large} is formalized with its corollary, the conjecture's sign pattern when both
shape parameters are at least $x_1$ (\texttt{thm\_nuB4\_bound\_large}).  Its first inequality,
\texttt{nuB4\_lt\_nuB1\_of\_large}, needs no hypothesis; the cubic lower bound on $e^{-t}$ comes from
Mathlib's error bound for the five-term Taylor polynomial of the exponential.  The second inequality is
\texttt{thm\_nuB1\_bound}, so the theorem carries the same \citet{berg2006chen} hypothesis as
Theorem~\ref{thm:kerman}.
\item In Lemma~\ref{lem:bracket}, the identity $\mathbb{E}W = \psi_0(\alpha) - \psi_0(\beta)$ is proved by
differentiating $B(s,\beta) = \int_0^1 x^{s-1}(1-x)^{\beta-1}\,dx$ under the integral sign, which gives
$\mathbb{E}\log X = \psi_0(\alpha) - \psi_0(\alpha+\beta)$, and reflecting $x \mapsto 1-x$ for
$\mathbb{E}\log(1-X)$.  That $\psi_0$ is the derivative of $\log\Gamma$ is proved by the convexity
argument that Mathlib uses for $\Gamma'$ at the positive integers.  The integral over $t$ is the
layer-cake formula, applied to the positive and negative parts of $\mu - W$.
\item For the low-$\alpha$ bounds, the polygamma functions are defined by their series,
$\psi_0(x) = -\gamma + \sum_{m\ge0}\left(\frac{1}{m+1} - \frac{1}{m+x}\right)$,
$\psi_1(x) = \sum_{j\ge0}(x+j)^{-2}$, and $\psi_2(x) = -2\sum_{j\ge0}(x+j)^{-3}$, and the proof of
Theorem~\ref{thm:lowbounds} does not differentiate $\log\Gamma$.  Euler's limit formula for $\Gamma$ gives
$\log M_\beta(\alpha)^\alpha = \sum_{m\ge1} f_m(\alpha)$ with
$f_m(\alpha) = \log(m+\beta+\alpha) - \log(m+\alpha) - \log(m+\beta) + \log m$.  Each $f_m$ has a negative third
derivative and a positive fourth derivative, so it lies below its quadratic and above its cubic Taylor polynomial at
$\alpha = 0$;
summing gives the two Taylor bounds on $\ell$ that the proof uses.  Each polygamma bound is proved by
telescoping:  if $R(x+n) \to 0$ and $t(y) < R(y) - R(y+1)$ for all $y \ge x$, then
$\sum_{j\ge0} t(x+j) < R(x)$.  The step inequalities reduce to rational functions whose numerators have
positive coefficients; for $\psi_0$, $\log(1+1/y)$ is bracketed by its series in $1/(2y+1)$.  The bounds
are proved where they are used:  for $\psi_0$ at $w = \beta+2 > 2$, and for $\psi_1$ and $\psi_2$ at
arguments $x \ge 1$.
\end{itemize}

\emph{What is assumed or left open.}
\begin{itemize}
\item Theorem~\ref{thm:kerman} takes as an explicit hypothesis the two facts from
\citet{berg2006chen} used in its proof:  that $\nu_\Gamma(\alpha) - \alpha$ is nonincreasing, and that
$\nu_\Gamma(\alpha) > \alpha - \tfrac{1}{3}$.  These published results are not formalized.  Because they enter
as a hypothesis rather than an axiom, the theorem's axiom list stays standard.
\item Conjecture~\ref{thm:B4bound} is stated (\texttt{thm\_nuB4\_bound}) and derived from a single
unproved inequality, the strong form of the triangle inequality of Lemma~\ref{lem:triangle}, with $m$
written \texttt{mr}:
{\small
\begin{alltt}
theorem conj_median_ratio_triangle :
    \(\forall\) x y z : \(\mathbb{R}\),
      0 < x \(\to\) x < y \(\to\) y < z \(\to\)
        mr x z < mr x y * mr y z := by
  sorry
\end{alltt}}
This is the only \texttt{sorry} in the project.  It is the one result whose axiom list includes
\texttt{sorryAx}, which makes the dependence explicit.
\item Of Lemma~\ref{lem:logconc}, the two Taylor bounds that follow from it are formalized, rather than the signs
of $\ell''$ and $\ell'''$.  The series for $\psi_0$ is proved to be the derivative of $\log\Gamma$;
that $\psi_1$ and $\psi_2$ are its derivatives is not formalized.
\item Theorem~\ref{thm:largeexp}, and with it Theorems~\ref{thm:whray} and~\ref{thm:cfamily}, is not
formalized:  it rests on Laplace's method, whose formalization would be a substantial project.
\item Theorem~\ref{thm:B4over} is not formalized.  Its proof uses Theorems~\ref{thm:lower_bound},
\ref{thm:gamma}, and~\ref{thm:lowasym}, which are, and Theorem~2 of \citet{lyon2026gamma}, which is formalized in
that paper's Lean project.
\item Conjectures~\ref{conj:wh} and~\ref{conj:largeR} are stated but are neither assumed nor left
as a \texttt{sorry}; their Lean names are \texttt{NuB2Conjecture} and
\texttt{UniformRemainderConjecture}.  The implication
from Conjecture~\ref{conj:largeR} to Conjecture~\ref{conj:wh} is proved.  The crossing argument after
Conjecture~\ref{conj:wh}, which uses continuity of $\nu_\textrm{B}$ in $\alpha$, and the numerical
observations are not formalized.
\end{itemize}

\emph{Reproducing the check.}  With Lean's package manager installed, run
\texttt{lake exe cache get} to download the prebuilt Mathlib, then \texttt{lake build} in the project
directory.  The build ends by printing the axiom list of each main result.

\end{appendix}

\begin{acks}[Acknowledgments]
The author thanks his many Google colleagues, but especially Matt Harvey, who helped him
see how AI tools could help with the proofs.  The author also acknowledges ChatGPT (OpenAI),
Gemini (Google), and Claude (Anthropic) for help with proofs, code, and drafting; their use is
described in the Statement on AI use and at the start of the appendix.
\end{acks}

\begin{supplement}
\stitle{Lean~4 formalization}
\sdescription{The Lean~4 project, with its 29 source files, \texttt{lakefile.toml}, toolchain pin, and \texttt{README.md}, which machine-checks the theorems and
lemmas of Appendix~\ref{app:main} with the exceptions listed in Section~\ref{app:lean}.  Build it with
\texttt{lake exe cache get} and \texttt{lake build}.}
\end{supplement}
\begin{supplement}
\stitle{Python code}
\sdescription{A Python notebook, with its two modules, that regenerates every figure and checks the
paper's numerical claims:  the error ranges of $\nu_{\textrm{B}3}$, $\nu_{\Gamma 3}$, the low-$\alpha$
bounds $\nu_{\textrm{B}6}$ to $\nu_{\textrm{B}9}$, $\nu_{1/30}$, and $\nu_{\textrm{B}10}$, as well as the values of
$c$ quoted in Section~\ref{sec:asymexp}.  A third module, \texttt{checks\_mp.py}, re-checks the error ranges
of $\nu_{\textrm{B}8}$, $\nu_{\textrm{B}9}$, $\nu_{1/30}$, and $\nu_{\textrm{B}10}$ with 30-digit medians
computed with mpmath, also at random off-grid points, follows the error of $\nu_{\textrm{B}10}$ below
$\beta = 0.001$, and
reproduces the numerical evidence for Conjectures~\ref{conj:wh},
\ref{conj:largeR}, and~\ref{thm:B4bound} and for the triangle inequality of Section~\ref{app:nuB4}.  A
fourth, \texttt{exact\_medians.py}, reproduces the survey of exact beta medians in
Section~\ref{app:ratio}.}
\end{supplement}

\section*{Statement on AI use}

The author wrote the first version of this paper in 2025 without AI tools.  It contained most of the results,
with computations and figures done in Matlab, and about half of the proofs, none of them formally verified.
Later in its development, in 2026, several AI systems were used to check, simplify, and improve proofs, to
search for more workable proof methods, and to prove further bounds:  ChatGPT (OpenAI), which was useful in
exposing gaps in candidate proofs, and Gemini (Google) and Claude (Anthropic), which found or simplified
several of the proofs.  Gemini was used in Google DeepMind's AI co-mathematician \citep{zheng2026ai}, with
Gemini~3 Pro and Flash models, and Claude in the Claude desktop app.  NotebookLM (Google), with Gemini~3
Flash, and Claude helped revise the draft.  Some of their contributions are noted at the start of the
appendix, where the proofs of the article's claims are collected.  Claude drafted much of the appendix's text
after the author decided to move the proofs out of the main body of the text.  The author reviewed and edited
it.  Claude also drafted parts of the main text, notably Sections~\ref{sec:odds} and~\ref{sec:whole}, which
the author revised.  The edge-exact construction of Sections~\ref{sec:odds} and~\ref{sec:whole}, which gives
$\nu_{\textrm{B}8}$, $\nu_{\textrm{B}9}$, and $\nu_{\textrm{B}10}$, and the edge-exact $c_\textrm{E}$ of
Section~\ref{sec:asymexp}, came from an exploration by Claude, prompted by the author's question whether the
factor $F$ could be made exact where $\nu_{\textrm{B}4}$ is.  Claude helped port the author's Matlab scripts
to Python, and Gemini in Antigravity reproduced the work from scratch in a Python Colab notebook.  Claude
wrote the Lean~4 formalizations and proofs, and the author checked the proof statements and ran the
verification.  None of these AI systems could provide complete proofs of the three remaining conjectures
(Conjectures~\ref{conj:wh}, \ref{conj:largeR}, and~\ref{thm:B4bound}), though numerical checks support all
three.  Google's Deep Review tool \citep{google2026deepreview} reviewed the manuscript.  The author checked
the drafts, revised the text throughout, and is responsible for its content.

\providecommand{\noopsort}[1]{}

\end{document}